\documentclass[11pt,letterpaper]{article}
\usepackage{geo}

\newcommand{\grg}{\mathsf{Geo}}
\newcommand{\rgg}{\mathsf{Geo}}
\newcommand{\ER}{\mathsf{G}}
\newcommand{\gnp}{{\mathsf{G}(n,p)}}
\newcommand{\dtv}[2]{\mathrm{d}_{\mathrm{TV}}\left(#1,#2\right)}
\newcommand{\GG}{\mathsf{gg}}
\newcommand{\Unif}{\rho}
\newcommand{\dkl}{\mathrm{D}}

\newcommand{\KL}[1]{\dkl(#1\|\Unif)}

\newcommand{\randnbr}{N_{\bG}}
\newcommand{\Leaves}{N_R}
\newcommand{\Balls}{K}
\newcommand{\BallsInt}{K}

\newcommand{\Rest}{R}

\newcommand{\Inst}{\calI}

\newcommand{\Dev}{\mathrm{Dev}}
\newcommand{\Spr}{\mathrm{Spr}}
\newcommand{\Res}{\mathrm{Anch}}
\newcommand{\Tier}{\mathrm{Tier}}
\newcommand{\Par}{\mathrm{Par}}
\newcommand{\Ch}{\mathrm{Ch}}
\newcommand{\Un}{\mathrm{Un}}

\newcommand{\Ber}{\mathsf{Ber}}
\newcommand{\Binom}{\mathsf{Binom}}

\newcommand{\iprod}[1]{\left\langle #1 \right\rangle}
\newcommand{\Iprod}[1]{\langle #1 \rangle}

\newcommand{\VecDist}[1]{\Unif^{#1}}

\newcommand{\subgnorm}{\psi_2}

\newcommand{\scap}{\mathrm{cap}}

\newcommand{\anticap}{\ol{\mathrm{cap}}}

\usepackage{scalerel}

\newcommand{\CPsiBound}{C_{\scaleto{\ref{lem:beta-pdf}}{5pt} }}
\newcommand{\Cbetaac}{C_{\scaleto{\ref{lem:beta-concentration}}{5pt} }}

\newcommand{\CMartingale}{C_{\scaleto{\ref{lem:martingale-azuma}}{5pt} }}
\newcommand{\CSubgauss}{C_{\scaleto{\ref{lem:subgaussian-equiv}}{5pt} }}

\newcommand{\Cconcalt}{C_{\scaleto{\ref{cor:two-sets-subexp}}{5pt} }}

\newcommand{\Ccaa}{C_{\scaleto{\ref{cor:caps-and-anti}}{5pt} }}

\renewcommand{\ln}{\log}

\begin{document}

\title{Testing thresholds for high-dimensional sparse random geometric graphs}
\author{Siqi Liu\thanks{UC Berkeley. \texttt{sliu18@berkeley.edu}. Supported in part by the Berkeley Haas Blockchain Initiative and a donation from the Ethereum Foundation.} \and Sidhanth Mohanty\thanks{UC Berkeley. \texttt{sidhanthm@cs.berkeley.edu}. Supported by a Google PhD Fellowship.} \and Tselil Schramm\thanks{Stanford University.  \texttt{tselil@stanford.edu}.} \and Elizabeth Yang\thanks{UC Berkeley. \texttt{elizabeth\_yang@berkeley.edu}.  Supported by the NSF GRFP under Grant No. DGE 1752814.}}
\date{\today}
\maketitle

\begin{abstract}
The random geometric graph model $\grg_d(n,p)$ is a distribution over graphs in which the edges capture a latent geometry.
To sample $\bG \sim \grg_d(n,p)$, we identify each of our $n$ vertices with an independently and uniformly sampled vector from the $d$-dimensional unit sphere $\bbS^{d-1}$, and we connect pairs of vertices whose vectors are ``sufficiently close,'' such that the marginal probability of an edge is $p$.
Because of the underlying geometry, this model is natural for applications in data science and beyond.

We investigate the problem of testing for this latent geometry, or in other words, distinguishing an \erdos-\renyi~graph $\ER(n, p)$ from a random geometric graph $\grg_d(n, p)$.
It is not too difficult to show that if $d\to \infty$ while $n$ is held fixed, the two distributions become indistinguishable; we wish to understand how fast $d$ must grow as a function of $n$ for indistinguishability to occur.

When $p = \frac{\alpha}{n}$ for constant $\alpha$, we prove that if $d \ge \polylog(n)$, the total variation distance between the two distributions is close to $0$; this improves upon the best previous bound of Brennan, Bresler, and Nagaraj (2020), which required $d \gg n^{3/2}$, and further our result is nearly tight, resolving a conjecture of Bubeck, Ding, Eldan, \& R\'{a}cz (2016) up to logarithmic factors.
    We also obtain improved upper bounds on the statistical indistinguishability thresholds in $d$ for the full range of $p$ satisfying $\frac{1}{n} \leq p \leq \frac{1}{2}$, improving upon the previous bounds by polynomial factors.

Our analysis uses the Belief Propagation algorithm to characterize the distributions of (subsets of) the random vectors {\em conditioned on producing a particular graph}. 
In this sense, our analysis is connected to the ``cavity method'' from statistical physics.
To analyze this process, we rely on novel sharp estimates for the area of the intersection of a random sphere cap with an arbitrary subset of $\bbS^{d-1}$, which we prove using optimal transport maps and entropy-transport inequalities on the unit sphere.
We believe these techniques may be of independent interest.
\end{abstract}

\setcounter{tocdepth}{2}
\setcounter{page}{-10}
\thispagestyle{empty}
\newpage
\tableofcontents
\thispagestyle{empty}
\thispagestyle{empty}
\newpage
\setcounter{page}{1}

\section{Introduction}

The study of random graphs has been incredibly influential, not only in modeling applications, but also in the development of algorithms and in the study of mathematics.
For example, consider the simple \erdos-\renyi random graph model $\ER(n,p)$, in which $n$ nodes are each connected independently with probability $p$.
This model was introduced by \erdos and \renyi in 1959 \cite{ER59}, and since then its study has blossomed.
To list only a few of the many fruits of this line of research:
$\ER(n,p)$ graphs have been used to demonstrate the existence of many combinatorial objects via the probabilistic method (e.g. \cite{alon-spencer});
they have been used in algorithm design both as a benchmark and as a starting point for analysis in worst-case or semirandom settings (e.g. \cite{feige-kilian,BCCFV10});
they have been useful for proving conditional and unconditional lower bounds in complexity theory (e.g. \cite{jerrum92,feige02,rossman08,grigoriev01,schoenebeck08});
and finally $\ER(n,p)$ graphs are a common model in statistical physics \cite{nishimori81,panchenko13} and network science \cite{HLBL83,CK16}.
Though deep questions remain, decades of intensive study of the $\ER(n,p)$ distribution have been rewarded with a rich understanding and many unforeseen insights.

Our primary object of study is the {\em random geometric graph} model on the sphere, $\grg_d(n,p)$. 
In this model, a graph is sampled by choosing $n$ vectors $\bv_1,\ldots,\bv_n$ uniformly at random from the $d$-dimensional sphere $\bbS^{d-1}$, identifying $\bv_i$ with vertex $i$, then connecting each pair $i,j$ whose vectors $\bv_i,\bv_j$ have inner product exceeds a threshold $\tau(p)$, chosen so that $\Pr[\iprod{\bv_i,\bv_j} > \tau(p)] = p$.
A compelling aspect of this model is that the graph structure is derived from a geometric representation of the vertices;
this makes it suitable for modeling applications in, say, data science, where we think of network nodes as representable by a feature vector in some $d$-dimensional space, with neighboring nodes sharing similar features.

Though random geometric graphs on the sphere (and also on other domains, such as $[0,1]^d$) have been studied extensively (see the monograph \cite{penrose03}), the focus has been on the low-dimensional setting, where we think of $d$ is fixed as $n \to \infty$.
However, the high-dimensional setting, in which $d \to \infty$ as a function of $n$, is still poorly understood.
In low dimensions, one can understand $\grg_d(n,p)$ as a fine random discretization of $\bbS^{d-1}$, but we no longer expect $\grg_d(n,p)$ to behave in this way in the high-dimensional setting. 

This opens an intriguing possibility: what are the properties of $\grg_d(n,p)$ in the high-dimensional setting?
And, what are the undiscovered applications of these random objects in algorithm design and in mathematics?
One reason to study this question is that the high-dimensional setting is a more faithful model for graphs arising from modern high-dimensional data (e.g. high-dimensional feature vectors).
Thus, high-dimensional geometric random graphs could constitute a more useful benchmark for algorithmic methods, or a more useful starting point for designing algorithms for the semirandom setting.
Another reason to study this distribution is that it may yield graphs with properties that we have not yet realized are possible; this has historical precedent, for example, \erdos-\renyi graphs witnessing the existence of certain Ramsey graphs \cite{alon-spencer}.
\medskip

Devroye, Gy\"{o}rgy, Lugosi, and Udina \cite{DGLU11} made the first exploration of $\rgg_d(n,p)$ in the high-dimensional regime.
They study the classic property of the chromatic number in random geometric graphs, and also raise the following fundamental question:
 for which $d= d(n,p)$ is it possible to test for the presence of an underlying geometry?
When can one distinguish $\rgg_d(n,p)$ from $\ER(n,p)$?

For intuition, one can see that if $d \to \infty$ for $n$ fixed, the $\grg_d(n,p)$ distribution approaches $\ER(n,p)$: the vectors $\bv_1,\ldots,\bv_n$ are effectively mutually orthogonal, $\tau(p)$ is very small, and conditioning on the presence of the edge $(i,j)$ (or equivalently on $\Iprod{\bv_i,\bv_j} > \tau(p)$) has little impact on the probability that the edges $(i,k)$ and/or $(j,k)$ are present. 
How fast must $d$ grow as a function of $n$ to realize this limiting behavior?
In \cite{DGLU11} the authors observe that if $d \gg \exp(n^2)$, a central limit theorem implies that $\rgg_d(n,p)$ and $\ER(n,p)$ are indistinguishable; they also note that this bound is likely far from tight.

Determining the asymptotic threshold in $d$ when geometric graphs become indistinguishable from \erdos-\renyi graphs is a most basic question that we must address if we wish to make a serious study of high-dimensional random geometric graphs.
Bubeck, Ding, Eldan, and R\'{a}cz \cite{BDER16} were the first to tackle this question, showing that in dense graphs (when $p = \Theta(1)$), the threshold occurs at $d \asymp n^3$; when $d \gg n^3$ the total variation distance goes to $0$, and when $d \ll n^3$ the ``signed triangle count'' provides a good test statistic.
But in the arguably more interesting sparse case $p = \Theta \left(\frac{1}{n} \right)$, \cite{BDER16} are only able to establish that signed triangle counts distinguish when $d = O(\log^3 n)$.
They conjecture that when $d = \Omega(\log^3 n)$, the distributions are indistinguishable.

The current best bound, due to Brennan, Bresler, and Nagaraj \cite{BBN20}, asserts that in the regime $p = \Theta \left(\frac{1}{n}\right)$, $\dtv{\grg_d(n,p)}{\ER(n,p)} \to 0$ so long as $d \gg n^{3/2}$.
In essence, their bound relies on the fact that independent random vectors $\bv_i,\bv_j \sim \bbS^{d-1}$ have $|\Iprod{\bv_i,\bv_j}| = \wt{O}(\frac{1}{\sqrt{d}})$ with high probability; when $d \gg n^{3/2}$, these inner products are small enough relative to $n$ that $\bv_i$ has negligible projection (of order $\approx \sqrt{\sum_{j \neq i} \Iprod{\bv_i,\bv_j}^2} = O(\sqrt{n/d})$) into $\mathrm{span}\{\bv_j\}_{j \neq i}$, which is enough to guarantee approximate independence of edges.
This argument is carried out with technical sophistication in \cite{BBN20}, but clearly, this technique cannot be extended to $d < n$, much less to $d = \polylog n$, as the vectors in $\{\bv_j\}_{j \neq i}$ then span all of $\R^d$.

In this work, we come close to closing this gap in the sparse regime, confirming the conjecture of \cite{BDER16} up to polylogarithmic factors: we show that if $p = \Theta(\frac{1}{n})$, the total variation distance goes to zero when $d \gg \log^{36} n$.
We also give a separate improved bound on the total variation distance for the entire parameter regime $p \ll 1$, showing that the total variation distance goes to $0$ when $d \gg n^3 p^2 $, improving upon previous bounds by polynomial factors. 
Our proof relies on a number of novel technical contributions.
Crucially, we must understand, for ``typical'' $\bG \sim \grg_d(n,p)$, the distribution of (subsets of the) vectors $\bv_1,\ldots,\bv_n$ sampled from $\bbS^{d-1}$ {\em conditioned on giving a vector embedding for $\bG$}. 
In order to understand this complicated conditional distribution, we use the Belief Propagation (BP) algorithm, where our variables are the vertices $[n]$ and their ``labels'' are vectors in $\bbS^{d-1}$. 
To analyze BP, we rely on a (to our knowledge) novel concentration inequality, which we prove using optimal transport, for the area of the intersection of a random spherical cap with any subset $L \subseteq \bbS^{d-1}$.
We also demonstrate a coupling of $\bG\sim\grg_d(n,p)$ and $\bG_+ \sim \ER(n,p + o(p))$ which produces $\bG \subseteq \bG_+$ (meaning $\bG$ is a subgraph of $\bG_+$) with high probability.
We feel that these techniques may find other applications, and be of independent interest.

\subsection{Our results}

Our main result is an indistinguishability result for sparse Random Geometric Graphs and \erdos-\renyi graphs when the dimension $d$ exceeds $\poly\log n$.

\begin{theorem} \torestate{\label{thm:sparse}
For any fixed constant $\alpha \ge 1$, if $d = \Omega(\log^{36} n)$, then 
\[
\lim_{n \to \infty}\dtv{\grg_d\left(n,\tfrac{\alpha}{n}\right)}{\ER\left(n,\tfrac{\alpha}{n}\right)} = 0.
\]}
\end{theorem}

Our result settles the conjecture of \cite{BDER16} up to logarithmic factors, an exponential improvement over the previous bound of \cite{BBN20}, which required $d \gg n^{3/2}$.
We remark that we have not made an effort to optimize the logarithmic factors; it is possible that our current proofs in combination with chaining-style arguments will yield $\log^3 n$, matching their conjecture.
We also obtain an improved result for general $p = \Omega \left(\frac{1}{n} \right)$:

\begin{theorem} \torestate{\label{thm:all-p}
For any fixed constant $\alpha > 0$, if $\frac{\alpha}{n} < p < \frac{1}{2}$ and $d = \wt{\Omega}(p^2n^3)$,
\[
\lim_{n \to \infty}\dtv{\grg_d\left(n,p\right)}{\ER\left(n,p\right)} = 0.
\]
}
\end{theorem}
This improves  by polynomial factors (in $p$ and $n$) on the previous bound of \cite{BBN20}, which required $d \gg \min\{pn^3 \log \frac{1}{p}, p^2 n^{7/2} \polylog n\}$ and $d \gg n \poly\log n$.
However, this result is not tight (at least for small $p$) since in particular it does not recover \pref{thm:sparse}.
Given that we have come close to establishing the conjecture of \cite{BDER16} in the sparse case, it is tempting to interpolate between the upper and lower bounds of \cite{BDER16} in the $p = \Theta(1)$ regime and their conjecture for the $p = \Theta(\frac{1}{n})$ regime and speculate that for all $p \le \frac{1}{2}$, the testing threshold occurs at $d \asymp (n H(p))^3 = O(n^3 p^3 \ln^3 \frac{1}{p})$, for $H(p)$ the binary entropy function. 
In \pref{app:signed-triangles}, we show that the ``signed triangle count'' test statistic analyzed by \cite{BDER16} in the $p = \Theta(1)$ regime can in fact distinguish whenever $d \ll (n H(p))^3$ for all $p = \Omega(\frac{1}{n})$, establishing that the testing threshold occurs at some $d = \Omega((nH(p)^3)$.
If $d \asymp (nH(p))^3 $ is indeed the threshold, our \pref{thm:all-p} is tight up to a factor of $\wt{O}(p)$.

\paragraph{Relaxed stochastic dominance by \erdos-\renyi graphs.} En route to proving \pref{thm:sparse} and \pref{thm:all-p}, we establish a result which may be of independent interest.
We show that whenever $d = \wt\Omega(p^2n^2)$, a random geometric graph $\bG \sim \grg_d(n,p)$ can be coupled with an \erdos-\renyi graph $\bG_+ \sim \ER(n,p + o(p))$ (and with $\bG_- \sim \ER(n,p-o(p))$) in such a way that the \erdos-\renyi~graph ``stochastically dominates''\footnote{Strictly speaking we only have stochastic dominance conditioned on the success of the coupling; that is why we say that it is ``relaxed'' stochastic dominance.} the geometric graph with high probability over the coupling, in the sense that every edge in $\bG$ is also contained in $\bG_+$ (respectively, every edge in $\bG_-$ is also contained in $\bG$).
\begin{proposition}
\torestate{\label{prop:dom}
For any constant $\alpha > 0$ there exist constants $C_1,C_2 > 0$ such that if $\frac{\alpha}{n} \le p \le \frac{1}{2}$ and $d \ge C_1 \cdot ( n^2p^2 + \ln^4 n)\ln^4 n$, for any  $\eps \ge C_2 \sqrt{\frac{1}{d}(np + \ln n)\ln^4 n}$, one can simultaneously sample $\bG_- \sim \ER(n,(1-\eps)p)$, $\bG\sim \grg_d(n,p)$, and $\bG_+\sim \ER(n,(1+\eps)p)$ in a correlated manner so that with probability at least $1-n^{-\Omega(\ln n)}$, $\bG_- \subseteq \bG \subseteq \bG_+$.}
\end{proposition}
\noindent The notation $\bG \subseteq \bG_+$ means that $\bG$ is a subgraph of $\bG_+$.
We find it intriguing that this coupling succeeds whenever $d = \wt{\Omega}(p^2n^2)$, which is well below the speculative ``interpolated'' threshold $d = \wt{\Omega}(p^3 n^3)$ for large $p$.

\subsection{Prior and related work}

\paragraph{Prior Works.}
Random geometric graphs in fixed dimension are a well-studied model, with connections to Poisson processes and continuum percolation. 
The tools used to study fixed-dimensional random graphs are of a very different flavor from ours; for example, in low dimensions it is often useful to compare $\grg_d(n,p)$ to a random process on an appropriate low-dimensional infinite lattice which discretizes the space.
We refer the reader to the survey \cite{walters11} and the monograph \cite{penrose03}.

The study of random geometric graphs in high dimension was initiated by \cite{DGLU11}.
The main result of \cite{DGLU11} is a bound on the clique number of $\grg_d(n,p)$; for example in the dense case $p = \Theta(1)$, they show that the clique numbers of $\grg_d(n,p)$ and $\ER(n,p)$ become indistinguishable when $d = \Omega(\log^3 n)$.
The authors of \cite{DGLU11} also note that if $d \to \infty$ fast enough as a function of $n$ (they require $d \gg \exp(n^2)$), $\grg_d(n,p)$ and $\ER(n,p)$ are indistinguishable.
This naturally raises the question: for which $d$ is it possible to test for the underlying geometry?

The authors of \cite{BDER16} are the first to directly study the testing phase transition in $d$.
They show that at any density, if $d \gg n^3$, then the total variation distance between $\rgg_d(n,p)$ and $\ER(n,p)$ goes to zero with $n$, and conversely, if $d \ll n^3$, then in the dense regime $p = \Theta(1)$, the signed triangle count statistic furnishes a hypothesis test between the distributions.
However in the sparse regime $p = \Theta\left(\frac{1}{n}\right)$, their results are not as conclusive; they are able to show that triangle counts furnish a test if and only if $d \ll \log^3 n$, and they conjecture that if $d \gg \log^3 n$ the total variation distance goes to zero.
Briefly, their bound on the TV distance is via a reduction to the indistinguishability of Wishart matrices and matrices from the Gaussian Orthogonal Ensemble (GOE); 
the idea is that one can obtain a sample from $\ER(n,p)$ by thresholding the off-diagonal entries of a GOE matrix $\bB$ with independent entries sampled from $\cN(0,\frac{1}{d})$ at threshold $\approx \tau(p)$, and similarly obtain a sample from $\grg_d(n,p)$ by thresholding the off-diagonal entries of a Wishart matrix $\bA\bA^\top$ at $\approx \tau(p)$, where $\bA \in \R^{n \times d}$ with rows sampled independently from $\cN(0,\frac{1}{d}\Id_d)$.
Hence, if $\bB$ and $\bA\bA^\top$ are indistinguishable from their off-diagonal entries, one can conclude that $\ER(n,p)$ and $\rgg_d(n,p)$ are indistinguishable as well.
It makes sense, then, that the result may not be tight in the sparse regime, as thresholding at a higher value $\tau(p)$ (corresponding to a sparser graph) intuitively reveals less information about the original GOE or Wishart matrix.
The proof in \cite{BDER16} directly compares the Wishart and GOE densities to obtain the TV bound.
Independently, \cite{JL15} obtain the same bounds on the TV distance of Wishart and GOE matrices.

Following the work of \cite{BDER16}, \cite{BBN20} study the question of detecting underlying geometry in greater generality. 
The authors show that for any $\frac{\log n}{n^2}<p \le \frac{1}{2}$, if $d \gg \min\left\{pn^3\log\frac{1}{p},p^2 n^{7/2} \poly\log n\right\}$ and $d \gg n\log^4 n$, then $\dtv{\rgg_d(n,p)}{\ER(n,p)} \to_n 0$.
In particular, they match the bound appearing in \cite{BDER16} for dense graphs, and in sparse graphs with $p = \Theta\left(\frac{1}{n}\right)$ they improve the bound to $d \gg n^{3/2} \poly\log n$.
At a high level, their proof applies information-theoretic inequalities to reduce the question to bounding the Chi-square divergence of the marginal of a single edge in $\rgg_d(n,p)$ vs. $\ER(n,p)$; they then bound this Chi-squared divergence by showing that if one conditions $\bv_1,\ldots,\bv_n \sim (\bbS^{d-1})^{\otimes n}$ on the event that they produce a ``typical'' random geometric graph $G$ excluding edge $(n-1,n)$, then the vectors $\bv_n,\bv_{n-1}$ remain sufficiently independent that $\Pr[\Iprod{\bv_{n-1},\bv_n} > \tau(p)]\approx p$. 
They achieve this by showing that $\bv_{n-1}$ and $\bv_n$ each have a small projection onto $\mathrm{span}\{\bv_1,\ldots,\bv_{n-2}\}$, even after conditioning on their adjacency into $[n-2]$.
Their argument is essentially {\em oblivious to the particular choice of $G$}, and merely uses properties that hold with high probability over independently sampled vectors in $\bbS^{d-1}$.
This argument relies on $\bv_1,\ldots,\bv_{n-2}$ not spanning the entirety of $\R^d$, and the authors of \cite{BBN20} explicitly state that improving their bound to any $d < n$ requires a different approach.

\paragraph{Techniques.}
We describe our techniques in detail in \pref{sec:tech-overview}, but here we discuss some connections in the literature.
To improve upon the bounds of \cite{BBN20}, we draw upon tools from several areas.
At the heart of the proofs of both \pref{thm:sparse} and \pref{thm:all-p} are new sharp concentration of measure results for intersections of random spherical caps with arbitrary subsets of $\bbS^{d-1}$; to prove these bounds, we make use of optimal transport inequalities in the Wasserstein metric (see e.g. \cite{Vil08}). 
To our knowledge this is the first application of optimal transport in this context, and may be of independent interest. 
These tail bounds (in combination with Pinsker's inequality and some simple arguments) are enough to yield \pref{thm:all-p}, proving that the total variation distance goes to zero if $d \gg n^3 p^2$.

Then, to break the $d \asymp n$ barrier for sparse graphs and prove \pref{thm:sparse}, we must answer the following question: if $\bv_1,\ldots,\bv_n$ are sampled uniformly conditioned on producing a vector embedding of a fixed graph $G$, how strong are the correlations in the marginal distributions on $\{\bv_i \}_{i \in S}$ for small subsets $S \subset [n]$?
This is similar to establishing decay of correlations, as in the analysis of Gibbs sampling (c.f. \cite{DSW04,weitz06}).
To achieve this, we carry out a rigorous analysis in the style of the ``{\em cavity method}'' from statistical physics (c.f. \cite{MP03}): we compute the marginal over the depth-$\frac{\log n}{\log\log n}$-neighborhood of vertices in $S$ via the Belief Propagation (BP) algorithm, under arbitrary boundary conditions on the rest of the graph.
The cavity method has been previously used to compute solution geometry phase transitions for a number of prominent discrete spin systems, such as coloring, Ising and Potts models, and more (c.f. \cite{DMS14,CKPZ18,panchenko09,talagrand03}); 
cavity-style arguments have also been employed in a similar way to establish decay of correlation properties in the analysis of Glauber dynamics, as in \cite{GKS15}.
In our instantiation of Belief Propagation, the variables are the vertices in the graph, and their ``labels'' or ``assignments'' are vectors in $\bbS^{d-1}$, the constraints are that vectors corresponding to edges have inner product at least $\tau$, and hence the ``messages'' passed from variable to variable are convolutions of marginal distributions with spherical caps; the concentration of measure for intersections of sets with random spherical caps is again useful in the analysis of this BP.

\paragraph{Applications of high-dimensional geometric graphs in Theoretical CS.}
A line of work including \cite{FS02,FLS04} has utilized a distribution similar to $\grg_d(n,p)$, to obtain integrality gaps for semidefinite relaxations of max-cut or graph coloring.
In their setting, a graph $\bG$ is sampled by first placing $n$ vertices in a ``regular'' configuration on a $d$-dimensional sphere, after which $n' = \exp(\Theta(d)))$ vertices are randomly subsampled independently with some probability $q$,\footnote{For small enough $q$ these are equivalent to $n'$ independent samples from the uniform measure over $\bbS^{d-1}$.} and the sample is a graph induced on these $n'$ vertices, in which vertices are connected if their inner product lies in some range of distances $(\tau_1,\tau_2)$ (as opposed to the $\grg_d(n,p)$ case in which the connectivity criteria is inner product at least $\tau(p)$). 
The resulting $\bG$ comes with a natural embedding into $\bbS^{d-1}$, which is utilized in constructing the semidefinite programming certificate.
Similar constructions are also used in \cite{KTW14} to give the optimal approximation ratio of constraint satisfaction problems assuming the unique games conjecture.
Though these graphs are not sampled from $\grg_d(n,p)$, they are sampled from a distribution over a high-dimensional sphere which is qualitatively similar. 
We feel that this points to the promise of high-dimensional random geometric graphs for applications in theoretical computer science.

\paragraph{Phase transitions between Wishart and GOE matrices.}
The works mentioned above are those most closely related to our results.
We mention as well additional works concerning the phase transition between Wishart and GOE matrices: the work of \cite{RR19} studies the phase transition between Wishart and GOE at a higher resolution in the dense regime, deriving the expression for the total variation distance as a function of $c = d/n^3$ when $\lim_{n \to \infty} d/n$ is finite. 
The result \cite{BG18} generalizes the $d \gg n^3$ bound to Wishart matrices in vectors drawn from log-concave measures, and in \cite{EM20} the authors study the phase transition for Wishart matrices in vectors drawn from non-isotropic Gaussian measures, drawing conclusions for hypothesis testing dense geometric graphs derived from these ensembles as well. 
Other similar questions that have been studied are hypothesis testing in {\em noisy} Wishart matrices \cite{LR21}, in which the entries are with some probability independently resampled from the Gaussian distribution,  and {\em masked} hypothesis testing between Wishart and GOE matrices, in which only a subset of the matrix's entries are revealed \cite{BBH21}. 
In both of these cases, it has been demonstrated the presence of noise or masking can shift the threshold in interesting ways.

\subsection*{Organization of the paper}
In \pref{sec:tech-overview}, we explain our techniques and the proofs of \pref{thm:sparse} and \pref{thm:all-p} at a high level. 
\pref{sec:prelims} contains preliminaries and definitions. 
In \pref{sec:transport-result} we use optimal transport to derive tail bounds for the measure of the intersection of a random spherical cap with an arbitrary subset (or distribution) on $\bbS^{d-1}$, and in \pref{sec:martingale} we utilize these bounds to get tight concentration for a sequence of random caps and anti-caps.
In \pref{sec:stoch-dom} we prove our coupling \pref{prop:dom}.
Then, in \pref{sec:bp}, we use the Belief Propagation algorithm to analyze the marginal distributions on vectors conditioned on producing a specific graph $G$; we then put these ingredients together in \pref{sec:tv-bound} to prove \pref{thm:all-p} and \pref{thm:sparse}.
Finally, in \pref{app:signed-triangles} we show that the signed triangle count hypothesis test of \cite{BDER16} can be extended to work so long as $d \ll (np\ln\frac{1}{p})^3$ for the full range of $p$, and in \pref{app:prelim} we give some deferred proofs from \pref{sec:prelims}.

\section{Technical overview}\label{sec:tech-overview}

\subsection{Relative entropy tensorization} 
Our goal is to determine $d$ at which the total variation distance $\dtv{\ER(n, p)}{\grg_d(n, p)}$ goes to $0$ as $n \to \infty$. 
Like the authors of \cite{BBN20}, we relate the TV distance between these two distributions to their relative entropy (\pref{def:kl}) $\dkl(\grg_d(n, p) \| \ER(n, p))$ via Pinsker's inequality (\pref{thm:pinsker}), and then apply the tensorization of the relative entropy (\pref{claim:KL-seq}). 
Roughly, the tensorization says that given a decomposition of $\ER(n, p)$ as a product distribution, we can reduce the problem of bounding $\dkl(\grg_d(n, p) \| \ER(n, p))$ to bounding the relative entropy over (potentially simpler) distributions with smaller support.

$\ER(n, p)$ is conveniently a product distribution over edges.
However, unlike $\cite{BBN20}$, we do not use this straightforward decomposition of $\ER(n, p)$ by edge. 
Instead, let $\mu_t$ be the distribution of vertex $t$'s edges to $[t - 1]$.
Similarly, let $\nu_t$ be the marginal distribution of vertex $t$'s edges to $[t - 1]$ over the graph being sampled from $\grg_d(n, p)$. 
Our bound via tensorization now becomes
\[
 \dkl(\grg_d(n, p) \| \ER(n, p)) = \sum_{t = 1}^n \E_{\bG_{t-1} \sim \grg(t - 1, p)}\left[\dkl\left(\nu_t(\cdot \mid \bG_{t-1}) \, \| \, \mu_t \right) \right] \leq n \cdot \E_{\bG_{n-1} \sim \grg(n - 1, p)}\left[\dkl\left(\nu_n(\cdot \mid \bG_{n-1}) \, \| \, \mu_n \right) \right]
\]
where the final inequality follows after applying the chain rule for relative entropy (\pref{claim:n-nbrs}).

\paragraph{The coupling view.} 
The tensorization inequality reduces bounding the TV distance to comparing the probability distribution of the neighborhood of the ``final'' vertex in $\ER(n, p)$ and $\grg(n, p, d)$.
Specifically, we study $\E_{\bG_{n-1} \sim \grg(n - 1, p)}\left[\dkl\left(\nu_n(\cdot \mid \bG_{n-1}) \, \| \, \mu_n \right) \right]$ by considering the following scenario: we already have a graph $\bG_{n-1}$ sampled on $n - 1$ vertices, and we want to incorporate vertex $n$ into our graph. 
By the definition of the Erd\"os-R\'enyi distribution, $\mu_n$ will sample the neighbor set $S \subseteq [n-1]$ with probability $p^{|S|}(1-p)^{n-1-|S|}$.
For a random geometric graph, we can sample a vector $\bv_n \sim \Unif$, and take its dot products to vectors $\bv_1, \ldots, \bv_{n - 1}$ sampled uniformly from $\bbS^{d-1}$ conditioned on producing $\bG_{n-1}$, to determine the neighbors of $n$ in $\bG$ (which we denote by $N_{\bG}(n)$).
Our goal now is to compare $\Pr_{\bG \sim \ER(n, p)}[N_{\bG}(n) = S]$ and $\Pr_{\bG \sim \grg_d(n, p)}[N_{\bG}(n) = S]$ for $S \subseteq [n - 1]$. 

\subsection{A geometric interpretation of neighborhood probability}

For $\bG \sim \grg_d(n, p)$, if vertex $i$ is associated to a (random) vector $\bv_i$, and $(i, j)$ is an edge, we consequently know that $\langle \bv_i, \bv_j \rangle \geq \tau$. 
On the sphere $\bbS^{d - 1}$, the locus of points where $\bv_j$ can be, conditioned on $(i,j)$ being an edge, is a sphere cap centered at $\bv_i$ with a $p$ fraction of the sphere's surface area, which we denote by $\scap(\bv_i)$. 
Similarly, if we know that $i$ and $j$ are \emph{not} adjacent, the locus of points where $\bv_j$ can fall is the complement of a sphere cap, which we call an ``anti-cap,'' with measure $1 - p$. 

Equipped with this geometric picture, we can view the probability that vertex $n$'s neighborhood is exactly equal to $S \subseteq [n - 1]$ as the measure $\Unif(\bL_S)$, where $\bL_S \subseteq \bbS^{d - 1}$ is a random set defined as 
$$
\bL_S := \left(\bigcap_{i \in S} \scap(\bv_i) \right) \cap \left( \bigcap_{i \notin S} \overline{\scap(\bv_i)} \right)
$$
To show that the distance between $\grg_d(n,p)$ and $\ER(n,p)$ is small, we must show that $\rho(\bL_S)$ concentrates around $p^{\abs{S}} (1 - p)^{n - 1 - \abs{S}}$, which is the probability that $n$'s neighborhood is equal to $S$ under the \erdos--\renyi model. 

\paragraph{Optimal transport.}  The backbone of our result is a (to our knowledge) novel application of optimal transport.
In \pref{sec:transport-result}, we prove for a generic distribution $\calL$ supported on $\bbS^{d - 1}$, and $\bz$ sampled uniformly at random over $\bbS^{d - 1}$, that 
$$\Pr_{\bx \sim \calL}[\langle \bx, \bz \rangle \geq \tau] \in (1 \pm \varepsilon) \cdot p \text{ for } \varepsilon \leq \tilde{O} \left(\sqrt{\frac{\ln \norm{\calL}_\infty}{d}} \right)$$ 
with high probability over $\bz$. 
In other words, how tightly the random variable $X_{\calL}(\bz) = \Pr_{\bx \sim \calL}[\langle \bx, \bz \rangle \geq \tau]$ concentrates is directly related to the maximum value of its relative density, $\norm{\calL}_\infty$. 

To give some intuition for this result, first consider the case when $\calL = \Unif$, the uniform distribution over $\bbS^{d - 1}$. 
Then, the variable $X_\rho(\bz) = \Pr_{x \sim \rho}[\iprod{\bx,\bz} \ge \tau] = p$ deterministically.
Now, when $\calL \neq \Unif$, we can work with a transport map $\calD$ between $\Unif$ and $\calL$, and we can couple $\by \sim \calL$ and $\bx \sim \Unif$ according to $\calD$, so that 
\[
X_{\calL}(z) = \Pr_{\by \sim \calL}[\iprod{\by,z} \ge \tau] = \Pr_{\substack{(\bx,\by) \sim_{\calD} (\Unif,\calL)\\ \be = \by - \bx}}[\iprod{\bx,z} \ge \tau - \iprod{\be,z}].
\]
The smaller $\norm{\calL}_\infty$ is, the smaller the average of the transport distance $\|\be\| = \|\bx - \by\|$; further when $\bz \sim \rho$ the quantity $\iprod{\be,\bz}$ concentrates tightly around $\frac{1}{\sqrt{d}}\|\be\|$.
In this way, we translate the concentration of transport distance into tail bounds on $|X_\rho(\bz) - X_{\calL}(\bz)|$. 

To analyze $\Unif(\bL)$ for $\bL$ the intersection of caps and anti-caps defined above, we will apply the above in sequence inside a martingale concentration argument, building up $\bL$ one cap at a time (\pref{lem:martingale-intersect}, \pref{cor:caps-and-anti}).
Using this approach, our transport result alone is enough to conclude \pref{thm:all-p}. 
(The proof is assembled in \pref{sec:gen-p}.) 

\paragraph{The need to resample vectors.} 
In the general $p$ setting, we can think of our analysis of $\nu_n(\cdot  | \bG_{n - 1})$ as considering a fixed vector embedding $\bv_1,\ldots,\bv_{n-1}$ of $\bG_{n - 1}$, and then analyzing the probability that $n$ connects to some $S \subseteq [n - 1]$. 
When $p = \frac{\alpha}{n}$, this does not yield tight results; moreover, one can show that this is not due to loose tail bounds on $\rho(\bL)$, as our concentration results have matching anti-concentration results. 

Hence, in order to prove \pref{thm:sparse}, we must additionally consider the concentration of $\rho(\bL_S)$ {\em on average over vector embeddings of $\bG_{n-1}$} as well. 
We will first sample $\bG_{n-1}$, and then for each set $S$, we bound the deviation in the random variable $\Unif(\bL_S) = \Unif(\bigcap_{i \in S} \scap(\bu_i) \cap \bigcap_{j \not \in S} \anticap(\bu_j))$ conditioned on $\bu_1,\ldots,\bu_{n-1}$ producing $\bG_{n-1}$.
To do this, we will use a ``cavity-method'' style argument: we will view all vectors at distance $> \ell = \frac{\log n}{\log\log n}$ from $S$ as fixed and arbitrary, and then exactly compute the marginal distributions over $\bu_i$ for $i$ at distance $\le \ell$ from $S$, conditional on forming $\bG_{n - 1}$.

\subsection{Neighborhood containment as a constraint satisfaction problem.}

We first reduce the need for high-probability estimates for $\Pr_{\grg, \bG_{k - 1}}[N(k) = S]$ to obtaining estimates for $\Pr_{\grg, \bG_{k - 1}}[N(k) \supseteq S]$ instead.
This simplification is possible because the measure of anti-cap intersections concentrates dramatically better than the measure of cap intersections.
With this step, we eliminate the need to study anti-correlations between $\bv_i, \bv_j$ that do not have an edge between them.

Given $S$ and $\bG_{n -1}$ (and its corresponding vectors), we fix all vectors except those corresponding to the depth-$\frac{\log n}{\log\log n}$ neighborhood of $S$ in $\bG_{n - 1}$, which is with high probability a union of trees. 
To formally analyze the distribution of the unfixed vectors upon resampling them, we set up a $2$-CSP (constraint satisfaction problem) instance over a continuous alphabet that encodes the edges of $\bG_{n - 1}$ within the trees around $S$: each node has a vector-valued variable in $\bbS^{d-1}$, and the constraints are that nodes joined by an edge must have vectors with inner product at least $\tau$.

\paragraph{Belief propagation.} Since our $2$-CSPs are over trees, the {\em belief propagation (BP) algorithm} exactly computes the marginal distribution of each variable vector (see \pref{sec:prelim-bp} for the definition of BP).
Using our results on the concentration of $\Pr_{\bx \sim \calL}[\langle \bx, \bz \rangle \geq \tau]$ over $\bz$ uniform on $\bbS^{d - 1}$, we can quantify the TV distance between the marginal distributions of our resampled vectors and the uniform distribution over $\bbS^{d - 1}$.
At a high level, the farther some $\bv_i$ is from a fixed vector in our 2-CSP, the closer its distribution is to uniform. 
The key insight is that the message from $i$ to its neighbor $j$ in our belief propagation algorithm correspond to a convolution of the marginal distribution of $\bv_i$ with a spherical cap.
We can then use our concentration of measure for spherical caps from \pref{sec:transport-result} to show that convolutions of spherical caps mix to uniform rapidly, causing the correlations between far away vertices to decay.
This can be seen as a form of the ``decay of correlations'' phenomenon.
This analysis gives us the finer-grained control over $\Pr[N_{\bG}(n) = S]$ needed to conclude \pref{thm:sparse}.\footnote{Unfortunately, this analysis only works for $p = \frac{\alpha}{n}$; otherwise the neighborhoods around vertices in $S$ are only trees at a depth which is too shallow for the correlations to decay sufficiently, so the resampled vectors' distributions are not close enough to uniform.}

\subsection{Further connections between \erdos-\renyi and geometric random graphs.}

In the course of proving \pref{thm:sparse}, we prove \pref{prop:dom}, which demonstrates a coupling of $\bG_- \sim \ER(n, p- o(p))$, $\bG \sim \grg_d(n, p)$, and $\bG_+ \sim \ER(n, p + o(p))$ that satisfies $\bG_- \subseteq \bG \subseteq \bG_+$ with high probability. 
We use this coupling in the proof of \pref{thm:sparse} to leverage known structural results on \erdos--\renyi graphs to reason about the structure of random geometric graphs. 
For instance, to upper bound the probability that the depth-$\ell$ neighborhood of some $i \in [n]$ forms a tree under $\grg_d(n, p)$, we rely on known bounds on the probability of this event under $\ER(n, p(1 \pm o(1)))$. 
We remark that this coupling may be of independent interest.

\section{Preliminaries}\label{sec:prelims}

\paragraph{Basic notation.}  Throughout this paper, we use {\bf boldface} for random variables.  
We use $\log x$ to denote the natural base logarithm, and for $x \in [0,1]$, $H(x) = x \ln \frac{1}{x} + (1-x) \ln \frac{1}{1-x}$ denotes the binary entropy function (with the understanding that $H(0) = H(1) = 0$).
We use standard big-$O$ notation, and we use $\wt{O}$ and $\wt{\Omega}$ to hide $\polylog n$ factors. 
The notation $f(x) \gg g(x)$ to denote that $\lim_{x \to \infty}\frac{g(x)}{f(x)} = 0$; the argument $x$ will be clear from context.

\par We use $\Unif$ to denote the uniform distribution on $\bbS^{d-1}$.
Given a function $f:\bbS^{d-1}\to\R$, we define its $p$-th norm for $p < \infty$ as:
\[
    \norm*{f}_p \coloneqq \left(\E_{\bz\sim\Unif} |f(\bz)|^p\right)^{1/p}
\]
and for $p = \infty$ as:
\[
    \norm*{f}_{\infty} \coloneqq \sup_z |f(z)|.
\]
Given a distribution $\nu$ on $\bbS^{d-1}$, we overload notation and use $\norm{\nu}_{p}$ to denote $\norm*{\frac{d\nu}{d\Unif}}_{p}$.
We will also frequently use the symbol $\nu$ itself to denote its relative density to $\rho$.
For a set $A$, we overload notation and use $A$ to denote the uniform distribution on $A$ when it is clear from context.
We'll use $\bX \mid \calE$ to denote the random variable $\bX$ sampled from the conditional distribution of $\bX$ conditioned on the event $\calE$.

Given a graph $G$, and a subset of its vertices $S$, we use $G[S]$ to denote the induced subgraph of $G$ on $S$.  
We use $B_{G}(v,\ell)$ to denote the ball of radius-$\ell$ around a vertex $v$ in graph $G$, and similarly use $S_G(v,\ell)$ to denote the corresponding sphere of radius-$\ell$.
We use $N_G(i)$ to denote the neighbors of vertex $i$ in graph $G$. 

Given a collection of vectors $V = (v_1,\dots,v_n)\in\parens*{\bbS^{d-1}}^n$ the associated \emph{geometric graph} denoted $\GG(V,p) = ([n],E)$ is given by choosing the edge set as all $\{i,j\}$ where $\langle v_i, v_j\rangle \ge \tau(p)$.  In particular, when $\bV\sim\Unif^{\otimes n}$, $\GG(\bV,p)$ is distributed as $\grg_d(n,p)$.
For an $n$-vertex graph $G$, we use $\rho^G$ to denote the conditional distribution $\rho^G = \rho^{\otimes n} \mid \GG(V,p) = G$.

\subsection{Divergences between probability distributions}
We use the relative entropy to compare probability distributions.
\begin{definition}[Relative Entropy] \label{def:kl}
    Given two probability distributions $\mu$ and $\nu$ over $\Omega$ where $\nu$ is absolutely continuous with respect to $\mu$, the \emph{relative entropy} between $\mu$ to $\nu$ is:
    \[
        \dkl(\nu \| \mu) \coloneqq \int_{\Omega} \ln\left(\frac{d\nu}{d\mu}(x)\right) d\nu(x).
    \]
    If $\nu$ is not absolutely continuous with respect to $\mu$, then we define the relative entropy to be $\infty$.
\end{definition}
\noindent A simple but useful observation is the following.
\begin{observation} \label{obs:bounded-dists}
    If for all $x\in\Omega$, $\frac{d\nu}{d\mu}(x)\le C$: 
        (1) $\dkl(\nu \| \mu)\le \ln C$, and (2) for any event $\calE$: $\nu(\calE) \le C\cdot\mu(\calE)$.
\end{observation}

We will use Pinsker's inequality to bound the total variation distance between two probability distributions in terms of the relative entropy.  
\begin{theorem}[Pinsker's inequality] \label{thm:pinsker}
    For distributions $\mu,\nu$ over the same domain,
        $\dtv{\nu}{\mu}^2 \le \frac{1}{2}\dkl(\nu \| \mu)$.
\end{theorem}
\noindent See e.g. \cite[Theorem 2.16]{Mas07} for a proof.

\subsection{Sphere caps and dot products of unit vectors}

Here, we introduce some useful bounds on the measure of sphere caps, and a concentration bound on the dot products of two unit vectors. 
Proofs for some of these lemmas are provided in the appendix, along with some additional facts about the geometry of the unit sphere that we use to prove them.

\begin{definition}[{$p$-cap}]
	For a vector $v\in\bbS^{d-1}$, its \emph{$p$-cap} is $\scap_p(v) \coloneqq \{x\in\bbS^{d-1}:\langle v, x\rangle \ge \tau(p)\}$.  Similarly, we define its \emph{$p$-anticap} as $\anticap_p(v) \coloneqq \{x\in\bbS^{d-1}:\langle v, x\rangle < \tau(p)\}$.  We drop the $p$ in the subscript when its value is clear from context.
\end{definition}

\noindent Recall L\'evy's theorem for concentration of measure on the unit sphere \cite[Theorem 14.1.1]{Mat13}. We use it to upper bound the measure of a sphere cap with threshold $\tau$.  
\begin{lemma}   \label{lem:conc-measure}
	Let $y\in\bbS^{d-1}$ be any vector.  Then: $\Pr_{\bw\sim\bbS^{d-1}}[|\langle \bw, y\rangle| \ge \tau] \le 4\exp(-\tau^2d/2)$.
\end{lemma}

\noindent We now present a convenient upper bound on the dot product threshold $\tau(p)$ of a $p$-cap. 
Its proof is provided in \pref{app:cap-proofs}.

\begin{lemma} \torestate{\label{lem:upper-bound-tau}
	For any $p \leq \frac{1}{2}$, we have $\tau(p) \leq \sqrt{\frac{2 \log(1 / p)}{d}}$.}
\end{lemma}

\noindent The next lemma helps us understand the deviations in cap volume $p$ when we make small adjustments to its dot product threshold $\tau(p)$.
Its proof is also in \pref{app:cap-proofs}.

\begin{lemma} \torestate{\label{lem:beta-concentration}
	Fix $x \in \bbS^{d - 1}$. Let $p \coloneqq \Pr_{z \sim \rho}[\langle z, x \rangle \geq \tau]$. For any $\varepsilon \geq 0$, there is a universal constant $\Cbetaac$ such that:
	\begin{align*}
		\Pr_{z \sim \rho}[\tau - \varepsilon \leq \langle z, x \rangle \leq \tau + \varepsilon] &\leq p \cdot \left(\Cbetaac \varepsilon \exp(2d \tau \varepsilon) \sqrt{d \log(1 / p)} \right) 
	\end{align*} }
\end{lemma}

\subsection{Subgaussian random variables}

We will need concentration inequalities for martingales that arise as a sum of subgaussian random variables.
\begin{definition}[Subgaussian norm]
    The \emph{subgaussian norm} of a real-valued random variable $\bX$ is defined as:
    \[
        \|\bX\|_{\subgnorm} \coloneqq \inf \left\{K>0:\E\exp(\bX^2/K^2)\le 2\right\}.
    \]
    We say $\bX$ is a \emph{subgaussian random variable} if $\|\bX\|_{\subgnorm}<\infty$.
\end{definition}

We will need a version of Azuma's inequality for martingales with centered subgaussian increments.
A proof may be found in \cite{RVH}, Lemma 3.7.
\begin{lemma}[Subgaussian martingale concentration inequality]\torestate{ \label{lem:martingale-azuma}
There exists a constant $\CMartingale > 0$ such that if $\bX_0,\bX_1,\dots,\bX_m$ is a martingale sequence with respect to a filtration $(\calV_t)_{t\in[m]}$ and $K_i\coloneqq \sup_{\calV_{i-1}}\left\|\bX_i-\bX_{i-1}|\calV_{i-1}\right\|_{\subgnorm} < \infty$ for all $i\in[m]$, then:
    \[
        \Pr\left[|\bX_m-\bX_0| \ge t\right] \le 2\exp\left(-\CMartingale\frac{t^2}{\sum_{i=1}^m K_i^2}\right).
    \]}
\end{lemma}

Also, we will make use of the following statement to bound the subgaussian norm of a random variable in terms of its tail probabilities.

\begin{lemma}[\cite{Ver18}, Proposition 2.5.2]\label{lem:subgaussian-equiv}
There exists a constant $\CSubgauss > 0$ such that if $\bX$ is a random variable satisfying $\Pr[|\bX| > t] \le 2\exp\left(-\frac{t^2}{K^2}\right)$ for all $t \ge 0$, then $\|\bX\|_{\subgnorm} \le \CSubgauss\cdot K$.
\end{lemma}

\subsection{Random graphs}

\par In this section, we include some facts about \erdos--\renyi random graphs.  First, we state a high-probability upper bound of $O\parens*{\frac{\log n}{\log\log n}}$ on the maximum degree in both random graph models when $p = \frac{\alpha}{n}$ for constant $\alpha$. This allows us to ignore graphs where the max degree is too high when we upper bound $\dtv {\ER(n, p)}{\grg_d(n, p)}$.  Since the degree of each vertex is distributed like $\Binom(n,p)$, by applying the standard tail bound for a Binomial random variable and taking a union bound over all vertices we get the following.
\begin{lemma} \label{lem:degree-bound}
    Let $\Delta(G)$ denote the maximum degree of a graph $G$. If $p = \frac{\alpha}{n}$, then for both $\bG \sim \ER(n, p)$ and $\bG \sim \grg_d(n, p)$ and for all $d$:
    $$
        \Pr[\Delta(\bG) \geq k] \leq n\cdot\left(\frac{k}{e\alpha}\right)^{-k}.
    $$
\end{lemma}
More generally, we know the following bound on the number of vertices in an \erdos--\renyi graph at distance $\leq \ell$ from any given vertex, as well as a high probability statement about their structure.
\begin{lemma}[{\cite[Lemma 29]{BLM15}}] \label{lem:ball-bound}
    Let $\bG\sim\ER(n,p)$ for $p=\frac{\alpha}{n}$ for constant $\alpha$.  
    Define $\bh_t(v)$ as the number of vertices with distance exactly $t$ from $v$ in $\bG$.  Then for any vertex $v$, there are constants $c, C$ such that:
    \[
        \Pr\bracks*{\exists t \ge 0: \bh_t(v) > s\alpha^{t}} \le C\exp(-cs).
    \]
\end{lemma}

\begin{lemma}[{\cite[Lemma 30]{BLM15}}] \label{lem:neighborhood-tree}
	Let $\bG\sim\ER(n,p)$ for $p=\frac{\alpha}{n}$ for constant $\alpha$, and let $B_{\bG}(v, t)$ be the set of all vertices with distance $\leq t$ from vertex $v$ in $\bG$.
	Then, for any vertex $v$, there is a constant $c'$ such that:
	$$
	\Pr[B_{\bG}(v, t) \text{ is not a tree}] \leq \frac{c' \alpha^t}{n}
	$$
\end{lemma}

\subsection{Conditional vector distributions}
Given an $n$-vertex graph $G$, we are interested in $\bV\sim\Unif^{\otimes n}|\GG(\bV,p)=G$, whose distribution we shorten to $\VecDist{G}$.  A simple but crucial observation for us is the following.
\begin{observation} \label{obs:unif-to-conditional}
    Let $f:\parens*{\bbS^{d-1}}^n\to\{0,1\}$ be a Boolean-valued function.  If $\E_{\bV\sim\Unif^{\otimes n}} f(\bV) \le \delta$,
    then:
    \[
        \Pr_{\bG\sim\grg_d(n,p)} \bracks*{\E_{\bV\sim\VecDist{\bG}} f(\bV) \ge \sqrt{\delta} } \le \sqrt{\delta}.
    \]
\end{observation}
\begin{proof}
    We can write:
        $\E_{\bV\sim\Unif^{\otimes n}} f(\bV) = \E_{\bG\sim\grg_d(n,p)} \E_{\bV\sim\VecDist{\bG}} f(\bV) \le \delta$.
    The statement then follows from Markov's inequality.
\end{proof}

\subsection{Constraint satisfaction problems and Belief Propagation}
\label{sec:prelim-bp}

\begin{definition}
A \emph{constraint satisfaction problem instance} (CSP instance) $\Inst$ consists of a \emph{variable set} $V$ and a \emph{constraint set} $E$. The variables $v \in V$ each belong to an alphabet $\Sigma$, while the constraints $f \in E$ consist of a $k$-tuple of variables $\partial f$ and a function $\psi_f:\Sigma^k\to\{0,1\}$.  An assignment to variables $c: V \to \Sigma$ is \emph{satisfying} if for each constraint $f \in F$, $\psi_f(c(\partial f)) = 1$, i.e. it is satisfied.
\end{definition}

\begin{remark}
    In our setting, we choose $\Sigma = \bbS^{d-1}$.
\end{remark}

\begin{definition}
    We can represent any CSP instance $\Inst$ as a bipartite graph $F$, which we call a \emph{factor graph}.  The two sides of the bipartition are $V$ and $E$, and we place an edge between $v \in V$ and $f \in E$ if $v$ participates in $f$. We use $\partial v$ and $\partial f$ to denote the neighborhoods of variables $v$ and clauses $f$, respectively, in this graph. 
\end{definition}

When the factor graph $F$ is a does not contain cycles, the marginal on a variable $v$ can be computed exactly from the fixed point of the belief propagation algorithm. We state this precisely below.
\begin{definition}[Belief propagation fixed point]
    A \emph{belief propagation fixed point} for a factor graph $F$ is a collection of \emph{messages}
    \[
        \left\{ m^{v\to f}, m^{f\to v} \right\}_{v\in V, f\in E}
    \]
    for all pairs $v,f$ such that $f$ is a neighbor of $v$, where each message is a probability distribution on $\Sigma$, such that
    \begin{align} 
    m^{f \to v}(x) &\propto \int_{c\large|_{\partial f}: c(v) = x} \psi_f(c\large|_{\partial f}) \prod_{v' \in \partial f \setminus v} m^{v' \to f}(c(v')) \label{eq:clause-to-var} \\
    m^{v \to f}(x) &\propto \prod_{f' \in \partial v \setminus f} m^{f' \to v}(x) \label{eq:var-to-clause}
\end{align}
\end{definition}

\begin{theorem}[e.g. \cite{MM}, Theorem 14.1] \label{thm:BP-forest-marginals}
    Suppose $F$ is a forest factor graph corresponding to a CSP instance $\Inst$ where every vertex is attached to a unary constraint.  Then there is a unique belief propagation fixed point and the marginal distribution $\nu$ on variable $v$ over the uniform distribution over satisfying assignments to $i$ is given via the following formula.
    \[
        \nu \propto \prod_{f\in\partial v} m^{f\to v}.
    \]
\end{theorem}

\section{Concentration via optimal transport} \label{sec:transport-result}
In this section we establish that for a probability distribution $\nu$ over $\bbS^{d-1}$, a random $p$-cap on the sphere contains a $p$-fraction of $\nu$'s measure with high probability, where the strength of the concentration depends on $\|\nu\|_{\infty}$.
We do so by analyzing the optimal transport mapping $\calD$ between $\nu$ and the uniform measure $\rho$.

\subsection{Optimal transport and the Wasserstein metric}

The Wasserstein metric quantifies the ``physical distance'' between a pair of probability distributions.
We say $\pi(x, y)$ is a transport coupling between distributions $\mu$ and $\nu$ if $\pi(x, \cdot) = \mu(x)$ and $\pi(\cdot, y) = \nu(y)$.
\begin{definition}[Wasserstein Distance]
	Let $\mu$ and $\nu$ be two probability distributions over $\Omega$, and $\Pi$ be the set of all transport couplings $\pi(x, y)$ between $\mu$ and $\nu$. Then, their \emph{Wasserstein-2 distance} is 
	$$
	W_2(\mu, \nu) \coloneqq \sqrt{\inf_{\pi \in \Pi} \int_{\Omega \times \Omega} \|x - y\|_2^2 d\pi(x, y)}
	$$
	It is straightforward to verify that $W_2(\cdot, \cdot)$ is in fact a metric.
\end{definition}
In other words, the square of the Wasserstein distance is the average squared Euclidean distance $\|\bx-\by\|^2$ between $(\bx,\by)\sim \pi$ for the most efficient ``transport coupling'' $\pi$.
For intuition, we may think of $\mu$'s density of a pile of sand over $\Omega$, and of $\pi$ as a map that shifts grains of sand from $\mu$ to form the shape of $\nu$ in such a way that minimizes the average distance traveled.

\begin{fact}[{\cite[Proposition 9.1.2]{BGL13}}]    \label{fact:opt-coupling-exists}
    Given two probability distributions $\mu$ and $\nu$ over $\Omega$, there exists a coupling $\pi$ such that:
    \[
        \int_{\Omega \times \Omega} \|x - y\|_2^2 d\pi(x, y) = W_2(\mu,\nu)^2.
    \]
    We call any such $\pi$ an \emph{optimal coupling}.
\end{fact}

We specifically will need bounds on the Wasserstein-2 distance between an arbitrary distribution $\nu$ on the unit sphere and the uniform distribution on the unit sphere $\Unif$.  
We obtain a handle on the distances we need via an \emph{entropy-transport inequality}, which bounds the Wasserstein-2 distance in terms of the relative entropy.
The following is a direct corollary of \cite[Theorem 22.17, part (i)]{Vil08} and \cite[Corollary 2]{DEKL14}.
\begin{lemma}[{Talagrand's $T_2$ inequality on the sphere}]\label{lem:was-kl}
	For any distribution $\nu$ on $\bbS^{d-1}$, and $\rho$ the uniform measure on $\bbS^{d-1}$,
	\[
	W_2(\nu, \rho) \le \sqrt{\frac{2}{d-1} \cdot \KL{\nu}}.
	\]
\end{lemma}

	For a reader interested in the proof of \pref{lem:was-kl}, we recommend the proof of the analogous statement in Gaussian space by \cite{Tal96} due to its relative simplicity.  In fact, it is possible to derive a slightly weaker bound of $\sqrt{\frac{2}{d} \cdot \dkl(\nu \| \Unif)} + \frac{2}{\sqrt{d}}$ from the Gaussian case via an elementary proof.  We also find it worthwhile to point to \cite[Theorem 9.2.1]{BGL13} for a comprehensive exposition.

\subsection{The concentration of sphere cap measure}

    Let $\Unif$ be the uniform measure over $\bbS^{d-1}$, let $\nu$ be a probability measure over $\bbS^{d-1}$, and let $\calD$ be a coupling between them.
For $(\bx,\by)\sim\calD$ we use the convention that $\bx$ is distributed according to $\nu$ and $\by$ is distributed according to $\Unif$. 
    We first prove quantitative bounds on the concentration of $\|\bx - \by\|$ when $(\bx, \by) \sim \calD$. 
\begin{lemma} \label{lem:transport_concentration}
    Let $\nu$ be a distribution over $\bbS^{d-1}$ and let $\calD$ be the optimal transport coupling of $\nu$ and $\Unif$. 
    Then for all $t > 0$,
    \[
        \Pr_{(\bx,\by) \sim \calD}\left[ \|\bx-\by\| \ge t + \sqrt{\tfrac{2}{d-1}\ln \norm{\nu}_{\infty}} \right] \le \exp\left(- \frac{d-1}{8} t^2\right).
    \]
\end{lemma}
\begin{proof}
Let $(\bx,\by) \sim \calD$ with $\bx \sim \nu$ and $\by \sim \Unif$.
    Let $\calE_s$ be the event that $\|\bx-\by\| \ge s$, and let $p_s \coloneqq \Pr_{\calD}[\calE_s]$.
    Let $\calD^s = \calD| \calE_s$, and let $\calD_{\nu}^s, \calD_{\Unif}^s$ be the marginal distributions of $\calD^s$ on $\bx$ and $\by$ respectively.
    We claim
    \[
        W_2(\calD_{\nu}^s, \calD_{\Unif}^s) \ge s.
    \]
    Indeed, suppose not; then, one could obtain a coupling which further decreases the transport distance between $\nu$ and $\Unif$, contradicting the optimality of $\calD$.
    Now, since $W_2$ is a metric:
    \begin{equation}
        s 
        \le W_2(\calD_{\nu}^s, \calD_{\Unif}^s) 
        \le W_2(\calD_{\Unif}^s,\Unif) + W_2(\calD_{\nu}^s,\Unif)
        \le \sqrt{\frac{2}{d-1}\KL{\calD_{\Unif}^s}} + \sqrt{\frac{2}{d-1}\KL{\calD_{\nu}^s}},\label{eq:triangle}
    \end{equation}
    where we have used the triangle inequality in conjunction with Talagrand's $T_2$ inequality (\pref{lem:was-kl}).
    Finally, since $\calD_{\Unif}^s(x) = \Unif(x | \calE_s) \le \frac{1}{p_s}\Unif(x)$,
    \[
        \dkl(\calD_{\Unif}^s \| \Unif ) 
        = \int_{\bbS^{d-1}} \calD_{\Unif}^s(x) \cdot\ln \frac{\calD_{\Unif}^s(x)}{\Unif(x)} dx
        \le \int_{\bbS^{d-1}} \calD_{\Unif}^s(x) \cdot\ln \frac{1}{p_s} dx 
        =\ln \frac{1}{p_s}, 
    \]
    and similarly, because $\calD_{\nu}^s(x) = \nu(x | \calE_s) \le \frac{1}{p_s}\nu(x) \le \frac{\norm{\nu}_{\infty}}{p_s} \Unif(x)$,
    \[
        \dkl(\calD_{\nu}^s \| \Unif ) 
        = \int_{\bbS^{d-1}} \calD_{\nu}^s(x) \cdot\ln \frac{\calD_{\nu}^s(x)}{\Unif(x)} dx
        \le \int_{\bbS^{d-1}} \calD_{\nu}^s(x) \cdot\ln \frac{\norm{\nu}_{\infty}}{p_s} dx
        = \ln \frac{1}{p_s} + \ln \norm{\nu}_{\infty}.
    \]
    Putting these together with \pref{eq:triangle} and using $\sqrt{a + b} \le \sqrt{a} + \sqrt{b}$, we have
    \[
        s \le 2\sqrt{\frac{2}{d-1} \ln \frac{1}{p_s}} + \sqrt{\frac{2}{d-1}\ln \norm{\nu}_{\infty}},
    \]
    and then re-arranging we have
    \[
        p_s \le \exp\left(- \frac{d-1}{8} \left(s - \sqrt{\frac{2}{d-1}\ln \norm{\nu}_{\infty}}\right)^2 \right),
    \]
    when $s \ge \sqrt{\frac{2}{d-1}\ln \norm{\nu}_{\infty}}$.   Applying a change of variables completes the proof.
\end{proof}

Having established in \pref{lem:transport_concentration} that the optimal transport map $\pi(\nu, \Unif)$ between $\bx \sim \nu$ and $\by \sim \Unif$ has bounded length with high probability, we can translate this into a tail bound for the inner product $\iprod{\bz, \bx - \by}$ for a random vector $\bz \sim \Unif$.

\begin{lemma} \label{lem:inner-prod-concentration}
    Let $\nu$ be a distribution on $\bbS^{d-1}$, and let $\calD$ be the optimal transport coupling of $\nu$ and $\Unif$.  For $z\in\bbS^{d-1},t\in\R^+$ and any $\kappa > 0$, define $X(z,t)$ as:
    \[
        X(z,t) \coloneqq \Pr_{(\bx,\by)\sim\calD}\left[\left|\langle z, \bx-\by \rangle\right| \ge \left(\sqrt{\frac{2}{d-1}\ln\norm{\nu}_{\infty}} + \sqrt{\frac{8\kappa}{d-1}} \right)\cdot t \right].
    \]
    Then:
    \[
        \Pr_{\bz\sim\Unif}\left[X(\bz,t) \ge 2\exp(-dt^2/4) + \exp(-\kappa)\right] \le 2\exp(-dt^2/4).
    \]
\end{lemma}

\begin{remark}
    One should think of $X(z, t)$ as a measure of how often a randomly chosen transport vector $\bx - \by$, with $(\bx, \by) \sim \calD$, has a large projection in the $z$ direction.
\end{remark}

\begin{proof}[Proof of \pref{lem:inner-prod-concentration}]
    For any $z\in\bbS^{d-1}$ and $(\bx,\by)\sim\calD$, suppose $|\langle z, \bx-\by\rangle| \ge \left( \sqrt{\frac{2}{d-1}\ln\norm*{\nu}_{\infty}} + \sqrt{\frac{8\kappa}{d-1}} \right)\cdot t$, then $|\langle z, \bx-\by\rangle| \ge \|\bx-\by\| \cdot t$ or $\|\bx-\by\| \ge \sqrt{\frac{2}{d-1}\ln\norm{\nu}_{\infty}} + \sqrt{\frac{8\kappa}{d-1}}$.  Defining
    \[
        Y(z,t) \coloneqq \Pr_{(\bx,\by)\sim\calD}\left[\left|\langle z, \bx-\by \rangle\right| \ge \|\bx-\by\|\cdot t\right],
    \]
    we can write
    \[
        X(z,t) \le Y(z,t) + \Pr_{(\bx,\by)\sim\calD}\left[\|\bx-\by\| \ge \sqrt{\frac{8\kappa}{d-1}} + \sqrt{\frac{2}{d-1}\ln\norm{\nu}_{\infty}}\right] \le Y(z,t) + \exp\left(-\kappa\right) \numberthis \label{eq:bound-x-two-events}
    \]
where to obtain the upper bound we have applied \pref{lem:transport_concentration}.
    Next, we prove that
    \[
        \Pr_{\bz\sim\Unif}\left[Y(\bz,t) \ge 2\exp(-dt^2/4)\right] \le 2\exp(-dt^2/4), \numberthis \label{eq:Y-prob-bound}
    \]
    by showing $\E\left[Y(\bz,t)\right]\le 4\exp(-dt^2/2)$, which implies \pref{eq:Y-prob-bound} via Markov's inequality.
    \begin{align*}
        \E_{\bz\sim\Unif}\left[Y(\bz,t)\right] &= \Pr_{\substack{\bz\sim\Unif \\ (\bx,\by)\sim\calD}}\left[|\langle \bz, \bx-\by\rangle| \ge \|\bx-\by\|\cdot t\right] = \Pr_{\substack{\bz\sim\Unif \\ (\bx,\by)\sim\calD}}\left[|\langle \bz, \frac{\bx-\by}{\norm{bx - \by}} \rangle| \ge t\right] \\
        &\le 4\exp(-dt^2/2). &\text{(by \pref{lem:conc-measure})}
    \end{align*}
    We can then complete the proof starting at \pref{eq:Y-prob-bound} as follows:
    \begin{align*}
        \Pr_{\bz\sim\Unif}\left[Y(\bz,t) + \exp\left(-\kappa\right) \ge 2\exp(-dt^2/4) + \exp\left(-\kappa\right)\right] &\le 2\exp(-dt^2/4) \\
        \Pr_{\bz\sim\Unif}\left[X(\bz,t) \ge 2\exp(-dt^2/4) + \exp(-\kappa)\right] &\le 2\exp(-dt^2/4). &\text{(by \pref{eq:bound-x-two-events})}
    \end{align*}
This yields the desired conclusion.
\end{proof}

Since a $p$-cap around $x$ is given by the set of vectors $z$ with inner product $\iprod{x,z} \ge \tau(p)$, and $\iprod{x,z} = \iprod{y,z} \pm |\iprod{x-y,z}|$, we can finally use \pref{lem:inner-prod-concentration} to relate $X(z) = \Pr_{\bx \sim \nu}[\iprod{z,\bx} > \tau(p)]$, the measure of $\nu$ that falls into the $p$-cap of $z$, to $p = \Pr_{\by \sim \Unif} [\iprod{z,\by} > \tau(p)]$.
That is, we can now show that $X(\bz)$ for $\bz \sim \Unif$ concentrates tightly around $p$, so that most vectors in $\bbS^{d-1}$ contain very close to a $p$-fraction of $\nu$'s mass in their $p$-caps.

\begin{lemma} \label{lem:two-sets-concentrate}
    Let $\nu$ be a distribution on $\bbS^{d-1}$.  
For $z\in\bbS^{d-1}$, let $X(z)\coloneqq \Pr_{\bx\sim\nu}[\langle \bx, z \rangle > \tau(p)]$, and for any $\kappa > 0$, let $u(t) \coloneqq \left(\sqrt{\frac{8}{d-1}\ln\norm*{\nu}_{\infty}} + \sqrt{\frac{8\kappa}{d-1}}\right)\cdot t$.  Then for any $t \ge 0$:
    \[
        \Pr_{\bz\sim\Unif}\left[|X(\bz)-p| > p\cdot \eps(t)\right] \le 2\exp(-dt^2/4)
    \]
    where $\eps(t)\coloneqq \Cbetaac \cdot u(t)\cdot\exp\left(2d\tau(p)u(t)\right) \cdot \sqrt{d\log\frac{1}{p}}+\frac{2\exp(-dt^2/4)+\exp(-\kappa)}{p}$.
\end{lemma}
\begin{proof}
    Let $\calD$ be the optimal transport coupling between $\nu$ and $\Unif$.  For any $z\in\bbS^{d-1}$ and $t\ge 0$:
    \begin{align*}
        X(z) &= \Pr_{(\bx,\by)\sim\calD}\left[\langle \by, z\rangle > \tau(p) - \langle z, \bx-\by\rangle\right] \\
        &\le \Pr_{(\bx,\by)\sim\calD}[\langle \by, z\rangle > \tau(p) - \max\{\langle z,\bx-\by\rangle, u(t)\}] \\
        &\le \Pr_{\by\sim\Unif}[\langle \by, z\rangle > \tau(p)-u(t)] + \Pr_{(\bx,\by)\sim\calD}[|\langle z,\bx-\by\rangle|>u(t)] \\
        &\le p\left(1+\Cbetaac \cdot u(t)\cdot\exp\left(2d\tau(p)u(t)\right) \cdot \sqrt{d\log\frac{1}{p}}\right) + \Pr_{(\bx,\by)\sim\calD}[|\langle z,\bx-\by\rangle|>u(t)] \numberthis \label{eq:upper-bound-det-X}
    \end{align*}
    where the last step of the chain of inequalities follows from \pref{lem:beta-concentration}. Identically,
    \[
        X(z) \ge  p\left(1-\Cbetaac \cdot u(t)\cdot\exp\left(2d\tau(p)u(t)\right) \cdot \sqrt{d\log\frac{1}{p}}\right) - \Pr_{(\bx,\by)\sim\calD}[|\langle z,\bx-\by\rangle|>u(t)] \numberthis \label{eq:lower-bound-det-x}
    \]
    Then by \pref{lem:inner-prod-concentration}, when $\bz\sim\Unif$, we can obtain an upper bound of $2\exp(-dt^2/4) + \exp(-\kappa)$ on the second term in the RHS of \pref{eq:upper-bound-det-X} and \pref{eq:lower-bound-det-x} that holds except with probability $2\exp(-dt^2/4)$, which implies
    \[
        \Pr_{\bz\sim\Unif}\left[|X(\bz)-p| > p\cdot \Cbetaac \cdot u(t)\cdot\exp\left(2d\tau(p)u(t)\right) \cdot \sqrt{d\log\frac{1}{p}} + 2\exp(-dt^2/4) + \exp(-\kappa) \right] \le 2\exp(-dt^2/4).
    \]
    thus completing the proof.
\end{proof}

\subsection{Different parameterizations of the sphere cap concentration}

We'll now derive a few useful corollaries of \pref{lem:two-sets-concentrate}.
First, for intuition, consider the following immediate consequence regarding the intersection of a set in $\bbS^{d-1}$ with a random sphere cap.
\begin{corollary}   \label{cor:cap-intersect-simple}
    Let $d\ge\log^{10} n$ and let $Q\subseteq\bbS^{d-1}$ be a set such that $\Unif(Q)\ge\frac{1}{n^{\log^3 n}}$.  Then for $\bz\sim\Unif$:
    \[
        \frac{\Unif\parens*{\scap(\bz)\cap Q}}{\Unif\parens*{\scap(\bz)}\cdot\Unif\parens*{Q}} \notin \parens*{1 \pm \frac{\log^5 n}{\sqrt{d}} \pm n^{-\log^2 n} }
    \]
    with probability at most $n^{-\Omega(\log^3 n)}$.
\end{corollary}

We will make use of the following convenient specialization of \pref{lem:two-sets-concentrate}.
\begin{corollary}   \label{cor:two-sets-subexp}
    Let $\nu$ be a distribution on $\bbS^{d-1}$ and for $z\in\bbS^{d-1}$, let $X(z)\coloneqq \Pr_{\bx\sim\nu}\bracks*{ \angles*{\bx, z} \ge \tau(p) }$.  Then for $s \le 1$ and for some constant $\Cconcalt$,
    \[
        \Pr_{\bz \sim \Unif}\bracks*{|X(\bz)-p|>ps} \le 2\exp\parens*{-\frac{ds^2}{\Cconcalt\parens*{\sqrt{\ln\norm*{\nu}_{\infty}} + \sqrt{\ln\frac{d}{p}} }^2 \cdot \log\frac{1}{p} \cdot \log\frac{d}{p}  }}
    \]
\end{corollary}
\begin{proof}
    The idea is to apply \pref{lem:two-sets-concentrate} by setting $\kappa = 4\ln\frac{d}{p}$ and using the parameterization
    \[
        t = \frac{s}{2\parens*{\sqrt{8\ln\norm*{\nu}_{\infty}} + \sqrt{32\ln\frac{1}{p}} }\sqrt{\ln\frac{1}{p}} }.
    \]
    Clearly, the statement is true when
    \[
        s \le \sqrt{\frac{\Cconcalt}{2d}}\cdot\parens*{ \sqrt{\ln\norm*{\nu}_{\infty} } + \sqrt{\ln\frac{d}{p}} }\cdot\sqrt{\log\frac{1}{p} \cdot \log\frac{d}{p}}
    \]
    since $\Pr\bracks*{|X(z)-p|>ps} \le 1 < 2\exp\parens*{-1/2}$.
    Hence, we restrict our attention to when
    \[
        s\in H \coloneqq \bracks*{\sqrt{\frac{\Cconcalt}{2d}}\cdot\parens*{ \sqrt{\ln\norm*{\nu}_{\infty} } + \sqrt{\ln\frac{d}{p}} }\cdot\sqrt{\log\frac{1}{p} \cdot \log\frac{d}{p}}, 1}.
    \]
    Let $\eps(t)$ and $u(t)$ be as in the statement of \pref{lem:two-sets-concentrate}.  Using $s\le 1$ and \pref{lem:upper-bound-tau}, we know $\exp(2d\tau(p)u(t)) \le O(1)$.  This tells us that for some constant $C$,
    \[
        \eps(t) \le C\parens*{\sqrt{\frac{8}{d-1}\ln\norm*{\nu}_{\infty}} + \sqrt{\frac{32\ln\frac{1}{p}}{d-1}}}\sqrt{d\log\frac{1}{p}}t + \frac{2\exp(-dt^2/4) + \frac{p^4}{d^4}}{p}.
    \]
    We can choose our constant $\Cconcalt$ to be a large enough so that when $s\in H$, then $t \ge 4\sqrt{\frac{\log\frac{d}{p}}{d}}$.
    Observe that once $t\ge4\sqrt{\frac{\log\frac{d}{p}}{d}}$, for large enough $d$:
    \begin{align*}
        \eps(t) &\le 2C\parens*{\sqrt{\frac{8}{d-1}\ln\norm*{\nu}_{\infty}} + \sqrt{\frac{32\ln\frac{1}{p}}{d-1}}}\sqrt{d\log\frac{1}{p}}t\\
        &= C'\cdot 2\parens*{\sqrt{8\ln\norm*{\nu}_{\infty}} + \sqrt{32\ln\frac{1}{p}}}\cdot\sqrt{\ln\frac{1}{p}}\cdot t\\
        &= C's
    \end{align*}
    for some other constant $C'$.  Then by \pref{lem:two-sets-concentrate}, when $s\in H$:
    \begin{align*}
        \Pr\bracks*{|X(z)-p| > ps} &\le 2\exp\parens*{-\frac{ds^2}{4C'^2\parens*{\sqrt{8\ln\norm*{\nu}_{\infty}} + \sqrt{32\ln\frac{1}{p}} }^2 \ln\frac{1}{p} }} \\
        &\le 2\exp\parens*{-\frac{ds^2}{\Cconcalt\parens*{\sqrt{\ln\norm*{\nu}_{\infty}} + \sqrt{\ln\frac{d}{p}} }^2 \cdot \log\frac{1}{p} \cdot \log\frac{d}{p}  }}
    \end{align*}
    where the second inequality arises from choosing $\Cconcalt$ to be large enough, which completes the proof.
\end{proof}

\pref{cor:two-sets-subexp} can be extended to the case when $\bz\sim\mu$ with a worse quantitative upper bound depending on $\mu$ from \pref{obs:bounded-dists}.
\begin{corollary}   \label{cor:two-sets-subexp-mu}
    Let $\nu$ and $\mu$ be distributions on $\bbS^{d-1}$ and for $z\in\bbS^{d-1}$, let $X(z)\coloneqq \Pr_{\bx\sim\nu}[\langle \bx, z \rangle > \tau(p)]$.  Then for $s\le 1$
    we have:
    \[
        \Pr_{\bz \sim \mu}\bracks*{|X(\bz)-p|>ps} \le 2\exp\parens*{-\frac{ds^2}{\Cconcalt\parens*{\sqrt{\ln\norm*{\nu}_{\infty}} + \sqrt{\ln\frac{d}{p}} }^2 \cdot \log\frac{1}{p} \cdot \log\frac{d}{p}  }}\cdot \|\mu\|_{\infty}
    \]
\end{corollary}

\section{Concentration for intersections of caps and anti-caps} \label{sec:martingale}

In this section, we prove concentration of measure for the intersection of random $p$-caps and $p$-anticaps with any fixed set $L\subseteq\bbS^{d-1}$.  

\begin{lemma}[Concentration for the intersection of $j$ caps and $k-j$ anti-caps]	\label{lem:martingale-intersect}
Let $\rho$ be the uniform measure over $\bbS^{d-1}$, and let $L \subset \bbS^{d-1}$, and let $k > 0$ be an integer.
For $(\bv_1,\dots,\bv_k)\sim\Unif^{\otimes k}$, let
\[
	\bS_i \coloneqq
	\begin{cases}
		\scap(\bv_i) &\text{if }i\le j \\
		\anticap(\bv_i) &\text{if }i>j,
	\end{cases}
\]
and let $\bL_t = L \cap \bigcap_{i=1}^{t}\bS_i$.  
	Then the ratio $\bR\coloneqq \frac{\Unif(\bL_k)}{\Unif(L)p^j(1-p)^{k-j}}$ is concentrated as follows:
	\[
		\Pr\left[|\bR-1|>s\right] \le \eps_1(s) + k\eps_1(.5) + \eps_2
	\]
	where
	\begin{align*}
		\eps_1(s) \coloneqq 2\exp\left(-\frac{ds^2}{C'(j+(k-j)p^2)F(j)}\right) \\
		\eps_2 \coloneqq \frac{4k}{p^2} \exp\left(-\frac{d}{CF(j)}\right)
	\end{align*}
	$F(j) \coloneqq \left(\sqrt{\ln\frac{1}{\Unif(L)} + j\ln\frac{1}{p} + (k-j)\ln\frac{1}{1-p}} + \sqrt{\ln\frac{d}{p}} \right)^2\log\frac{1}{p}\log\frac{d}{p}$, and $C,C' > 0$ are universal constants.
\end{lemma}

A caricature of the proof of the above lemma is the following.
The main takeaway from \pref{sec:transport-result} is that the intersection of a random $p$-cap with a set $L$ has measure $(1\pm\eps)p\Unif(L)$ where the fluctuation $\eps$ is centered and has typical magnitude $\wt{O}\parens*{\sqrt{\log\frac{1}{\Unif(L)}/d}}$.
As an upshot the intersection of a random $p$-anti-cap with $L$ has measure $\Unif(L) - (1\pm\eps)p\Unif(L) = \parens*{1\pm\frac{p\eps}{1-p}}(1-p)\Unif(L)$.
Thus, when intersecting $j$ random caps and $(k-j)$ random anticaps, we expect a multiplicative fluctuation that resembles $\left(1\pm\left(\sqrt{j}\eps + \sqrt{k-j}\cdot\frac{p}{1-p}\cdot\eps\right)\right)$.
Since the fluctuation term per anti-cap is roughly $p\eps$ instead of merely $\eps$, we get extremely strong concentration for intersections of anti-caps, as articulated by the below corollary.
For example, in the $p = \frac{\alpha}{n}$ regime, an intersection of $m = \Theta(n)$ anti-caps will have measure $(1-p)^m(1\pm \eps)$ with typical deviations $\eps = o(\sqrt{1/n})$ if $d > \polylog n$ for a sufficiently large power of $\log n$.
(See also \pref{cor:caps-and-anti} for another application of \pref{lem:martingale-intersect}).
\begin{corollary}[Intersection of anti-caps]\label{cor:anti-caps}
	Suppose $m,n,d \in \Z_+$, $a \ge 1$, and $p \in \R_+$ satisfy $m \le n$, $d \le n^{100}$, $\frac{1}{n^2} \ll p \le \frac{1}{2}$, and $mp \le \log^a n$ for $n$ sufficiently large.
	Let $L \subset \bbS^{d-1}$ with $\rho(L) \ge e^{-\log^a n}$.
	Let $\bv_1,\ldots,\bv_m \sim \bbS^{d-1}$ uniformly at random, and let $\bA = \{w \in \bbS^{d-1} \mid \iprod{w, \bv_i} \le \tau(p) \, \, \forall i \in [m] \}$ be the intersection of anti-caps of the $\bv_i$'s.
	Then there exist universal constants $C,C'$ such that for all $n$ sufficiently large, for all $t \ge 0$,
	\[
		\Pr_{\bA}\left[\left|\frac{\rho(\bA\cap L)}{(1-p)^{m}\rho(L)} -1\right| > t \cdot\sqrt{mp^2}\right] \le 
		\exp\left(- C \min\left\{\frac{t^2\cdot d}{(\ln n)^{a + 2}},\, \frac{\sqrt{d}-C'(\ln n)^{a + 3}}{(\ln n)^{a + 2}}\right\}\right).
	\]
\end{corollary}
\begin{proof}
Using that $\ln(1 + x) \le x$ for $x > 0$,
\[
\frac{1}{\log \frac{1}{p} \log \frac{d}{p}} \sqrt{F(j)} 
= \sqrt{\ln \tfrac{1}{\rho(L)} + m \ln(1 + \tfrac{p}{1-p})} + \sqrt{\ln \tfrac{d}{p}}
\le \sqrt{\ln \tfrac{1}{\rho(L)}} +\sqrt{m \frac{p}{1-p}} + \sqrt{\ln \tfrac{d}{p}} \le c\cdot \ln^{a} n,
\] 
for $c>0$ a universal constant, where in the final inequality we have applied the assumption $p \le \frac{1}{2}$ and the assumption $p \gg n^{-2}$.
Re-arranging, we have $F(j) \le c'(\ln n)^{a+2}$ for some constant $c'$.
Hence, \pref{lem:martingale-intersect} implies that there exist constants $c'',c'''$ so that for $n$ sufficiently large,
	\begin{align*}
		\Pr_{\bA}\left[\left|\frac{\rho(\bA\cap L)}{(1-p)^{m}\rho(L)} -1\right| > s\right] 
&\le 2\exp\left(-c''\frac{ds^2}{mp^2 (\ln n)^{a+2}}\right) + \left(2k + \frac{4m}{p^2}\right)\cdot \exp\left(-c''' \frac{d}{(\ln n)^{a+2}}\right),
	\end{align*}
where we used our bounds on $p,k,m$ to combine the $\eps_2$ and $k\eps_1(\frac{1}{2})$ terms from \pref{lem:martingale-intersect} into our second term.
The conclusion now follows by substituting $s = t \sqrt{mp^2}$ and applying asymptotic simplifications.
\end{proof}

Our proof proceeds by a martingale argument.  First, we notice that the (rescaled) area of the intersection of sets is a martingale.

\begin{observation}[Scaled intersection is a Martingale]\label{obs:mart}
	Define $\bR_t \coloneqq \frac{\Unif(\bL_t)}{\Unif(L)\prod_{i=1}^t \Unif(\bS_i)}$, and $\bL_t = L \cap \bigcap_{i=1}^{t}\bS_i$, as introduced in \pref{lem:martingale-intersect}.
	Then $(\bR_t)_{t \in [k]}$ is a martingale with respect to the filtration $(\calV_t)_{t\in[k]}$ induced by $\bv_1,\bv_2,\ldots$, with $\E[\bR_t \mid \calV_{t-1}] = \bR_{t-1}$ and $|\bR_t| \le \frac{1}{\rho(L)p^j(1-p)^{k-j}}$.
\end{observation}
\begin{proof}
The quantities $\Unif(\bS_i)$ are fixed for all $i$.
To see that $\bR_t \ge 0$ is bounded, note that $\bR_t \le \frac{1}{p^j(1-p)^{k-j}}$, as $\Unif(\bL_t) \le \Unif(L)$.
Now, note that by definition $\bR_t = \frac{\Unif(\bS_t \cap \bL_{t-1})}{\rho(L)\cdot \prod_{i=1}^t \Unif(\bS_i)}$.
Since $\bv_t \sim \bbS^{d-1}$ independently of $\calV_{t-1}$, $\E_{\bv_t \sim \bbS^{d-1}}[\Unif(\bS_t \cap \bL_{t-1})] = \Unif(\bS_t) \cdot \Unif(\bL_{t-1})$.
The conclusion now follows, since $\E_{\bv_t \sim \bbS^{d-1}}[\bR_t] = \bR_{t-1}$.
\end{proof}

Now, we will prove concentration for the martingale introduced in \pref{obs:mart}, using the concentration inequality for martingales with subgaussian increments, \pref{lem:martingale-azuma}.
The $(\bR_t)_{t\in[k]}$ do not quite have subgaussian increments as described.
Thus we must ``tame'' them by making some minor technical modifications.
\begin{proof}[Proof of \pref{lem:martingale-intersect}]
	Recall the martingale sequence $\bR_t$ defined in \pref{obs:mart}.  We are then interested in deviation bounds for $\abs*{\bR_k-1}$.
	Our proof strategy is to couple $\bR_t$ with a more well-behaved martingale sequence and use the subgaussian martingale concentration inequality on the more well-behaved sequence.  Let $\bT$ be the first time at which $\bR_{\bT} > 2$ or $\bR_{\bT} < \frac{1}{2}$.
	Define the process $(\bQ_t)_{t \ge 1}$ so that 
	\[
		\bQ_i =
			\begin{cases}
				\bR_i & i \le \bT \\
				\bR_{\bT} & \text{otherwise}.
			\end{cases}
	\]
	Note that the sequence $(\bQ_t)_{t \ge 1}$ is a martingale.  This is because when $t \le \bT$, $\bQ_t-\bQ_{t-1}\mid\calV_{t-1}$ has the same distribution as $\bR_t-\bR_{t-1}\mid\calV_{t-1}$ and hence $\E[\bQ_t - \bQ_{t-1}\mid\calV_{t-1}] = 0$, and when $t\ge \bT$, $\bQ_t-\bQ_{t-1}\mid\calV_{t-1}$ is identically zero.  
We will need to make an additional modification on top of $\bQ_i$ to make the random variable well-behaved, since \pref{cor:two-sets-subexp} only guarantees that $\Pr\left[|\bQ_i-\bQ_{i-1}|>s\right]$ remains subgaussian up to $s = 1$.
We will truncate $\bQ_i$ as described by the following definition:
\begin{definition}[{$(\alpha,\beta)$-truncation}]
	Given a centered random variable $\bX$, we define a random variable $\bX_{\alpha,\beta}$ for $\alpha,\beta\in\R_{\ge 0}$.  First, let $\theta = \min \left(\frac{\Pr[|\bX|>\alpha]}{\beta},1\right)$.  Now, define
	\[
		\bX_{\alpha,\beta} \coloneqq
		\begin{cases}
			\bX \mid |\bX| \le \alpha &\text{with probability $(1-\theta)\Pr[|\bX| \le \alpha]$} \\
			\E\left[\bX \mid |\bX|>\alpha \right]\cdot \beta &\text{with probability $(1-\theta)\cdot \theta$} \\
			0 &\text{otherwise.}
		\end{cases}
	\]
$\bX_{\alpha,\beta}$ is well-defined; in the case when $\theta = 1$ it takes value $0$ deterministically, and otherwise if $\theta < 1$ the probabilities sum to 1.
\end{definition}

\noindent We make two useful observations about $(\alpha,\beta)$-truncations.
\begin{observation}	\label{obs:exp-value-trunc}
	For a random variable $\bX$ and parameters $\alpha,\beta$,
		$\E \bX_{\alpha,\beta} = (1-\theta) \E \bX$.
\end{observation}

\begin{observation}	\label{obs:TV-dist-trunc}
	For a random variable $\bX$ and parameters $\alpha,\beta$ satisfying $\beta \le 1$,
		$\dtv{\bX}{\bX_{\alpha,\beta}} \le 2\theta$.
\end{observation}
\noindent We prove these observations at the end of the section; for now, we proceed with the proof of \pref{lem:martingale-intersect}.

Define $\bDelta_i\coloneqq \bQ_i - \bQ_{i-1}$, and define $\wt{\bDelta}_i \mid \calV_{i-1} \coloneqq \left(\bDelta_i \mid \calV_{i-1} \right)_{\alpha,\beta}$ where
	\[
		\alpha =
		\begin{cases}
			1 & \text{if $i \le j$} \\
			\frac{p}{1-p} & \text{if $i > j$.}
		\end{cases}
	\]
	and $\beta = p^2\alpha$. 
	We have chosen $\alpha$ such that \pref{cor:two-sets-subexp} guarantees that we can control the tail probability $\Pr[|\bDelta_i| > t]$ for $t < \alpha$.
	Since $|\bDelta_i| \le \frac{1}{p}$, $\abs*{\E\bracks*{\bDelta_i \mid \abs*{\bDelta_i}>\alpha}}\cdot\beta \le p\alpha$, and consequently $|\wt{\bDelta}_i|\le\alpha$. 
 	Let $(\wt{\bQ}_t)_{t\ge 1}$ be the random process obtained by setting $\wt{\bQ}_1=\bQ_1$ and $\wt{\bQ}_t = \wt{\bQ}_{t-1}+\wt{\bDelta}_{t}$ for $t\ge 2$.  
 	By \pref{obs:exp-value-trunc}, $(\wt{\bQ}_t)_{t\ge 1}$ is a martingale.

	Next, we obtain bounds on the subgaussian norms of $\wt{\bDelta}_i|\calV_{i-1}$.  For $\calV_{i-1}$ such that $i > \bT$, $\left\|\wt{\bDelta}_i|\calV_{i-1}\right\|_{\subgnorm} = 0$.  For any $\calV_{i-1}$ where $i\le\bT$ and $s\ge p\alpha$:
	\begin{align*}
		\Pr\left[\left|\wt{\bDelta}_i\right| > s \mid \calV_{i-1}\right] &\le \Pr[|\bDelta_i| > s \mid \calV_{i-1}] \Ind[s\le\alpha] \\
		&= \Pr[|\bR_i - \bR_{i-1}| > s \mid \calV_{i-1}] \Ind[s\le\alpha] \\
		&= \Pr\left[\bR_{i-1}\cdot\left|\frac{\Unif(\bS_i\cap\bL_{i-1})}{\Unif(\bS_i)\cdot\Unif(\bL_{i-1})}-1\right| > s \mid \calV_{i-1}\right] \Ind[s\le \alpha].	\numberthis \label{eq:bound-truncated-prob}
	\end{align*}
	Since $i \le \bT$, we know $\bR_{i-1}\le 2$ and $\Unif(\bL_{i-1})\ge\frac{\Unif(L)\prod_{t=1}^i \Unif(\bS_t)}{2}\ge \frac{\Unif(L)p^j(1-p)^{k-j}}{2}$.
	When $i\le j$:
	\[
		\Pr\left[\left|\wt{\bDelta}_i\right| > s \mid \calV_{i-1}\right] \le 2 \exp\left(-\frac{ds^2}{C\cdot\left(\sqrt{\ln\frac{1}{\Unif(L)} + j\ln\frac{1}{p} + (k-j)\ln\frac{1}{1-p}} + \sqrt{\ln\frac{d}{p}} \right)^2\log\frac{1}{p}\log\frac{d}{p}}\right)
	\]
	and when $i > j$:
	\[
		\Pr\left[\left|\wt{\bDelta}_i\right| > s \mid \calV_{i-1}\right] \le 2 \exp\left(-\frac{ds^2}{C\cdot p^2\cdot\left(\sqrt{\ln\frac{1}{\Unif(L)} + j\ln\frac{1}{p} + (k-j)\ln\frac{1}{1-p}} + \sqrt{\ln\frac{d}{p}} \right)^2\log\frac{1}{p}\log\frac{d}{p}}\right)
	\]
	for all $s > 0$ and for some constant $C>0$, where the case of $s \le p\alpha$ holds vacuously since the right hand side exceeds $1$, and the case of $s > p\alpha$ holds by \pref{eq:bound-truncated-prob} and \pref{cor:two-sets-subexp}.  Using $F(j)$ to denote
	\[
		\left(\sqrt{\ln\frac{1}{\Unif(L)} + j\ln\frac{1}{p} + (k-j)\ln\frac{1}{1-p}} + \sqrt{\ln\frac{d}{p}} \right)^2\log\frac{1}{p}\log\frac{d}{p},
	\]
	we can more legibly write the above as:
	\[
		\Pr\bracks*{\abs*{\wt{\bDelta}_i} > s | \calV_{i-1}} \le 2\exp\left(-\frac{ds^2}{C p^{2 \cdot \Ind[i>j]}\cdot F(j)}\right)
	\]
	Consequently, $\left\|\wt{\bDelta}_i | \calV_{i-1}\right\|_{\subgnorm} \le C p^{\Ind[i>j]}\cdot \sqrt{\frac{F(j)}{d}}$ for all $\calV_{i-1}$ by \pref{lem:subgaussian-equiv}.

	By the subgaussian martingale concentration inequality (\pref{lem:martingale-azuma}), for some constant $C'>0$:
	\[
		\Pr\bracks*{\abs*{\wt{\bQ}_k-1}\ge s} \le \eps_1(s) \coloneqq 2\exp\left(-\frac{ds^2}{C'(j+(k-j)p^2)F(j)}\right).
	\]
	By a union bound we also know:
	\[
		\Pr\bracks*{\exists i: \abs*{\wt{\bQ}_i-1}>s} \le k \eps_1(s).
	\]
	It remains to relate $\wt{\bQ}_k$ back to $\bR_k$.
	By \pref{obs:TV-dist-trunc},  and a union bound,
 there is a coupling between the sequence of $\wt{\bDelta}_i$ and the sequence of $\bDelta_i$ such that $\wt{\bDelta}_i = \bDelta_i$ for all $i$ except with probability at most $2k \max_{i \in [k]} \frac{\Pr[|\bDelta_i|>\alpha]}{p\alpha}$, which with \pref{cor:two-sets-subexp} we can bound by
	\[
		\eps_2 \coloneqq \frac{4k}{p^2} \exp\left(-\frac{d}{CF(j)}\right) 
	\]
	The upshot of the above is:
	\[
		\Pr\bracks*{\exists i \in [k]: \abs*{\bQ_i-1} > s\text{ or }\bQ_i\ne\wt{\bQ}_i} = \Pr\bracks*{\exists i \in [k]: \abs*{\wt{\bQ}_i-1} > s\text{ or }\bQ_i\ne\wt{\bQ}_i} \le k\eps_1(s)+\eps_2.
	\]
	Suppose $\abs*{\bQ_i-1}<.5$ for $i=1,\dots,k$, then $\bQ_i = \bR_i$ for $i=1,\dots,k$.  Thus:
	\begin{align*}
		\Pr\bracks*{\exists i\in[k]: \wt{\bQ}_i\ne\bR_i} &\le \Pr\bracks*{\exists i\in[k]: \wt{\bQ}_i\ne\bR_i\text{ or }\bQ_i\ne\wt{\bQ}_i} \\
		&= \Pr\bracks*{\exists i\in[k]: \bQ_i\ne\bR_i\text{ or }\bQ_i\ne\wt{\bQ}_i} \\
		&\le k\eps_1(.5)+\eps_2.
	\end{align*}
	Consequently,
	\begin{align*}
		\Pr\bracks*{\abs*{\bR_k-1} > s} &\le \Pr\bracks*{\abs*{\wt{\bQ}_k-1}>s} + \Pr\bracks*{\wt{\bQ}_k \ne \bR_k} \\
		&\le \eps_1(s) + k\eps_1(.5) + \eps_2.\qedhere
	\end{align*}
\end{proof}
We now give the deferred proofs of our observations regarding truncated variables.

\begin{proof}[Proof of \pref{obs:exp-value-trunc}]
If $\theta = 1$, $\bX_{\alpha,\beta} = 0$ and the statement holds.
Otherwise (and actually in any case), we can write:
	\begin{align*}
		\E \bX_{\alpha,\beta} &= (1-\theta) \Pr[|\bX|\le\alpha]\E[\bX\mid |\bX| \le \alpha] + (1-\theta) \cdot \frac{\Pr[|\bX| > \alpha]}{\beta} \cdot \E\left[\bX \mid |\bX|>\alpha \right]\cdot \beta  \\
		&= (1-\theta) \E [\bX \cdot \Ind[|\bX|\le \alpha]] + (1-\theta) \E [\bX \cdot \Ind[|\bX|>\alpha]] \quad = (1-\theta) \E \bX. \qedhere
	\end{align*}
\end{proof}

\begin{proof}[Proof of \pref{obs:TV-dist-trunc}]
If $\theta = 1$, the bound holds trivially, so assume that $\theta < 1$.
	We can couple $\bX$ and $\bX_{\alpha,\beta}$ by sampling $\bX_{\alpha,\beta}$ in the following way: sample $\bc\sim\Ber(1-\theta)$ and $\bX$.  If $\bc = 1$ and $|\bX|\le\alpha$, let $\bX_{\alpha,\beta} = \bX$; otherwise let $\bX_{\alpha,\beta} = \E\left[\bX \mid |\bX|>\alpha \right]\cdot \beta$ with probability $(1-\theta)\theta$ and $0$ with the remaining probability. 
If $\beta \le 1$, then by definition $\theta \ge \beta \theta = \Pr\left[|X| > \alpha \right]$.
 Hence under this coupling, $\bX\ne\bX_{\alpha,\beta}$ with probability at most $\theta + \beta \theta \le 2\theta$, yielding the total variation bound.
\end{proof}

\section{Stochastic domination by \erdos-\renyi graphs}\label{sec:stoch-dom}

Here, we will prove \pref{prop:dom}: we'll show that so long as $d = \tilde{\Omega}((np)^2)$, there is a three-way ``coupling'' between $\bG \sim \grg_d(n,p)$, $\bG_- \sim \ER(n,(1- \tilde{O}(\sqrt{\frac{np}{d}}))p)$, and $\bG_+ \sim \ER(n,(1+ \tilde{O}(\sqrt{\frac{np}{d}}))p)$ such that with high probability $\bG$ is sandwiched between $\bG_-$ and $\bG_+$.
This will be crucial for the proof of \pref{thm:sparse} in reasoning about the structure of random geometric graphs.

We will need the following corollary of \pref{lem:martingale-intersect}, which lower bounds the area of intersections of $j$ random $p$-caps and $k-j$ random anti-$p$-caps.
\begin{corollary}[Area of $\approx pk$ caps and $\approx (1-p)k$ anti-caps]\label{cor:caps-and-anti}
Let $n,j,k,d \in \bbZ_+$, and let $p \in \R_+$, satisfying $\frac{1}{n^2} \ll p \le \frac{1}{2}$, $1 \le d \le n^{100}$, $j \le k \le n$, with $j = pk + \Delta$. 
Draw $\bv_1,\ldots,\bv_k$ uniformly from $\bbS^{d-1}$ and let $\bL = \left( \bigcap_{i=1}^{j} \scap(\bv_i) \right) \cap \left( \bigcap_{i=j+1}^k {\anticap(\bv_i)} \right)$.
Then there exist constants $\Ccaa,\Ccaa' > 0$ such that for all $n$ sufficiently large,
\[
\Pr\left[\left|\frac{\rho(\bL)}{e^{-kH(p)} \left(\frac{p}{1-p}\right)^{\Delta}} - 1\right|\ge t\right] 
\le 
\exp\left(- \Ccaa \cdot\min\left(\frac{dt^2}{M(k,p,\Delta)^2\ln n},\frac{d }{M(k,p,\Delta)^2 \ln n}-\Ccaa' \log n\right)\right),
\]
where $M(k,p,\Delta) = \max(k H(p), |\Delta|\ln \frac{1}{p}, \ln n)$.
\end{corollary}
\begin{proof}
We will apply \pref{lem:martingale-intersect}.
Applying asymptotic simplifications, in our case, 
\[
\frac{\sqrt{F(j)}}{\ln \frac{1}{p} \ln \frac{d}{p}} 
= \sqrt{k H(p) + \Delta \ln \tfrac{1-p}{p}} +\sqrt{\ln \tfrac{d}{p}} 
\le c \cdot \left(\sqrt{kH(p)+ |\Delta|\ln\tfrac{1}{p}} + \sqrt{\ln n}\right),
\] 
for $c > 0$ a universal constant.
Re-arranging and applying the inequality $a^2 + b^2 \ge \frac{1}{2} ab$,
\[
F(j) \le c' \cdot \ln\frac{1}{p}\ln n \left(kH(p) + |\Delta|\ln \tfrac{1}{p} + \ln n\right).
\]
Now applying \pref{lem:martingale-intersect} and letting $B = kH(p) + |\Delta|\ln \frac{1}{p}$ (and using $j + (k-j)p^2 = kp + (1-p^2)\Delta + k(1-p)p^2 \le 2kp + |\Delta|$, and hence $F(j)(j+(k-j)p^2) \le c'' (B+\ln n)B\ln n$ for some constant $c''$),
\begin{align*}
\Pr\left[\left|\frac{\rho(\bL)}{e^{-kH(p)} \left(\frac{p}{1-p}\right)^{\Delta}} - 1\right|\ge t\right] 
&\le 
2\exp\left(-c_1\frac{ds^2}{(B+\ln n) B\ln n}\right) + 2k\exp\left(-c_2\frac{d}{(B+\ln n) B\ln n}\right) \\
&\qquad + \frac{4k}{p^2}\exp\left(-c_3\frac{d}{(B+\ln n)\ln \frac{1}{p} \ln n}\right)\\
&\le 
2\exp\left(-c_1\frac{ds^2}{M^2\ln n}\right) + \left(2k+ \frac{4k}{p^2}\right)\exp\left(-c_3\frac{d}{M^2\ln n}\right),
\end{align*}
for universal constants $c_1,c_2,c_3 > 0$ and $M = \max(kH(p),|\Delta|\ln \frac{1}{p},\ln n)$.
Applying asymptotic simplifications yields our desired result.
\end{proof}

\restateprop{prop:dom}
\begin{proof}[Proof of \pref{prop:dom}]
We describe how to sample $\bG_-,\bG,\bG_+$ jointly so that the above holds.
We consider vertices arriving one at a time, and we'll determine the adjacency of the arriving vertex to all previous vertices in a way that correlates $\bG_-,\bG,\bG_+$.
In $\bG$, each vertex is identified with a vector; an arriving vertex's vector starts with the entirety of $\bbS^{d-1}$, then iteratively refines the set of candidates by intersecting with the spherical caps (or anti-caps) of the previous vectors.

\begin{enumerate}
\item Choose $\bv_1 \sim \bbS^{d-1}$ uniformly at random.
\item For each $i \in \{2,\ldots,n\}$:
\begin{enumerate}
\item Initialize the set of candidates for vector $\bv_i$ as $A_{i,0} = \bbS^{d-1}$.
\item For each $k \in \{1,\ldots,i-1\}$:
\begin{enumerate}
\item Choose a threshold value $\btheta \sim \mathrm{Unif}([0,1])$.
\item If $\btheta < (1-\eps) p$, add the edge $(i,k)$ to $\bG_-$
\item Let $\bS_k = \{u \in \bbS^{d-1} \mid \Iprod{u,\bv_k} \ge \tau(p)\}$. If $\btheta < \frac{\rho(\bA_{i,k-1} \cap \bS_k)}{\rho(\bA_{i,k-1})}$, let $\bA_{i,k} = \bA_{i,k-1} \cap \bS_k$. Otherwise let $\bA_{i,k} = \bA_{i,k-1}\cap \ol{\bS_k}$. 
\item If $\btheta < (1+\eps) p$, add the edge $(i,k)$ to $\bG_+$.
\end{enumerate}
\item Choose $\bv_i$ uniformly at random from $\bA_{i,i-1}$.
\end{enumerate}
\item Set $\bG = \GG_{\tau(p)}(\bv_1,\ldots,\bv_n)$.
\end{enumerate}
The marginal distributions over $\bG_-$ and $\bG_+$ are clearly $\ER(n,(1-\eps)p)$ and $\ER(n,(1+\eps)p)$ respectively.
The marginal distribution over $\bG$ is $\grg_d(n,p)$: since we choose $\bA_{i,k} = \bA_{i,k-1} \cap \bS_k$ independently with probability proportional to the measure $\rho(\bA_{i,k-1} \cap \bS_k)$, each $\bv_i \sim \bbS^{d-1}$ independently of $\bv_1,\ldots,\bv_{i-1}$.

Note that for $i > k$, the edge $(i,k)$ will be present in $\bG$ if and only if $\bA_{i,k} = \bA_{i,k-1} \cap \bS_k$.
Hence, the sampled graphs will satisfy $\bG_- \subseteq \bG \subseteq \bG_+$ if and only if for every $i>k$ it is the case that 
\begin{equation}
(1-\eps)p \le \frac{\rho(\bA_{i,k-1} \cap \bS_k)}{\rho(\bA_{i,k-1})} \le (1+\eps)p.\label{eq:nest}
\end{equation}
We now argue that this occurs with high probability.
Each $\bA_{i,k-1}$ is an intersection of $\bj = \Binom(p,k-1)$ caps and $k-1-\bj$ anti-caps.
Using Bernstein's inequality, with probability $1 - n^{-\log n}$, $\bj = kp + \bDelta$ for $|\bDelta| \le C \cdot \ln n \max(\sqrt{kp},\ln n)$ for $C$ a universal constant.
Hence, applying \pref{cor:caps-and-anti} with $t = \frac{1}{2}$ and a union bound, we have that so long as $d \gg \max(kH(p),\ln^3 n)^2 \log^2 n = O((np + \ln^2 n)^2 \ln^4 n)$,  with probability at least $1 - \binom{n}{2}\cdot n^{-\Omega(\log n)}$, 
$\rho(\bA_{i,k}) \ge \frac{1}{2}\exp^{-C' \cdot kH(p) - \ln^2 n}$ for all $1 \le k < i \le n$, 
and hence $\log(\|\Unif_{\bA_{i,j}}\|_\infty) \le C' \cdot \left(n \cdot H(p) + \ln^2 n\right)$.

Using this, we may apply \pref{lem:two-sets-concentrate} with $\kappa = 4\ln \frac{d}{p}$ and $t = 4\ln n\sqrt{\frac{1}{d}} $, in conjunction with a union bound over all $1 \le j < i \le n$ to conclude that the probability that \pref{eq:nest} holds for all such $i > j$ is at least $1-n^{-\Omega(\log n)}$ for any $\eps \ge C'' \sqrt{\frac{1}{d}(n p + \ln n)\ln^4 n}$, where $C''$ is a universal constant.
This completes the proof.
\end{proof}

\section{Distribution of a neighborhood} \label{sec:bp}
In this section, we analyze the probability distribution of the neighborhood of a vertex in a random geometric graph conditioned on the remaining graph.
Throughout this section we will assume that $\log^{36} n \le d \le n^{100}$ and $p = \frac{\alpha}{n}$ for constant $\alpha\ge 1$.
\footnote{The upper bound of $n^{100}$ on $d$ isn't actually necessary for our proof techniques, but this assumption slightly simplifies calculations.  Since indistinguishability results are already known when $d\ge n^{100}$ from prior work, we prioritize having somewhat less symbol-dense calculations.}

Concretely, let $\bG\sim\grg_d(n,p)$ and let $\bG_{n-1}$ be the induced subgraph on $[n-1]$.  
We analyze the distribution of $\randnbr(n)|\bG_{n-1}$, and in particular prove that it closely tracks the neighborhood distribution in an \erdos-\renyi graph.
\begin{lemma}   \torestate{\label{lem:main-nbrhood-lemma}
    With probability $1-n^{-\Omega(\log n)}$ over the randomness of $\bG_{n-1}$, for every $S\subseteq[n-1]$ such that $|S|\le\log^2 n$, the following is true.  Let $\ell<\frac{\log n}{\log\log n}$ be such that the balls $\{B_{\bG_{n-1}}(x,\ell):x\in S\}$ are pairwise disjoint trees, and let $d > \log^{20}n$.  Then:
    \[
       \Pr\bracks*{\randnbr(n) = S | \bG_{n-1}} \in \left(1\pm2\eta(\ell)\right) p^{|S|}(1-p)^{n-1-|S|}
    \]
    where
    \[
       \eta(\ell) \coloneqq \min\braces*{\frac{\log^8 n}{\sqrt{d}}, \max\braces*{ 2\left(\sqrt{\frac{\log^{28}n}{d}}\right)^{\ell}, 8 \cdot \sqrt{\frac{\log^{11} n}{ nd } }} }.
    \]}
\end{lemma}

We prove \pref{lem:main-nbrhood-lemma} in \pref{sec:containment-to-exact}, after establishing the necessary ingredients.

\subsection{Setup and outline}
The distribution of $\randnbr(n)$ is given by sampling $\bw_{n}\sim\Unif$, then sampling $(\bw_1,\dots,\bw_{n-1})\sim\VecDist{\bG_{n-1}}$, and choosing $\randnbr(n) = \{x:\angles*{\bw_x, \bw_{n}} \ge \tau(p\}$.  
For every $S\subseteq[n-1]$ satisfying $|S|\le\log^2 n$, we are interested in $\Pr[\randnbr(n)=S|\bG_{n-1}]$, which is the expected measure of the intersection of some collection of caps and anticaps.  
In particular:
\[
    \Pr\bracks*{\randnbr(n)=S|\bG_{n-1}} = \E_{(\bw_i)_{i\in[n]}\sim\VecDist{\bG_{n-1}}} \Unif\parens*{\bigcap_{j\in S}\scap(\bw_j)\cap\bigcap_{j\in[n-1]\setminus S} \anticap(\bw_j)}.
\]
The random variable $\Unif\parens*{\bigcap_{j\in S}\scap(\bw_j)\cap\bigcap_{j\in[n-1]\setminus S} \anticap(\bw_j)}$ may appear daunting at first due to the complicated correlation structure of $(\bw_i)_{i\in[n-1]}$. 
Our situation is greatly simplified by tight concentration for the measure of the intersection of random anticaps with sets of lower bounded measure.  
In particular, we show in \pref{sec:containment-to-exact} that with high probability:
\[
    \Unif\parens*{\bigcap_{j\in S}\scap(\bw_j)\cap\bigcap_{j\in[n-1]\setminus S} \anticap(\bw_j)} \in \parens*{1\pm\wt{O}\parens*{\sqrt{\frac{1}{nd}}}} \cdot \Unif\parens*{\bigcap_{j\in S}\scap(\bw_j)}(1-p)^{n-1-|S|}.
\]
This lets us write $\Pr\bracks*{\randnbr(n)=S|\bG_{n-1}}$ as
\begin{align*}
    \Pr\bracks*{\randnbr(n)=S|\bG_{n-1}} \in \parens*{1\pm\wt{O}\parens*{\sqrt{\frac{1}{nd}}}} (1-p)^{n-1-|S|} \cdot \E\Unif\parens*{\bigcap_{j\in S}\scap(\bw_j)}.
\end{align*}
Now,
\[
    \E \Unif\parens*{\bigcap_{j\in S}\scap(\bw_j)} = \Pr[\randnbr(n)\supseteq S|\bG_{n-1}]
\]
and so to study the probability that $\randnbr(n)=S$, it suffices to study the probability that $\randnbr(n)\supseteq S$, which can equivalently be written as:
\[
    \Pr[\randnbr(n)\supseteq S|\bG_{n-1}] = \Pr \bracks*{\forall x\in S: \angles*{\bw_x,\bw_{n}}\ge\tau(p)}.
\]
We next notice that we are spared from working with the potentially complicated full ensemble $(\bw_i)_{i\in[n-1]}$, since the above depends only on its marginal distribution on $\bw_i$ for $i \in S$. 
Indeed, we prove that for most sets $S$, the vectors $(\bw_i)_{i\in S}$ roughly behave like independent and uniform vectors on the unit sphere, a statement which is made concrete below.

\begin{lemma}   \torestate{\label{lem:nbrhood-containment}
    With probability $1-n^{-\Omega(\log n)}$, for every
    $S\subseteq[n-1]$ such that $|S|\le\log^2 n$, the following is true.  Suppose $\ell<\frac{\log n}{\log\log n}$ is such that the balls $\{B_{\bG_{n-1}}(x,\ell):x\in S\}$ are all trees and are pairwise disjoint, then:
    \[
       \Pr\bracks*{\randnbr(n) \supseteq S | \bG_{n-1}} \in \left(1\pm\eta(\ell)\right) p^{|S|}
    \]
    where
    \[
       \eta(\ell) \coloneqq \min\braces*{\frac{\log^8 n}{\sqrt{d}}, \max\braces*{2\sqrt{\frac{\log^{28}n}{d}}^{\ell}, 8\sqrt{\frac{\log^{11}n}{d}} }}.
    \]}
\end{lemma}
We now give a brief proof overview of \pref{lem:nbrhood-containment}.
The way we think about the quantity we wish to obtain a handle on is articulated by the following equality.
\[
    \Pr\bracks*{N_{\bG}(n)\supseteq S\mid \bG_{n-1}} = \E_{\substack{\bw_n\sim\Unif \\ (\bw_i)_{i\in[n-1]}\sim \VecDist{\bG_{n-1}}}} \prod_{x\in S}\Ind\bracks*{\angles*{\bw_x,\bw_n}\ge\tau(p)}.
\]
If $(\bw_x)_{x\in S}$ was a collection of \textbf{independent} random vectors distributed \textbf{uniformly} on the sphere then the above quantity is exactly equal to $p^{|S|}$.
Our proof of \pref{lem:nbrhood-containment} is by establishing that both of these properties are approximately true.
The following gives a rough outline of the proof along with how it's organized in this section.
\begin{itemize}
    \item Our intuition is that any ``large'' correlations between vertices $x$ and $y$ are explained by the existence of paths in between them.  
    Thus, we study the distribution of $(\bw_x)_{x\in S}|(\bw_x)_{x\notin B(S,\ell)}$ for some $\ell$ where $(\bw_x)_{x\notin B(S,\ell)}$ is sampled according to $\VecDist{\bG_{n-1}}$ and attempt to prove approximate independence for this ensemble.
    \item In \pref{sec:containment-prob}, we define a constraint satisfaction problem instance on a forest factor graph with variable set indexed by the vertices of $B(S,\ell)$, and in \pref{sec:relax-to-true} show that the marginals of the uniform distribution $\calF'$ on satisfying assignments to this CSP instance on the variables indexed by $S$ is very close to the distribution of $(\bw_x)_{x\in S}$.
    \item $(\bw_x)_{x\in S}$ is an independent ensemble when sampled from $\calF'$ for the simple reason that they are all in different connected components of the factor graph.
    The bulk of the technical work in this section is in proving that each $\bw_x$ ``behaves like'' a uniform vector, which is performed in \pref{sec:BP-relaxed}.
    \item The key relevant property of the uniform distribution on the sphere $\Unif$ for us is that for every $w$, $\Pr_{\bv\sim\Unif}\bracks*{\angles*{\bv, w}\ge\tau(p)} = p$.
    Thus, for each vertex $x\in S$, we show the following about the marginal of $\calF'$ on $\bw_x$.
    \begin{displayquote}
        For \emph{most} $w$, $\Pr\bracks*{\angles*{\bw_x,w}\ge\tau(p)} = p(1\pm\eps)$.
    \end{displayquote}
    \item The way we establish this ``pseudo-uniformity'' property of $\bw_x$ is by using the fact that its marginal can be computed exactly from belief propagation fixed point messages by virtue of arising from a forest factor graph.
    Our analysis of the belief propagation fixed point messages is based on diffusion properties of the random walk on the sphere where each transition moves from a point $v$ to a random point in the $p$-cap around $v$.
\end{itemize}

\subsection{A relaxation of containment probabilities} \label{sec:containment-prob}
We now turn our attention towards proving \pref{lem:nbrhood-containment}.

We use $\bW$ to refer to the collection $(\bw_i)_{i\in[n-1]}\sim\VecDist{\bG_{n-1}}$, and for any subset $A\subseteq[n]$, we use the notation $\bW_A$ to refer to the subcollection $(\bw_i)_{i\in A}$.
Let $\ell$ be an integer chosen so that $\{B_{\bG_{n-1}}(x,\ell):x\in S\}$ are pairwise disjoint trees, and $\ell < \frac{\log n}{\log\log n}$.  
Define $\Balls \coloneqq \bigcup_{x\in S} V\parens*{B_{\bG_{n-1}}(x,\ell-1)}$ and denote the remaining vertices as $\Rest\coloneqq [n-1]\setminus \Balls$.
Note that it is possible for $\ell=0$, in which case $\Balls$ is empty.
We can write
\[
    \Pr\bracks*{\randnbr(n) \supseteq S | \bG_{n-1}} = \E_{\bw_{n}\sim\Unif} \bracks*{ \E_{\bW_{\Rest}}  \bracks*{ \Pr_{\bW_{\BallsInt}|\bW_{\Rest}} \bracks*{\forall x\in S: \angles*{\bw_x,\bw_{n}}\ge\tau(p)} } }. \numberthis \label{eq:containment-prob}
\]
The distribution of $\bW_{\BallsInt} | \bW_{\Rest}$ is uniform on the space of vectors producing $\bG_{n-1}$, conditioned on the choice of $\bW_{\Rest}$:
\[
    \calF\parens*{\bW_{\Rest}}\coloneqq \{(w_i)_{i\in\BallsInt}:\GG\parens*{w_1,\ldots,w_{n-1}}=\bG_{n-1}, \, \text{ and } w_r = \bw_r\text{ for }r\in\Rest\}.
\]
We ``relax'' $\calF\parens*{\bW_{\Rest}}$ to a slightly larger set $\calF'\parens*{\bW_{\Rest}}\supseteq\calF\parens*{\bW_{\Rest}}$, which admits a description in terms of the solution space of a constraint satisfaction problem on a forest; in particular, the induced subgraph $T\parens*{\bG_{n-1}}\coloneqq\bG_{n-1}[\BallsInt]$.
The reason for doing so is we can precisely calculate marginals of the uniform distribution on $\calF'\parens*{\bW_{\Rest}}$ using the belief propagation algorithm. 
To help define $\calF'\parens*{\bW}$, we let $L\parens*{\bW_{\Rest}}\coloneqq\bigcap_{j\in \Rest} \anticap\parens*{\bw_j}$, and similarly let $L_{i}\parens*{\bW_{\Rest}} \coloneqq \bigcap_{j\in \Rest\setminus\Leaves(i)}  \anticap\parens*{\bw_j}\cap\bigcap_{j\in\Leaves(i)}\scap\parens*{\bw_j}$ where $\Leaves(i)$ denotes the set of neighbors of $i$ within $\Rest$. 
Observe that for $i \in V$ with distance less than $\ell-1$ to some $j \in S$, the set $L_i = L(\bW_R)$.

\begin{definition}  \label{def:CSP-dist}
    $\calF'\parens*{\bW_{\Rest}}$ is the collection of satisfying assignments to the following CSP instance $\Inst\parens*{\bW_{\Rest}}$ on variables indexed by $V(\BallsInt)$, comprised of unary constraints and binary constraints.
    \begin{itemize}
       \item {\bf Unary constraints.}  For each $i\in V(\BallsInt)$, $v_i\in L_i(\bW_{\Rest})$.
       In words, the set of neighbors of the vector assignment $v_i$ within $\Rest$ must be equal to $N_{\bG_{n-1}}(i)\cap \Rest$.
       \item {\bf Binary constraints.}  For every $\{i,j\}\in T\parens*{\bG_{n-1}}$: $\angles*{v_i,v_j}\ge\tau(p)$.
       In words, for every edge $\{i,j\}$, the vector assignments $v_i$ and $v_j$ must also have an edge in between them.
    \end{itemize}
    We use $F(\bW_R)$ to denote the factor graph representation of $\Inst\parens*{\BallsInt}$.
\end{definition}
\noindent Notice that \pref{def:CSP-dist} enforces that $\GG(\bW_{\BallsInt}\cup\bW_{\Rest})$ {\em contains} $\bG_{n-1}$ as an edge-induced subgraph. 
In \pref{sec:relax-to-true}, we will show that with high probability, $\GG(\bW_k \cup \bW_R)$ is also precisely equal to $\bG_{n-1}$.

Our main result about $\calF'$ is the following.
\begin{lemma}   \torestate{\label{lem:main-BP-lemma}
    Suppose $\ell\ge 1$.
    Then for every $S$ such that $|S|\le \log^2 n$:
    \[
       \Pr_{\substack{\bw_n\sim\bbS^{d-1} \\ \bW_{\Balls}\sim\calF'\parens*{\bW_{\Rest}}}}\bracks*{\forall x\in S:\angles*{\bw_x,\bw_n}\ge\tau(p)} \in \left(1\pm\eta(\ell)\right) p^{|S|}
    \]
    except with probability at most $n^{-\Omega(\log n)}$ over the randomness of $\bG_{n-1}$ where
    \[
       \eta(\ell) \coloneqq \max\braces*{\left(C\sqrt{\frac{\log^{27}n}{d}}\right)^{\ell}, 4\sqrt{\frac{\log^{11} n}{nd}}}
    \]
    for some absolute constant $C$.}%
    \footnote{Recall that \pref{lem:main-BP-lemma} is a stepping stone towards proving \pref{lem:nbrhood-containment}, and the case of $\ell = 0$ is handled directly within the proof of \pref{lem:nbrhood-containment}.}
\end{lemma}
The proof of \pref{lem:main-BP-lemma} is carried out in \pref{sec:BP-relaxed}.

\subsection{Belief propagation for relaxed distribution}    \label{sec:BP-relaxed}

\paragraph{BP update rule.}  Recall that given a CSP instance $\Inst$ with variable set $V$ and constraint set $E$, the marginals on specific variables can be accurately computed from the belief propagation fixed point messages when the factor graph is a forest using \pref{thm:BP-forest-marginals}.
\par In our setting, the BP equations are the following: 
\begin{align}
    \textbf{Variable-to-constraint messages. } & m^{v\to f} = \frac{\prod_{e\in \partial v\setminus f} m^{e\to v}}{\int \prod_{e\in\partial v\setminus f} m^{e\to v}(x) d\Unif(x)} \\
    \textbf{Constraint-to-variable messages (unary case). } & m^{f\to v} = \frac{f}{\int f(x) d\Unif(x)} \\
    \textbf{Constraint-to-variable messages (binary case).} & \text{ Let $\partial f = \{v,w\}$, then:} \nonumber \\
    & m^{f\to v} (x_v) = \frac{\int f(x_v,x_w)\cdot m^{w\to f}(x_w)d\Unif(x_w)}{\int \int f(x_v,x_w)\cdot m^{w\to f}(x_w) d\Unif(x_w)d\Unif(x_v)}.
\end{align}

\subsubsection{Interpreting binary constraint-to-vertex messages as convolutions}

The denominator in the binary case of constraint-to-vertex messages further simplifies as follows:
\begin{align*}
    \int \int f(x_v, x_w) \cdot m^{w\to f}_t(x_w) d\Unif(x_v) d\Unif(x_w) &= \int m^{w\to f}_t(x_w) \int \Ind[\angles*{x_v,x_w}\ge\tau(p)] d\Unif(x_v) d\Unif(x_w) \\
    &=  p \int m^{w\to f}_t(x_w) d\Unif(x_w) \\
    &= p.
\end{align*}
The final equality comes from the fact that each message $m_t^{w \to f}$ is a distribution.
Thus, for binary constraints $f$ such that $\partial f = \{v,w\}$ we can rewrite the constraint-to-vertex messages as:
\[
    m^{f\to v}_{t+1}(x_v) = \int \frac{f(x_v, x_w)}{p} m^{w\to f}_t(x_w) d\Unif(x_w),
\]
which, in particular, can be written as $Pm^{w\to f}$ for a linear operator $P$,
defined as follows:
\begin{definition}\label{def:Pnu-is-cap}
Let $P$ be the linear operator defined so that for any function $h:\bbS^{d-1} \to \R$,
\[
P h (x) = \frac{1}{p} \int_{\scap_p(x)} h(y)\, d\rho(y),
\]
which we alternately denote $\frac{1}{p}h(\scap_p(x))$.
\end{definition}
In words, the operator $P$ convolves its input with the uniform distribution over a spherical cap. 
Since $P$ is a convolution operator, it preserves the $\ell_1$-norm of nonnegative functions.
\begin{observation} \label{obs:P-fixes-norm}
    Suppose $\nu$ is a nonnegative function, then $\norm*{P\nu}_1 = \norm*{\nu}_1$.  Additionally, for an arbitrary function $\nu$, $\norm*{P\nu}_1 \le \norm*{\nu}_1$.
\end{observation}
When $\nu$ is a distribution in particular, $P$ can be construed as the transition operator of the Markov chain on $\bbS^{d-1}$ where a single step entails walking from $v$ to a uniformly random point in $\scap(v)$.
\par The following useful observations are immediate from how we define the operator $P$.
\begin{observation} \label{obs:Pnu-bounded}
    For any function $\nu$, $\norm*{P\nu}_{\infty}\le\frac{\norm*{\nu}_1}{p}$. This is because $\nu(\scap(z)) \leq \norm*{\nu}_1$ for any $z \in \bbS^{d - 1}$.
\end{observation}

\subsubsection{A potential function for analyzing BP.}

Here, we introduce the notion of the {\em spread profile}, which we will use as potential function for tracking how close the density defined by a message is to uniform over $L(\bW_R)$.
We'll show that with each successive application of the operator $P$ which occurs when messages pass from the leaves to the root of the tree, the spread of the function improves, until ultimately the spread of the root's function guarantees that it is close to uniform.

\begin{definition}
    Given a function $\nu$ on $\bbS^{d-1}$, its \emph{deviation profile} is the function $\Dev_{\nu}:\R_+\to[0,1]$:
    \[
       \Dev_{\nu}(\eps) \coloneqq \Pr_{\bz\sim\Unif} \bracks*{ \nu(\bz) \notin \E_{\by \sim \Unif}[\nu(\by)] \pm \eps \cdot \norm*{\nu}_1 }.
    \]
    The \emph{spread profile} of $\nu$, denoted $\Spr_\nu(\delta)$, is
    $$
       \Spr_\nu(\delta) \coloneqq \inf \left\{\eps \in \bbR_+: \Dev_\nu(\eps) \leq \delta \right\}.
    $$
\end{definition}
In the special case where $\nu$ is a relative density of a distribution with respect to $\rho$, 
$$\Dev_{\nu}(\eps) = \Pr_{\bz \sim \rho}[|\nu(\bz)-1|> \eps].$$
We comment that we can think of the spread profile as an ``inverse'' to the deviation profile, since it takes in a tail probability and returns the corresponding deviation $\varepsilon$.

It now follows from averaging arguments that the deviation profile and spread profile of a distribution $\nu$ give us useful upper and lower bounds on $\dtv{\nu}{\Unif}$.
\begin{observation} \label{obs:tv-to-spread}
    For any $\eps > 0$ and distribution $\nu$, $$\Dev_\nu(\eps) \cdot \eps \leq \dtv{\nu}{\Unif} \leq \eps + \Dev_{\nu}(\eps)\cdot\norm*{\nu}_{\infty}$$ 
    Similarly, for any $\delta > 0$ and distribution $\nu$, 
    $$\delta \cdot \Spr_\nu(\delta) \leq \dtv{\nu}{\rho} \leq \Spr_\nu(\delta) + \delta \cdot\norm*{\nu}_{\infty}$$ 
\end{observation}

\noindent We list some properties of the deviation and spread profiles that will be useful in our analysis.
\begin{observation}
    Both the deviation profile and spread profile are non-increasing functions.
\end{observation}
\begin{observation} \label{obs:spread-of-unif}
    For $\nu = \Unif$, we have $\Dev_\nu(\eps) = 0$ for all $\eps > 0$ and $\Spr_\nu(\delta) = 0$ for all $\delta > 0$. 
\end{observation}
\begin{observation} \label{obs:spread-invariant}
    Both the deviation and spread profiles are invariant under a constant factor multiplication to $\nu$, i.e. for $\alpha \ne 0$:
    $$
    \Dev_{\alpha \nu}(\eps) = \Dev_{\nu}(\eps) \text{ and } \Spr_{\alpha \nu}(\delta) = \Spr_{\nu}(\delta)
    $$ 
\end{observation}
\begin{observation} \label{obs:spread-triangle-inequality}
    For $\nu = \nu_1 + \nu_2$, and $\eps > 0$, the deviation profile satisfies:
    $$
       \Dev_\nu\parens*{\eps\cdot\frac{\norm*{\nu_1}_1 + \norm*{\nu_2}_1}{\norm*{\nu}_1} } \leq \Dev_{\nu_1}(\eps) + \Dev_{\nu_2}(\eps)
    $$
    The ``triangle inequality'' follows from the following containment of events:
    $$
       \{\nu(\bz) \notin \E_{\by \sim \Unif}[\nu(y)] \pm \varepsilon \cdot \parens*{\norm*{\nu_1}_1 + \norm*{\nu_2}_1}\} \subseteq \{\nu_1(\bz) \notin \E_{\by \sim \Unif}[\nu_1(y)] \pm \varepsilon \cdot \norm*{\nu_1}_1\} \cup \{\nu_2(\bz) \notin \E_{\by \sim \Unif}[\nu_2(y)] \pm \varepsilon \cdot \norm{\nu_2}_1\}.
    $$
    Similarly, we also have a ``triangle inequality'' for the spread profile. For $\nu = \nu_1 + \nu_2$ and $\delta_1, \delta_2 > 0$:
    $$
        \norm{\nu}_1 \cdot \Spr_\nu(\delta_1 + \delta_2) \leq \norm{\nu_1}_1 \cdot \Spr_{\nu_1}(\delta_1) + \norm{\nu_2}_1 \cdot \Spr_{\nu_2}(\delta_2)
    $$
    We can think of the left hand side as the  $(1 - \delta_1 - \delta_2)$-confidence interval of $\nu$, and the right hand terms as the $(1 - \delta_1)$- and $(1 - \delta_2)$-confidence intervals of $\nu_1$ and $\nu_2$, respectively.
\end{observation}

\subsubsection{Diffusion of distributions under convolutions with caps.}
We will show that the cap-convolution operator $P$ ``flattens'' functions over $\bbS^{d-1}$.
Concretely:
\begin{lemma} \label{lem:diffusion-any-dist}
    Let $\nu$ be any function on $\bbS^{d-1}$ with $\E_{\by\sim\Unif} \nu(\by) = 0$, and suppose $d\ge \log^{10} n\cdot \parens*{1+\log^3 \frac{\norm*{\nu}_{\infty}}{ \norm*{\nu}_{1}}}$.  Then:
    \[
       \Dev_{P\nu}\left(\eps\cdot\frac{\norm*{\nu}_1}{\norm*{P\nu}_1}\right)\le n^{-\log^4 n},
    \]
    where $\eps \coloneqq \sqrt{\frac{1}{d}}\cdot C\log^{5.5}n\cdot\parens*{2 \ln \frac{\norm*{\nu}_{\infty}}{ \norm*{\nu}_{1}} +8\ln\frac{d}{p}}$ for $C$ an absolute constant. 
Thus, $\Spr_{P\nu}(n^{-\log^4 n})\le \eps \cdot \frac{\norm*{\nu}_1}{\norm*{P\nu}_1}$.
\end{lemma}
\noindent A special case of \pref{lem:diffusion-any-dist} is where $\nu = \mu - 1$ for $\mu$ a relative density.
In this scenario we have
\[
n^{-\log^4 n} 
\ge \Dev_{P\nu}\left(\eps\frac{\|\nu\|_1}{\|P\nu\|_1}\right) 
= \Pr_{\bz\sim \rho}\left[|P\nu(\bz)| >  \eps\frac{\|\nu\|_1}{\|P\nu\|_1} \cdot \|P\nu\|_1\right]
= \Pr_{\bz\sim \rho}\left[|P\mu(\bz) - 1| > \eps \cdot \dtv{\mu}{\rho}\right],
\]
and hence by an averaging argument $\dtv{P\mu}{\rho} \le \eps\cdot \dtv{\mu}{\rho} + n^{-\log^4 n}$.
In such a way, \pref{lem:diffusion-any-dist} characterizes the diffusion of distributions under $P$.

This is our main ingredient in the proof of \pref{lem:main-nbrhood-lemma}; the geometric decay of $\left(\sqrt{\frac{\log^{28} n}{d}}\right)^{\ell}$ in the definition of $\eta(\ell)$ is due to a factor of $\eps^\ell$ arising from the $\ell$ applications of $P$ occur as the messages travel from from the leaves of $\Balls$ to the roots (i.e., the vertices of $S$). 

The key property we use to establish \pref{lem:diffusion-any-dist} is that the value of $\nu(\scap(\bz))$ concentrates when $\bz$ is chosen uniformly at random from $\bbS^{d-1}$ and $\norm*{\nu}_{\infty}$ is reasonably bounded.
\begin{proof}[Proof of \pref{lem:diffusion-any-dist}]
    First, let us write $\nu = \nu^+ - \nu^-$
    where $\nu^+ = \max(\nu, 0)$ and $\nu^- = -\min(\nu, 0)$.
    Here, both $\nu^+$ and $\nu^-$ are nonnegative functions, and $\norm*{\nu} = 2\norm*{\nu^+} = 2\norm*{\nu^-} = \norm*{\nu^+} + \norm*{\nu^-}$ for $\|\cdot\|$ any $\ell_p$ norm.
    Since $P\nu = P\nu^+ - P\nu^-$, by \pref{obs:spread-triangle-inequality} and \pref{obs:spread-invariant}:
    $$
       \Dev_{P\nu}\parens*{\eps\cdot\frac{\norm*{P\nu^+}_1+\norm*{P\nu^-}_1}{\norm*{P\nu}_1}} \leq \Dev_{P\nu^+}(\eps) + \Dev_{P\nu^-}(\eps) = \Dev_{P\frac{\nu^+}{\norm*{\nu^+}}}(\eps) + \Dev_{P\frac{\nu^-}{\norm*{\nu^-}}}(\eps).
    $$
    Note that $\frac{1}{\|\nu^+\|_1}\nu^+$ and $\frac{1}{\|\nu^-\|_1}\nu^-$ are both probability measures.
    By \pref{def:Pnu-is-cap}, for any $x \in \bbS^{d - 1}$:
    \[
       P\frac{\nu^+(x)}{\norm*{\nu^+}_1} = \frac{1}{p} \cdot \frac{\nu^+(\scap(x))}{\norm*{\nu^+}_1} \qquad\text{and}\qquad
       P\frac{\nu^-(x)}{\norm*{\nu^-}_1} = \frac{1}{p} \cdot \frac{\nu^-(\scap(x))}{\norm*{\nu^-}_1}
    \]
    By \pref{cor:two-sets-subexp} with $s$ set as $\frac{1}{d} \cdot C \log^{5} n \cdot \left(2 \ln \frac{\norm*{\nu^+}_\infty}{\norm*{\nu^+}_1}  + 8 \ln \frac{d}{p} \right)$ for some constant $C > 0$:
    \[
       \Pr_{\bx\sim\Unif}\left[ \left|\frac{1}{p} \cdot \frac{\nu^+(\scap(\bx))}{\norm*{\nu^+}_1} - 1 \right| > C \cdot \Cconcalt \sqrt{\frac{1}{d} \cdot \log\left(\frac{1}{p}\right)}\cdot \log^{5}(n)  \cdot \left(2 \ln \frac{\norm{\nu^+}_\infty}{\norm{\nu^+}_1}  + 8 \ln \frac{d}{p} \right) \right] \le \frac{1}{2} \cdot n^{-\log^4 n}.
    \]
    Using $p\ge\frac{1}{n}$, $\norm*{\nu^+}_{\infty}\le\norm*{\nu}_{\infty}$ and $\norm*{\nu^+}_1 = \frac{1}{2}\norm*{\nu}$ we can conclude:
    \[
       \Dev_{P\frac{\nu^+}{\norm*{\nu^+}}}(\eps) \le \frac{1}{2}n^{-\log^4 n}.
    \]
    Identically, $\Dev_{P\frac{\nu^-}{\norm*{\nu^-}}}(\eps) \le \frac{1}{2}n^{-\log^4 n}$, and thus:
    \[
       \Dev_{P\nu}\parens*{\eps\cdot\frac{\norm*{P\nu^+}_1+\norm*{P\nu^-}_1}{\norm*{P\nu}_1}} \le n^{-\log^4 n}.
    \]
    Since $\nu^+$ and $\nu^-$ are nonnegative functions, $\norm*{P\nu^+}=\norm*{\nu^+}$ and $\norm*{P\nu^-}=\norm*{\nu^-}$, and the desired statement then immediately follows.
\end{proof}

\paragraph{Diffusion of intersection of anti-caps under $P$.}
Since the BP constraints require that the vectors be contained in $L\parens*{\bW_{\Rest}}$, we would also like to understand the diffusion of $L\parens*{\bW_{\Rest}}$ under $P$.  
Since $L\parens*{\bW_{\Rest}}$ is an intersection of anticaps, we can greatly strengthen \pref{lem:diffusion-any-dist} for such sets.
\begin{lemma}   \label{lem:diffusion-correlated-anticap-intersection}
    Define
    \begin{align*}
       g\parens*{\bG_{n-1}} &\coloneqq \Pr_{\bw_1,\dots,\bw_{n-1}\sim\VecDist{\bG_{n-1}}} \bracks*{\Dev_{P L\parens*{\bW_{\Rest}}} \parens*{\eps} > n^{-\log^4 n}} \\
       &=\Pr_{\bw_1,\dots,\bw_{n-1}\sim\VecDist{\bG_{n-1}}} \bracks*{\Spr_{P L\parens*{\bW_{\Rest}}} \parens*{n^{-\log^4 n}} > \eps}
    \end{align*}
    for $\eps = \sqrt{\frac{\log^{11}n}{nd}}$. 
    Then, $g\parens*{\bG_{n-1}}$ is at most $O\parens*{n^{-\log^2 n}}$ except with probability $O\parens*{n^{-\log^2n}}$.
\end{lemma}
\begin{proof}
       While we ultimately want control on $\Dev_{P L(\bW_{\Rest})}$, our starting point is to control only $\Dev_{PL(\bW)}$.  
We take advantage of concentration over $\bW$ to deduce concentration over $\bW_{T}$, for any $T\subseteq [n - 1]$ of sufficiently large size. 
    We will show below that the following hold:
    \begin{enumerate}[wide, labelwidth=!, labelindent=0pt]
       \item[(I)] {Concentration of anticap intersection:} \mylabel{property:anticap-intersect}{(I)} with probability $1- O({n^{-\log^2 n}})$ over randomness of $\bG_{n-1}$,
       $$\Pr_{\bw_1,\dots,\bw_{n-1}\sim\VecDist{\bG_{n-1}}} \bracks*{ \Dev_{P L(\bW)}(\eps/2) > n^{-\log^5 n}} \le O\parens*{n^{-\log^2 n}}.$$
       \item[(II)] { Bounded pairwise cap intersections:} \mylabel{property:cap-pairwise}{(II)}
       for any $x\in\bbS^{d-1}$:
       \[
         \Pr_{\bw_1,\dots,\bw_{n-1}\sim\VecDist{\bG_{n-1}}} \bracks*{\exists i\in[n-1]\text{ s.t. }\Unif\parens*{\scap(x)\cap\scap(\bw_i)} \ge 2p^2} \le O\parens*{ n^{-\log^3 n / 2} }
       \]
       except with probability $O\parens*{ n^{-\log^2 n} }$ over the randomness of $\bG_{n-1}$.
    \end{enumerate}
    Assume first that \pref{property:anticap-intersect} and \pref{property:cap-pairwise} hold.
    Define 
    $$h\parens*{w_1,\dots,w_{n-1})} \coloneqq \Pr_{\bx\sim\Unif}\bracks*{ \exists i \in [n-1]\text{ s.t. }\Unif\parens*{\scap(\bx)\cap\scap(w_i)} \ge 2p^2 }$$
    \pref{property:cap-pairwise} implies:
    \[
       \E_{\bw_1,\dots,\bw_{n-1}\sim\VecDist{\bG_{n-1}}} h\parens*{\bw_1,\dots,\bw_{n-1})} \le O\parens*{ n^{-\log^3 n/4} } \numberthis \label{eq:pairwise-expected}
    \]
    except with probability $O\parens*{ n^{-\log^3 n/4} }$.  By Markov's inequality, whenever \pref{eq:pairwise-expected} holds:
    \[
       \Pr_{\bw_1,\dots,\bw_{n-1}\sim\VecDist{\bG_{n-1}}} \bracks*{h\parens*{\bw_1,\dots,\bw_{n-1}} \le O\parens*{n^{-\log^3 n/8}} } \le O\parens*{n^{-\log^3 n / 8}}.
    \]
    Suppose $\bw_1,\dots,\bw_{n-1}$ are such that $\Dev_{P L(\bW)}(\eps/2)\le n^{-\log^4 n}$ and $h(\bw_1,\dots,\bw_{n-1}) \le O\parens*{n^{-\log^3 n/8}}$.
    By our assumptions, these conditions hold with probability at least $1-O\parens*{n^{-\log^2 n}}$ when $\bw_1,\dots,\bw_{n-1}\sim\VecDist{\bG_{n-1}}$, with probability at least $1-O\parens*{n^{-\log^2 n}}$ over the randomness of $\bG_{n-1}$.
    For $\bx\sim\bbS^{d-1}$, with probability at least $1-O\parens*{n^{-\log^2 n}}$:
    \begin{align*}
       \Unif\parens*{\bigcap_{i=1}^{n-1} \anticap(\bw_i) \cap \scap(\bx)} &\in (1\pm\eps/2) (1-p)^{n-1} p \\
       \Unif\parens*{\scap(\bx)\cap\scap(\bw_i)} &\le 2p^2 & \forall i\in[n-1].
    \end{align*}
    As a consequence, with probability at least $1-O\parens*{n^{-\log^2 n}}$ as well:
    \begin{align*}
       \Unif\left(\bigcap_{i\in [n-1]}\anticap(\bw_i)\cap\scap(\bx)\right) &= \Unif \left( \left( \bigcap_{i \in R} (\anticap(\bw_i) \cap \scap(\bx)) \right) \setminus \bigcup_{j \in [n - 1] \setminus R} (\scap(\bw_j) \cap \scap(\bx)) \right) \\
       &\ge \Unif\left(\bigcap_{i\in R} \anticap(\bw_i) \cap \scap(\bx) \right) - 2 (n-1-|R|)p^2.
    \end{align*}
    The inequality follows from assuming the sets $\scap(\bw_j)$ for $j \in [n - 1] \setminus R$ are disjoint in the worst case.
    On the other hand, as $R \subseteq [n - 1]$,
    \[
       \Unif\left(\bigcap_{i\in [n-1]}\anticap(\bw_i)\cap\scap(\bx)\right) \le \Unif\left(\bigcap_{i\in R} \anticap(\bw_i) \cap \scap(\bx) \right).
    \]
    The above can be rearranged as:
    \[
       \Unif\left(\bigcap_{i\in [n-1]}\anticap(\bw_i)\cap\scap(\bx)\right) \le \Unif\left(\bigcap_{i\in R} \anticap(\bw_i) \cap \scap(\bx) \right) \le \Unif\left(\bigcap_{i\in [n-1]}\anticap(\bw_i)\cap\scap(\bx)\right) + 2 (n-1-|R|)p^2,
    \]
    and thus
    \[
       \Unif\parens*{ \bigcap_{i\in R} \anticap(\bw_i) \cap \scap(\bx) } \in p\cdot(1-p)^{|R|}\cdot\left((1-\eps)\cdot(1-p)^{n-1-|R|}, (1+\eps)\cdot(1+C(n-1-|R|)p)\right)
    \]
    for constant $C$.
    By applying a union bound to \pref{lem:ball-bound} along with \pref{prop:dom} and our bounds on $\ell$ and $|S|$:
    \[
       n-1-|R| = |\Balls| \le n^{\log (2\alpha)/\log\log n}\cdot\log^2 n
    \]
    except with probability at most $n^{-\Omega(\log n)}$.
    Thus, the above simplifies to:
    \[
       \Unif\parens*{\bigcap_{i\in R}\anticap(\bw_i)\cap\scap(\bx)} \in p\cdot(1-p)^{|R|}\cdot\left(1-\eps, 1+\eps\right).
    \]
    This establishes that $\Dev_{P L(\bW_R)}(\eps) \le O\parens* {n^{-\log^5 n}} \le n^{-\log^4 n}$.

    It remains to prove \pref{property:anticap-intersect} and \pref{property:cap-pairwise}.
    \pref{property:anticap-intersect} follows from a combination of \pref{obs:unif-to-conditional} and \pref{cor:anti-caps} with parameter setting $t=C\sqrt{\frac{\log^{11}n}{d}}$ for some constant $C$ and $m = n$.  
    To prove \pref{property:cap-pairwise}, first observe that applying a union bound to \pref{cor:cap-intersect-simple} implies that when $\bw_1,\dots,\bw_{n-1}\sim\Unif^{\otimes n-1}$, $\Unif\parens*{ \scap(x) \cap \scap(\bw_i) }$ is at most $2p^2$ for all $i\in[n-1]$ except with probability at most $n^{-\log^3 n/2}$.  \pref{property:cap-pairwise} then follows from \pref{obs:unif-to-conditional}, hence completing our proof. 
\end{proof}

\paragraph{Applying $P$ to distributions close to $L\parens*{\bW_{\Rest}}$.}
We now describe our main result about the diffusion properties of $P$ when applied to distributions close to the uniform distribution on $L\parens*{\bW_{\Rest}}$.
\begin{lemma} \label{lem:convolution-tv}
    Let $\nu$ be a distribution over $\bbS^{d - 1}$ such that $\norm*{\nu}_\infty \leq n^{\log^5 n}$, and let $\delta = \dtv \nu {L\parens*{\bW_{\Rest}}}$.
    Then, 
    $$
       \Spr_{P \nu}\parens*{2n^{-\log^4 n}} \leq \max \left\{2\sqrt{\frac{\log^{11}n}{nd}}, C\delta \cdot \sqrt{\frac{\log^{23}n}{d}}\right\}
    $$
    for a universal constant $C$.
\end{lemma}
\begin{proof}
    We can express $\nu = L\parens*{\bW_{\Rest}} + \Delta$. 
    Then, the triangle inequality for $\Spr_{P\nu}$ as articulated in \pref{obs:spread-triangle-inequality} implies:
    \begin{align*}
       \Spr_{P\nu}(2 n^{-\log^4 n}) &\le \Spr_{P L\parens*{\bW_{\Rest}})}(n^{-\log^4 n}) + \norm{P\Delta}_1 \cdot \Spr_{P\Delta}(n^{-\log^4 n}).    \numberthis \label{eq:spread-decomp}
    \end{align*}
    By \pref{lem:diffusion-correlated-anticap-intersection}, $\Spr_{PL\parens*{\bW_{\Rest}}}(n^{-\log^4 n}) \le \sqrt{\frac{\log^{11}n}{nd}}$.

    We now turn our attention to $P\Delta$.
    We would like to understand $P\Delta$ through \pref{lem:diffusion-any-dist}, but the lower bound on $d$ of $\log^{10}n\cdot\parens*{1+\log^3\frac{\norm*{\nu}_{\infty}}{\norm*{\nu}_1}}$ in the hypothesis of its statement prevents us from applying it when $\delta$ is too small, and hence we case on the value of $\delta$.
    Suppose $\delta\le n^{-2\log^4 n}$, then by \pref{obs:P-fixes-norm}, $\norm*{P\Delta}_1\le n^{-2\log^4 n}$ and thus by Markov's inequality $\Spr_{P\Delta}(n^{-\log^4 n}) \le n^{\log^4 n}$.
    Plugging this into \pref{eq:spread-decomp} gives:
    \[
        \Spr_{P\nu}(2n^{-\log^4 n}) \le \sqrt{\frac{\log^{11}n}{nd}} + n^{-\log^4 n} \le 2\sqrt{\frac{\log^{11}n}{nd}}.
    \]
    When $\delta > 2n^{-\log^4 n}$, by \pref{lem:diffusion-any-dist}, we have for some constant $C$:
    \begin{align*} 
       \Spr_{P\Delta}(n^{-\log^4 n}) &\leq \sqrt{\frac{1}{d}} \cdot C \log^{5.5}n \cdot \left(2 \ln \frac{n^{\log^5 n}}{\norm*{\Delta}_1} + 8 \ln \frac{d}{p} \right)\cdot\frac{\norm*{\Delta}_1}{\norm*{P\Delta}_1} \\ 
       &\leq \sqrt{\frac{1}{d}} \cdot C \log^{5.5}n \cdot \left(4 \log^6 n + 8 \ln \frac{d}{p} \right) \cdot\frac{\norm*{\Delta}_1}{\norm*{P\Delta}_1}
    \end{align*}
    where the second inequality uses $\norm*{\Delta}_1 = \delta > n^{-2 \log^4 n}$.
    Consequently, plugging into \pref{eq:spread-decomp},
    \[
       \Spr_{P\nu}\parens*{n^{-\log^4 n}} \le \sqrt{\frac{\log^{11}n}{nd}} + C\sqrt{\frac{\log^{23}n}{d}}\cdot\delta
    \]
    which completes the proof.
\end{proof} 

\subsubsection{Analyzing vertex-to-constraint messages.}

Recall that our vertex-to-constraint messages are given as a product of the vertex's incoming messages.
Each vertex's incoming messages are themselves measures over $\bbS^{d-1}$ with bounded spread; we show that as long as a vertex's degree is bounded, the spread of their product is not too large.
\begin{lemma} \label{lem:product-tv}
    Let $\nu_1, \ldots, \nu_j$ be distributions over $\bbS^{d - 1}$ with intersecting support.
    For each $i \in [j]$, suppose $\Spr_{\nu_i}(\eps)\le\eta$ and $\norm*{\nu_i}_\infty \leq \frac{1}{p}$.  
    Further, assume $\eta (j+1) < \frac{1}{2}$ and $\frac{j\eps}{p^j} \le \eta$.
    Let $\nu$ be the distribution whose density is $\nu(x) \propto \prod_{i = 1}^j \nu_i(x)$.
    Then,
    \[
        \Spr_{\nu}(j\cdot\eps) \le 8j\eta\qquad\text{and}\qquad\norm*{\nu}_{\infty} \le \frac{2}{p^j}.
    \]
    As a consequence of \pref{obs:tv-to-spread}:
    $$
        \dtv{\nu}{\Unif} \leq 10j\eta.
    $$
\end{lemma}
\begin{proof}
    Let $B$ be the ``bad'' domain defined $B\coloneqq\braces*{x\in\bbS^{d-1}: \exists i\text{ s.t. }\nu_i\notin[1-\eta,1+\eta]}$.
    We know $\Unif(B)\le j \eps$ by our assumption on the spread profiles of $\nu_1,\dots,\nu_j$.
    Thus, for $x\in\bbS^{d-1}\setminus B$:
    \[
        \frac{(1-\eta)^j}{Z} \le \nu(x) \le \frac{(1+\eta)^j}{Z}
    \]
    where
    \[
        Z = \int \prod_{i=1}^j \nu_i(x) d\Unif(x).
    \]
    To obtain upper and lower bounds on $Z$, we write
    \begin{align*}
        Z &= \int \prod_{i=1}^j \nu_i(x) \Ind[x\in\bbS^{d-1}\setminus B]d\Unif(x) + \int \prod_{i=1}^j \nu_i(x) \Ind[x\in B] d\Unif(x) \\
        &\begin{cases}
            \ge (1-\eta)^j(1-j\eps) \ge (1-\eta)^{j+1} \ge 1-(j+1)\eta \\
            \le (1+\eta)^j + \frac{j\eps}{p^j} \le 1+\frac{j\eta}{1-j\eta} + \eta \le 1 + \frac{(j+1)\eta}{1-j\eta}.
        \end{cases}
    \end{align*}
    We can obtain lower and upper bounds on $\nu(x)$ for $x\in\bbS^{d-1}\setminus B$ as follows.
    \begin{align*}
        \nu(x) &\ge \frac{(1-\eta)^j}{1+\frac{(j+1)\eta}{1-j\eta}} \ge \frac{1-j\eta}{1+\frac{(j+1)\eta}{1-j\eta}} = 1 - \frac{j\eta + \frac{(j+1)\eta}{1-j\eta}}{1+\frac{(j+1)\eta}{1-j\eta}} \ge 1 - 5j\eta \\
        \nu(x) &\le \frac{(1+\eta)^j}{1-(j+1)\eta} \le \frac{1+\frac{j\eta}{1-j\eta}}{1-(j+1)\eta} = 1 + \frac{(j+1)\eta + \frac{j\eta}{1-j\eta}}{1-(j+1)\eta} \le 1 + 8j\eta.
    \end{align*}
    Since $Z\ge 1-(j+1)\eta\ge\frac{1}{2}$, and each $\norm*{\nu_i}_{\infty}\le\frac{1}{p}$, $\norm{\nu}_{\infty} \le \frac{2}{p^j}$.
\end{proof}

\subsubsection{Tracking the spread in BP fixed point messages}
We now study the spread properties of the BP fixed point messages.

In this section let $T$ be a tree with a distinguished root vertex $r$.
Let $F$ be a factor graph associated to the tree where the variable vertices are the vertices of $T$, with a binary constraint $\Ind\bracks*{\angles{v_i,v_j}\ge\tau(p)}$ for every edge $\{i,j\}\in T$, and a unary constraint $\Ind\bracks*{v_i\in L_i}$ for every vertex $i\in V(T)$.
\begin{assumption} \label{assumption:tree-properties}
    We assume that the degree of every vertex in $T$ is bounded by $\log^2 n$, and that for every $i\in V(T)$, $\Unif(L_i)\ge n^{-\log^3 n}$.%
\end{assumption}

For a variable vertex $v\in F$ for $v\ne r$, we use $\Par(v)$ to denote the constraint vertex corresponding to the edge $\{v,w\}$ where $w$ is the parent of $v$ in $T$.
We use $\Ch(v)$ to denote the set of constraint vertices corresponding to edges $\{v,w\}$ for $w$ that are children of $v$ in $T$.
Finally, we use $\Un(v)$ to denote the unary constraint vertex incident to $v$.

\begin{remark}
    Observe that trees that arise from rooting connected components of $T\parens*{\bG_{n-1}}$ at vertices in $S$ always have all leaves equidistant to the root and also satisfy the inequality
    \[
        \Spr_{PL_i}\parens*{n^{-\log^4 n}} \le \sqrt{\frac{\log^{11}n}{nd}}
    \]
    for every internal vertex $i$, which are properties that trees meeting our above assumptions need not satisfy.
    This may suggest adding these properties as assumptions too.
    However, we stick to the more general setting as it is crucial in our proof relating the relaxed and true distributions.
\end{remark}

We place our vertices in to ``tiers'', which will let us quantify the spread of messages sent out by them.

\begin{definition}
    We say $i\in V(T)$ is a \emph{free vertex} if $L_i = L\parens*{\bW_{\Rest}}$ and $\Spr_{PL_i}\parens*{n^{-\log^4 n}} \le \sqrt{\frac{\log^{11}n}{nd}}$.
    We say $i\in T$ is a \emph{anchored vertex} if it is not a free vertex.  We use $\Res(T)$ to denote the set of anchored vertices in $T$.
\end{definition}

\begin{definition}   \label{def:tiers}
    Given a vertex $i\in V(T)$ we say its \emph{tier} is the distance to its closest descendant that is an anchored vertex.  In particular,
    \[
        \Tier(v) = \min\{\beta:v\text{ has a distance-$\beta$ descendant in }\Res(T)\}.
    \]
    We use the convention that $v$ is a descendant of itself and hence $\Tier(v) = 0$ when $v$ is a anchored vertex.
\end{definition}

For a variable vertex $i$, the closeness of the distribution obtained by combining messages from its children to $\Unif$ is governed by its tier.
\begin{lemma}   \label{lem:tiering}
    For a variable vertex $i$ in $F$, let $\nu_i$ be the distribution
    \[
        \nu_i \propto \prod_{a\in\Ch(i)} m^{a\to i}.
    \]
    Then:
    \[
        \Spr_{\nu_i}\parens*{n^{-\log^4 n/2}} \le
        \min\left\{\eps, \max\left\{\eps^{\Tier(i)}, C\sqrt{\frac{\log^{15}n}{nd}}\right\}\right\} \qquad \text{and} \qquad \norm*{\nu_i}_{\infty} \le \frac{2}{p^{\log^2 n}}.
    \]
    where $\eps = C\sqrt{\frac{\log^{27}n}{d}}$ for some universal constant $C$.
\end{lemma}
\begin{proof}
    In this proof, we will use $j_1,\dots,j_s$ to denote the children of $i$ in $T$.
    Then $\Ch(i) = \{a_1,\dots,a_s\}$ where $a_t$ connects $i$ and $j_t$ in the factor graph.

    We first show the weaker statement that $\Spr_{\nu_i}\parens*{n^{-\log^4 n/2}} \le \eps$ and $\norm*{\nu_i}_{\infty} \le \frac{2}{p^{\log^2 n}}$ for every $i\in V(T)$ via induction on the depth of the subtree rooted at $i$.
    When this depth is equal to $0$, the vertex $i$ has no children, and then the result is clearly true since $\nu_i = \Unif$.

    Now, suppose the weaker statement is true for every vertex with a depth-$q$ subtree.
    Let $i$ be a vertex with a depth-$(q+1)$ subtree.
    Note that $m^{a_t\to i} = Pm^{j_t\to a_t}$.
    We know that $m^{j_t\to a_t}\propto \nu_{j_t}\cdot\Ind\bracks*{L_{j_t}}$.
    On one hand
    \[
        \norm*{\nu_{j_t}\cdot\Ind\bracks*{L_{j_t}}}_{\infty} \le \frac{2}{p^{\log^2 n}}
    \]
    and on the other hand by the induction hypothesis and our assumption on $L_{j_t}$,
    \[
        \norm*{\nu_{j_t}\cdot\Ind\bracks*{L_{j_t}}}_1 \ge (1-\eps)\cdot \parens*{n^{-\log^3 n} - n^{-\log^4 n/2}} \ge \frac{1}{2}n^{-\log^3 n},
    \]
    and thus
    \[
        \norm*{m^{j_t\to a_t}}_{\infty} \le n^{\log^5 n}.
    \]
    By \pref{lem:convolution-tv}, we know:
    \[
        \Spr_{m^{a_t\to i}}\parens*{n^{-\log^4 n}} \le C\sqrt{\frac{\log^{23}n}{d}}
    \]
    for some absolute constant $C$.
    We also know $\norm*{m^{a_t\to i}}_{\infty}\le\frac{1}{p}$ by \pref{obs:Pnu-bounded}, and that $s\le\log^2 n$ by assumption.
    Therefore, by \pref{lem:product-tv},
    \[
        \Spr_{\nu_i}\parens*{n^{-\log^4 n/2}} \le C\sqrt{\frac{\log^{27}n}{d}}
    \]
    which establishes the weaker statement.

    Next, we prove $\Spr_{\nu_i}\parens*{n^{-\log^4 n / 2}} \le \max\braces*{\eps^{\Tier(i)}, \sqrt{\frac{\log^{15} n}{nd}}}$.
    We prove this statement by induction on $\Tier(i)$.
    The base case of $0\le \Tier(i)\le 1$ follows from the weaker statement proved above.

    Now, suppose $\Tier(i) \ge 2$.
    Note that $\norm*{\nu_{j_t}}_{\infty} \le \frac{2}{p^{\log^2 n}}$ and $m^{j_t\to a_t}\propto \nu_{j_t}\cdot\Ind[L_{j_t}]$.
    Since $\Tier(j_t)\ge \Tier(i)-1$, we know $\Tier(j_t)\ge 1$, which in particular means that $j_t$ is a free vertex and consequently
    \[
        \Spr_{PL_{j_t}}\parens*{n^{-\log^4 n}} \le \sqrt{\frac{\log^{11}n}{nd}}.
    \]
    By our induction hypothesis, the value of $\nu_{j_t}\cdot\Ind\bracks*{L_{j_t}}$ is in $1\pm\max\left\{\eps^{\Tier(j_t)}, \sqrt{\frac{\log^{15}n}{nd}} \right\}$ on a subregion of $L_{j_t}$ of measure at least:
    \[
        \Unif\parens*{L_{j_t}} - n^{-\log^4 n} \ge \parens*{1 - n^{-\log^4 n / 2}}\cdot\Unif\parens*{L_{j_t}}
    \]
    by our assumption that $\Unif(L_{j_t})\ge n^{-\log^3 n}$.
    Combined with our bound on $\norm*{\nu_{j_t}}_{\infty}$, we know that:
    \[
        \norm*{\nu_{j_t}\cdot\Ind\bracks*{L_{j_t}}}_1 \in \Unif\parens*{L_{j_t}} \cdot \parens*{1\pm \max\left\{\eps^{\Tier(j_t)}, \sqrt{\frac{\log^{15}n}{nd}} \right\} \pm n^{-\log^4 n / 2}}.
    \]
    We can then write $m^{j_t\to a_t}$ as $L_{j_t}+\Delta$ where $\norm*{\Delta}_1 \le \max\left\{\eps^{\Tier(j_t)}, \sqrt{\frac{\log^{15}n}{nd}} \right\}$ and $\norm*{\Delta}_{\infty} \le n^{\log^5 n}$.
    By \pref{lem:convolution-tv}, $m^{a_t\to i} = P m^{j_t\to a_t}$ satisfies the following for some constant $C$:
    \begin{align*}
        \Spr_{m^{a_t\to i}}\parens*{n^{-\log^4 n}} &\le \max\braces*{
            \sqrt{\frac{\log^{11}n}{nd}},
            C\sqrt{\frac{\log^{23}n}{d}}\max\braces*{
                \eps^{\Tier(j_t)}, \sqrt{\frac{\log^{15}n}{nd}
            }}
        }\\
        &= \max\braces*{\sqrt{\frac{\log^{11}n}{nd}}, C\sqrt{\frac{\log^{23}n}{d}}\cdot\eps^{\Tier(j_t)} }
    \end{align*}
    where the equality is due to our lower bound on $d$.  Additionally, since $\norm*{m^{a_t\to i}}_{\infty}\le\frac{1}{p}$ and $s\le\log^2 n$, we can use \pref{lem:product-tv} to conclude:
    \begin{align*}
        \Spr_{\nu_i}\parens*{n^{-\log^4 n/2}} &\le C\log^2 n \max\braces*{\sqrt{\frac{\log^{11}n}{nd}}, \eps^{\Tier(i)-1}\sqrt{\frac{\log^{23}n}{d}} } \\
        &\le \max\braces*{\eps^{\Tier(i)}, C\sqrt{\frac{\log^{15}n}{nd}}},
    \end{align*}
    which completes the proof.
\end{proof}

\subsubsection{Neighborhood containment for relaxed distribution}
As final preparation to prove \pref{lem:main-BP-lemma}, we list some properties that hold with high probability over the randomness of $\bG_{n-1}$ and $\bW_{\Rest}$ for the forest $T\parens*{\bG_{n-1}}$, and proceed by conditioning on these high probability events for the rest of this section.
\begin{lemma}   \label{lem:high-prob-properties}
    The following events simultaneously occur with probability at least $1-n^{-\Omega(\log n)}$ over the randomness of $\bG_{n-1}$ and $\bW_{\Rest}$, and all choices of $T(\bG_{n - 1})$, 
    \begin{enumerate}[(1)]
        \item \label{property:max-deg-bound} The maximum degree in $\bG_{n-1}$, and hence the maximum degree in $T\parens*{\bG_{n-1}}$ is bounded by $\log^2 n$.
        \item \label{property:all-sets-lower-bound} For every $x\in T(\bG_{n-1})$, $\Unif\parens*{L_x\parens*{\bW_{\Rest}}} \ge n^{-\log^3 n}$.
        \item \label{property:superconcentration} $\Spr_{PL(\bW_R)}\left(n^{-\log^4 n}\right) \le \sqrt{\frac{\log^{11}n}{nd}}$.
        \item \label{property:bounded-size} $|\Balls|\le n^{\max\{\log(2\alpha),1\}/\log\log n}$.
    \end{enumerate}
\end{lemma}
\begin{proof}
    \pref{property:max-deg-bound} follows from \pref{lem:degree-bound}, \pref{property:superconcentration} follows from \pref{lem:diffusion-correlated-anticap-intersection}, and \pref{property:bounded-size} follows from \pref{prop:dom} and \pref{lem:ball-bound}.
    It remains to prove \pref{property:all-sets-lower-bound}.
    To obtain a lower bound on $\Unif\parens*{ L_x\parens*{\bW_{\Rest}} }$, we use:
    \[
        \Unif\parens*{L_x\parens*{\bW_{\Rest}}} \ge \Unif\parens*{L_x\parens*{\bW}}.
    \]
    We will first prove that with probability $1-n^{-\Omega(\log n)}$, for all $Y\subseteq[n-1]$ such that $|Y|\le\log^2 n$,
    \[
        \Unif\parens*{\bigcap_{i\in Y}\scap(\bw_i) \cap \bigcap_{i\in[n-1]\setminus Y}\anticap(\bw_i)} \ge n^{-\log^3 n},
    \]
    which would then imply the desired statement.
    Suppose $\bw_1,\dots,\bw_{n-1}\sim\Unif^{\otimes n-1}$, then by \pref{lem:martingale-intersect} under setting the $s$ parameter to, say, $1/2$, and a union bound over all $Y$ such that $|Y|\le\log^2 n$:
    \begin{align*}
         \Unif\parens*{\bigcap_{i\in Y}\scap(\bw_i) \cap \bigcap_{i\in[n-1]\setminus Y}\anticap(\bw_i)} \ge n^{-\log^3 n} && \forall Y\subseteq[n-1],~|Y|\le\log^2 n \numberthis \label{eq:lower-bound-all}
    \end{align*}
    except with probability at most $n^{-\log n}$.
    By \pref{obs:unif-to-conditional}, \pref{eq:lower-bound-all} is true with probability $1-n^{-\Omega(\log n)}$ when $\bw_1,\dots,\bw_{n-1}\sim\VecDist{\bG_{n-1}}$ with probability $1-n^{-\Omega(\log n)}$ over the randomness of $\bG_{n-1}$, which establishes \pref{property:all-sets-lower-bound}.
\end{proof}

Now we proceed with the proof of \pref{lem:main-BP-lemma}.
\restatelemma{lem:main-BP-lemma}
\begin{proof}[Proof of \pref{lem:main-BP-lemma}]
    Let $\bW_{\Balls}\sim\calF'\parens*{\bW_{\Rest}}$, and without loss of generality let $S=\{1,\dots,s\}$.
    First observe that $\bw_1,\dots,\bw_s$ are independent conditioned on $\bW_{\Rest}$.
    Now, we use $C_i$ to denote the connected component of $T\parens*{\bG_{n-1}}$ where $i$ belongs.
    With probability $1-n^{-\Omega(\log n)}$ over the randomness of $\bG_{n-1}$ and $\bW_{\Rest}$, the events in the statement of \pref{lem:high-prob-properties} occur.
    For the rest of the proof we assume that these events occur.
    Then every $C_i$ satisfies \pref{assumption:tree-properties} and $\Res(C)\subseteq S_{\bG_{n-1}}(v,\ell-1)$.
    Thus, $\Tier(i)=\ell-1$, and consequently by \pref{lem:tiering} the distribution
    \[
        \nu_i \propto \prod_{a\in \Ch(i)} m^{a\to i}
    \]
    satisfies the following for some absolute constant $C$:
    \[
        \Spr_{\nu_i}\parens*{n^{-\log^4 n}} \le
        \max\left\{
            \sqrt{\frac{\log^{27}n}{d}}^{\ell-1}, C\sqrt{\frac{\log^{15}n}{nd}}
        \right\} \qquad \text{and} \qquad \norm*{\nu_i}_{\infty} \le \frac{2}{p^{\log^2 n}}.
    \]
    The marginal on $\bw_i$ is
    \[
        \wt{\nu}_i \propto \nu_i \cdot \Ind\bracks*{L_x\parens*{\bW_{\Rest}}}.
    \]
    $L_x\parens*{\bW_{\Rest}} = L\parens*{\bW_{\Rest}}$ when $\ell\ge 2$, and hence for $\ell\ge 1$, this satisfies
    \[
        \wt{\nu}_i = L\parens*{\bW_{\Rest}} + \Delta
    \]
    with $\norm*{\Delta}_1 \le 2\max\left\{\left(C\sqrt{\frac{\log^{27}n}{d}}\right)^{\ell-1}, C\sqrt{\frac{\log^{15}n}{nd}}\right\}$ and $\norm*{\wt{\nu}_i}_{\infty} \le {n^{\log^5n}}$.

    Then by the definition of $P$:
    \begin{align*}
        \Pr_{\substack{\bw_n \sim\bbS^{d-1} \\ \bW_{\Balls} \sim\calF'\parens*{\bW_{\Rest}}}} \bracks*{\forall x\in S:\angles*{\bw_x,\bw_n}\ge\tau(p)} &= p^{s}\E_{\bw_n \sim\bbS^{d-1}}\prod_{i=1}^s P\wt{\nu}_i (\bw_n)
    \end{align*}
    From \pref{lem:convolution-tv}, we know:
    \begin{align*}
        \Spr_{P\wt{\nu}_i}\parens*{2n^{-\log^4 n}} &\le
        \max\braces*{
            2\sqrt{\frac{\log^{11}n}{nd}},
            2C \cdot\left(C\sqrt{\frac{\log^{27}n}{d}}\right)^{\ell-1}\cdot \sqrt{\frac{\log^{23} n}{d}}
        }\\
        &\le \max\braces*{2\sqrt{\frac{\log^{11}n}{nd}}, \left(C\sqrt{\frac{\log^{27}n}{d}}\right)^{\ell}}.
    \end{align*}
    Using the above along with the fact that $\norm*{P\wt{\nu}_i}_{\infty}\le\frac{1}{p}$ for all $i$ implies:
    \[
        \Pr_{\substack{\bw_n\sim\bbS^{d-1} \\ \bW_{\Balls}\sim\calF'\parens*{\bW_{\Rest}}}} \bracks*{\forall x\in S:\angles*{\bw_x,\bw_n}\ge\tau(p)} \in \left(1\pm \max\braces*{4\sqrt{\frac{\log^{11}n}{nd}}, \left(2C\sqrt{\frac{\log^{27}n}{d}}\right)^{\ell}} \right) p^{s}
    \]
    which completes the proof.
\end{proof}

\subsection{Relating relaxed and true distributions}    \label{sec:relax-to-true}
In this section, we relate the uniform distribution on $\calF(\bW_{\Rest})$ to the uniform distribution on $\calF'(\bW_{\Rest})$.
We first prove that a random sample from $\calF'(\bW_{\Rest})$ falls in $\calF(\bW_{\Rest})$ with high probability.
\begin{lemma}   \label{lem:F-within-Fp}
    With probability at least $1-n^{-\Omega(\log n)}$ over the randomness of $\bG_{n-1}$ and $\bW_{\Rest}$,
    \[
       \Pr_{ \parens*{\bw_i}_{i\in \Balls} \sim \calF'(\bW_{\Rest}) }\bracks*{ \parens*{\bw_i}_{i\in \Balls} \notin \calF(\bW_{\Rest}) } \le pn^{O\parens*{1/\log\log n}}.
    \]
\end{lemma}
\begin{proof}
    It suffices to prove that for every pair of vertices $x,y\in \Balls$ such that $d_{T(\bG_{n-1})}(x,y)\ge 2$, the probability that $\angles*{\bw_x,\bw_y}\ge\tau(p)$ is at most $2p$.
    We can then apply the union bound over all $x, y$ pairs, of which there are $n^{O(1/\log\log n)}$ by \pref{property:bounded-size} of \pref{lem:high-prob-properties}.
    Let's assume that the events in \pref{lem:high-prob-properties} hold, which happens with probability at least $1-n^{-\Omega(\log n)}$.
    Consequently, $T(\bG_{n-1})$ satisfies \pref{assumption:tree-properties}.
    The marginal on $\bw_x$ is given by
    \[
        \wt{\nu}_x \propto \prod_{a\in\Ch(x)} m^{a\to x} \cdot \Ind\bracks*{L_x}.
    \]
    By \pref{lem:tiering} and the fact that $\Unif(L_x)\ge n^{-\log^{3}n}$, we can write $\wt{\nu}_x$ as $L_x + \Delta$ where $\norm*{\Delta}_{1} \le C\sqrt{\frac{\log^{27}n}{d}}$ for some constant $C$ and $\norm*{\wt{\nu}_x}_{\infty}\le n^{\log^5 n}$.
    By \pref{cor:two-sets-subexp-mu}, for $\bw_x\sim\wt{\nu}_x$, (1) $\Unif(\scap(\bw_x)\cap L_y) \le \frac{3p}{2} \cdot \Unif(L_y)$, and (2) $\Unif(\scap(\bw_x)\cap L_z)\ge\frac{p}{2} \cdot \Unif(L_z)$ for all $z\in V(T(\bG_{n-1}))$ simultaneously except with probability at most $n^{-\log^2 n}$.

    The distribution of $\bw_y|\bw_x$ is the marginal on variable indexed by $y$ of the uniform distribution over satisfying assignments to the following CSP instance.
    Define $T'$ as the tree obtained by deleting the subtree rooted at $x$ from $T\parens*{\bG_{n-1}}$.
    Let $u$ denote the parent of $x$ in $T$.
    Now, let $F'$ be the factor graph of the CSP instance $\Inst'$ with variables indexed by $V(T')$ and constraints:
    \begin{itemize}
        \item {\bf Unary constraints.}  For each $i\in V(T')$, if $i \ne u$, then $v_i\in L_i(\bW_R)$, and if $i = u$, then $v_u\in L_u(\bW_R)\cap\scap(\bw_x)$.
        \item {\bf Binary constraints.}  For every $\{i,j\}\in T'$: $\angles*{v_i,v_j}\ge\tau(p)$.
    \end{itemize}
    The distribution of $\bw_y|\bw_x$ can be computed from the belief propagation fixed point on $F'$.

    Suppose $\bw_x$ is such that (1) and (2) hold.
    Since (2) holds, $F'$ satisfies \pref{assumption:tree-properties}.
    By \pref{lem:tiering} along with $\Unif(L_y)\ge n^{-\log^3 n}$ and the assumption that (1) holds, we know:
    \[
        \Pr_{\bw_y|\bw_x}\bracks*{\bw_y\in\scap(\bw_x)\cap L_y} \le 1.6p.
    \]
    Since (1) and (2) hold simultaneously except with probability at most $n^{-\log^{2}{n}}$:
    \[
        \Pr_{\bw_x,\bw_y}\bracks*{\angles*{\bw_x,\bw_y}\ge\tau(p)} = \Pr_{\bw_x,\bw_y}\bracks*{\bw_y\in\scap(\bw_x)\cap L_y} \le 2p. \qedhere
    \]
\end{proof}
Our next ingredient is to prove that individual instantiations of vectors from the uniform distribution over $\calF'(\bW_{\Rest})$ produce an intersection of sphere caps whose measure concentrates reasonably well.
\begin{lemma}   \label{lem:mostly-decent}
    With probability at least $1-n^{-\Omega(\log n)}$ over the randomness of $\bW_{\Rest}$ and $\bG_{n-1}$,
    \[
       \Pr_{(\bw_i)_{i\in \Balls}\sim\calF'(\bW_{\Rest})} \bracks*{ \Unif\parens*{\bigcap_{j\in S} \scap\parens*{\bw_j} } \notin \left(1\pm O\parens*{\frac{\log^{15} n}{\sqrt{d}}}\right) p^{|S|} } \le n^{-\log^5 n}.
    \]
\end{lemma}
\begin{proof}
    Each $j\in S$ is in a separate connected component of the factor graph $F(\bW_R)$ and so the collection $(\bw_j)_{j\in S}$ is a collection of independent random vectors.
    Denoting the distribution of $\bw_j$ with $\wt{\nu}_j$, we know by \pref{lem:high-prob-properties} that with probability at least $1-n^{-\Omega(\log n)}$ over the randomness of $\bG_{n-1}$ and $\bW_{\Rest}$ that every connected component of $T(\bG_{n-1})$ satisfies \pref{assumption:tree-properties}.
    \pref{lem:tiering} along with \pref{property:all-sets-lower-bound} from \pref{lem:high-prob-properties} then implies that $\norm*{\wt{\nu}_j}_{\infty} \le n^{\log^5 n}$.
    Since $|S|\le\log^2 n$, the statement follows from an iterated application of \pref{cor:two-sets-subexp-mu} and a union bound over all applications.
\end{proof}

We are now finally prepared to prove \pref{lem:nbrhood-containment}.
\restatelemma{lem:nbrhood-containment}
\begin{proof}[{Proof of \pref{lem:nbrhood-containment}}]
    We first establish that $\eta(\ell)\le\frac{\log^8 n}{\sqrt{d}}$.  Recall that:
    \begin{align*}
       \Pr\bracks*{\randnbr(n) \supseteq S | \bG_{n-1}} &= \E_{\bW\sim\VecDist{\bG_{n-1}}} \Unif\parens*{ \bigcap_{i\in S} \scap(\bw_i) }.
    \end{align*}
    Suppose $\bW\sim\Unif^{\otimes n-1}$, then by an iterated application of \pref{cor:cap-intersect-simple} and a union bound, with probability at least $1-n^{-\Omega(\log^3 n)}$,
    \[
       \Unif\parens*{ \bigcap_{i\in S} \scap(\bw_i) } \in \parens*{1\pm \frac{\log^8{n}}{\sqrt{d}}}. \numberthis \label{eq:conc-cap-intersect}
    \]
    By \pref{obs:unif-to-conditional}, the same probability bound for \pref{eq:conc-cap-intersect} holding is true even when $\bW\sim\VecDist{\bG_{n-1}}$ with probability at least $1-n^{-\Omega(\log^2 n)}$ over the randomness of $\bG_{n-1}$.

    The fact that $\Unif\parens*{ \bigcap_{i\in S} \scap(\bw_i) }$ is always in $[0,1]$ and the upper bound on $d$ together imply that
    \[
       \Pr\bracks*{\randnbr(n) \supseteq S | \bG_{n-1}} \in \left(1 \pm O\parens*{\frac{\log^8}{\sqrt{d}}} \right) p^{|S|}
    \]
    with probability at least $1-n^{-\Omega(\log^3 n)}$.  This establishes the desired statement when $\ell = 0$.

    Henceforth, we can assume $\ell\ge 1$, which in particular means $\Balls$ is nonempty.
    Now we turn our attention to proving that:
    \[
       \Pr\bracks*{ N_{\bG}(n) \supseteq S \mid \bG_{n-1} } \in \parens*{ 1\pm 2\eta(\ell) }\cdot p^{|S|}
    \]
    where $\eta(\ell) \coloneqq \max\braces*{\sqrt{\frac{\log^{28} n}{d}}^{\ell}, 4\sqrt{\frac{\log^{11} n}{nd}}}$.
    Let $\bG_{n-1}$ and $\bW_{\Rest}$ be such that the conclusions of \pref{lem:main-BP-lemma}, \pref{lem:F-within-Fp}, and \pref{lem:mostly-decent} hold.
    By \pref{lem:main-BP-lemma},
    \[
       \E_{\bW_{\Balls}\sim\calF'(\bW_{\Rest})} \Unif\parens*{ \bigcap_{i\in S} \scap(\bw_i) } \in \left(1\pm \eta(\ell) \right) p^{|S|}.
    \]
    Denoting $\Pr_{\bW_{\Balls}\sim\calF'\left(\bW_{\Rest}\right)}\bracks*{ \bW_{\Balls}\in\calF\parens*{\bW_{\Rest}} }$ as $q$, we can also write:
    \begin{align*}
       \E_{\bW_{\Balls}\sim\calF'(\bW_{\Rest})} \Unif\parens*{ \bigcap_{i\in S} \scap(\bw_i) } &=  q\cdot\E_{\bW_{\Balls}\sim\calF'\left(\bW_{\Rest}\right) | \bW_{\Balls} \in \calF\left(\bW_{\Rest}\right)} \Unif\parens*{\bigcap_{i\in S}\scap(\bw_i)} + (1-q)\cdot \E_{\bW_{\Balls}\sim\calF'\left(\bW_{\Rest}\right) | \bW_{\Balls} \notin \calF\left(\bW_{\Rest}\right)} \Unif\parens*{\bigcap_{i\in S}\scap(\bw_i)}.
    \end{align*}
    The first term on the right hand side of the above can be written as:
    \[
       q\cdot \E_{\bW_{\Balls}\sim\calF\left(\bW_{\Rest}\right)} \Unif\parens*{ \bigcap_{i\in S} \scap(\bw_i) }.
    \]

    By \pref{lem:mostly-decent} and \pref{obs:bounded-dists} along with the fact that $\Unif\parens*{\bigcap_{i\in S}\scap(\bw_i)}$ is in $[0,1]$, the second term in the right hand side is in
    \[
        (1-q)\cdot\parens*{ \parens*{ 1 \pm \frac{\log^{15} n}{\sqrt{d}}} \cdot p^{|S|} \pm \frac{n^{-\log^5 n}}{1-q} } \subseteq (1-q)\cdot\parens*{ 1 \pm \frac{\log^{15} n}{\sqrt{d}}}\cdot p^{|S|} \pm n^{-\log^5 n}
    \]
    Rearranging the above in the expression
    \[
        q\cdot\E_{\bW_{\Balls}\sim\calF'\left(\bW_{\Rest}\right) | \bW_{\Balls} \in \calF\left(\bW_{\Rest}\right)} \Unif\parens*{\bigcap_{i\in S}\scap(\bw_i)} + (1-q)\cdot \E_{\bW_{\Balls}\sim\calF'\left(\bW_{\Rest}\right) | \bW_{\Balls} \notin \calF\left(\bW_{\Rest}\right)} \Unif\parens*{\bigcap_{i\in S}\scap(\bw_i)} \in \parens*{1\pm\eta(\ell)} p^{|S|}
    \]
    gives
    \[
       \E_{\bW_{\Balls}\sim\calF\parens*{\bW_{\Rest}}} \Unif\parens*{\bigcap_{i\in S} \scap(\bw_i)} \in \parens*{1 \pm \frac{1-q}{q}\cdot\frac{\log^{15} n}{\sqrt{d}} \pm \frac{\eta(\ell)}{q} }p^{|S|}
    \]
    By \pref{lem:F-within-Fp}, $q\ge1-pn^{O(1/\log\log n)}$, which implies that:
    \[
       \Pr\bracks*{ N_{\bG}(n) \supseteq S \mid \bG_{n-1} } = \E_{\bW_{\Balls}\sim\calF\parens*{\bW_{\Rest}}} \Unif\parens*{\bigcap_{i\in S} \scap(\bw_i)} \in \parens*{1 \pm 2\eta(\ell)}\cdot p^{|S|}
    \]
    The above holds with probability at least $1-n^{-\Omega(\log n)}$ over the randomness of $\bG_{n-1}$ and $\bW_{\Rest}$ since we only assumed the conclusions of \pref{lem:main-BP-lemma}, \pref{lem:F-within-Fp} and \pref{lem:mostly-decent} hold, which completes the proof of \pref{lem:nbrhood-containment}.
\end{proof}

\subsection{From neighborhood containment to neighborhood equality probabilities}   \label{sec:containment-to-exact}
In this section we prove \pref{lem:main-nbrhood-lemma} using \pref{lem:nbrhood-containment}.
\restatelemma{lem:main-nbrhood-lemma}
\begin{proof}[Proof of \pref{lem:main-nbrhood-lemma}]
    Let $\bC = \bigcap_{j \in S}\scap(\bw_j)$, and let $\bA = \bigcap_{j \in [n-1]\setminus S} \ol{\scap}(\bw_j)$.
    We note that the area of $\bC$ is unlikely to be too small:
    \begin{claim}\label{claim:lb-area}
        With probability at least $1 - n^{-\Omega(\log n})$, $\Unif\left(\bC\right) \ge \frac{1}{2}p^{|S|} \ge \Omega(n^{-\log n})$.
    \end{claim}
    We defer the proof of the claim to the end.
    We can thus use \pref{claim:lb-area} to apply \pref{cor:anti-caps} and conclude that with probability at least $1 - n^{-\Omega(\log n)}$,
    \begin{equation}
        \rho\left(\bA \cap \bC\right) \in \left(1 \pm \sqrt{\frac{\log^{11} n}{nd}}\right)\cdot \Unif(\bC)\cdot (1-p)^{n-1-|S|}. \label{eq:contained}
    \end{equation}
    Now from \pref{lem:nbrhood-containment}, with probability $1 - n^{-\Omega(\log n)}$,
    \begin{equation}
        \E_{\VecDist{\bG_{n-1}}} [\Unif \left(\bC\right)]= \Pr[N_{\bG}(n) S \mid \bG_{n-1}] \in (1 \pm \eta(\ell)) p^{|S|}.
        \label{eq:contained2}
    \end{equation}
    Let $\calE$ be the event that \pref{claim:lb-area}, \pref{eq:contained} and \pref{eq:contained2} hold.
    We then have that
    \begin{align*}
        \Pr[N_{\bG}(n) = S\mid \bG_{n-1},\calE]
        = \E[\Unif(\bA \cap \bC)\mid \calE]
        &= \left(1 \pm \sqrt{\frac{\log^{11} n}{nd}}\right)\cdot \E[\Unif(\bC) \mid \calE]\cdot (1-p)^{n-1-|S|}\\
        &= \left(1 \pm \sqrt{\frac{\log^{11} n}{nd}}\right) \cdot (1 \pm \eta(\ell)) p^{|S|} (1-p)^{n-1-|S|},
    \end{align*}
    and since $\sqrt{\frac{\log^{11} n}{nd}} \le \eta(\ell)/2 \ll 1$, we have our conclusion.

    Now we prove \pref{claim:lb-area}.
    We'll apply \pref{lem:martingale-intersect} to conclude that any $m= |S|$ vectors from $\Unif^{\otimes n}$ have cap intersection area $\ge e^{-\poly\log n}$ with overwhelmingly large probability, then apply \pref{obs:unif-to-conditional} to conclude that the same is true for vectors sampled from $\VecDist{\bG_{n-1}}$.

    If $\bv_1,\ldots,\bv_n \sim \Unif^{\otimes n}$, then from \pref{lem:martingale-intersect} combined with a union bound over subsets of $[n]$ of size $|S|$, so long as $d \ge \log^{20} n$,
    \[
        \Pr_{\Unif^{\otimes n}}\left[\Unif\left(\bigcap_{j\in S} \scap(\bv_j)\right) < \frac{1}{2} p^{|S|}\right] 
        \le \binom{n}{|S|} \cdot \Pr_{\Unif^{\otimes |S|}}\left[\Unif\left(\bigcap_{j=1}^{|S|} \scap(\bv_j)\right) < \frac{1}{2} p^{|S|}\right]
        \le \binom{n}{|S|}\cdot n^{-\log^{3} n} \le n^{-.5 \log^3 n},
    \]
    for $n$ sufficiently large.
    Hence by \pref{obs:unif-to-conditional}, 
    \[
        \Pr_{\bw_1,\ldots,\bw_{n-1}\sim \VecDist{\bG_{n-1}}}\left[\Unif\left(\bigcap_{j\in S} \scap(\bw_j)\right) < \frac{1}{2} p^{|S|}\right] \le n^{-.25\log^3 n},
    \]
    with probability at least $1-n^{-.25\log^3 n}$ over the randomness of $\bG_{n-1}$, completing the proof.
\end{proof}

\newcommand{\goodgs}{Z}
\newcommand{\lngbh}{N_{\ell}}
\newcommand{\lfngbh}{N_{\geq\ell}}
\newcommand{\ngbh}{N_{1}}

\section{Total variation bound} \label{sec:tv-bound}

In this section, we prove \pref{thm:all-p} and \pref{thm:sparse}.
As was done in the prior work by Brennan, Bresler, and Nagaraj \cite{BBN20}, we use an analogue of the tensorization of the relative entropy for non-product measures:

\begin{claim}[Relative entropy tensorization, similar to Lemma 2.1 of \cite{BBN20}]\label{claim:KL-seq} 
Suppose $\mu = \mu_1 \otimes \cdots \otimes \mu_n$ is a product measure and $\nu$ is a measure over the same domain.
Let $\nu_t$ denote the marginal of $\nu$ on the $t$-th coordinate $x_t$, and let $x_{a:b}$ denote coordinates $a$ through $b$ of $x$.
Then
\[
\dkl(\nu\, \|\,\mu) = \sum_{t=1}^n \E_{x_{1:t-1} \sim \nu} \left[\dkl\left( \nu_t(x_t \mid x_{1:t-1}) \,\|\, \mu_t \right)\right].
\]
\end{claim}
\begin{proof}
	By the chain rule for relative entropy,
	\[
	\E_{\bx \sim \nu} \log \frac{\nu(\bx)}{\mu(\bx)}
	= \sum_{t=1}^n \E_{\bx \sim \nu} \log \frac{\nu_t(\bx_{t} \mid \bx_{1:t-1})}{\mu_t(\bx_t)}
	= \sum_{t=1}^n \E_{\bx_{1:t-1} \sim \nu}\left( \E_{\bx_t \sim \nu_t \mid \bx_{1:t-1}} \log \frac{\nu_t(\bx_{t} \mid \bx_{1:t-1})}{\mu_t(\bx_t)}\right),
	\]
	by linearity of expectation and by definition of the marginal distribution.
	The expression on the right simplifies using the definition of the relative entropy, completing the proof.
\end{proof}
In combination with Pinsker's inequality, this lemma reduces bounding the TV distance between a product measure $\mu$ and a general measure $\nu$, to bounding the relative entropy $\dkl\left( \nu_t(x_t \mid x_{1:t-1}) \,\|\, \mu_t \right)$.
\begin{claim} \label{claim:n-nbrs}
	Let $\mu_t$ be the distribution of the neighborhood of $t$ to vertices $[t - 1]$ under $\gnp$, and let $\nu_t(\cdot \mid G_{t-1})$ be the distribution of the neighborhood of $t$ under $\grg_d(n, p)$, conditioned on the subgraph $\bG_{t -1} \sim \grg_d(t - 1, p)$ on the vertices $[t - 1]$. Then,
\begin{align*}
2\dtv {\grg_d(n, p)} \gnp ^2 
&\le n \cdot \E_{\bG_{n-1} \sim \grg(n - 1, p)}\left[\dkl\left(\nu_n(\cdot \mid \bG_{n-1}) \, \| \, \mu_t \right) \right]
\end{align*}
\end{claim}

\begin{proof}
Applying Pinsker's inequality (\pref{thm:pinsker}), $2 \cdot \dtv{\nu}{\mu}^2 \le \dkl(\nu\, \|\,\mu)$.
	We then apply \pref{claim:KL-seq} with $\mu = \gnp$ and $\nu = \grg_d(n, p)$, and $\mu_t, \nu_t$ as defined in the claim:
	$$
	2\dtv {\grg_d(n, p)} \gnp ^2 \le \sum_{t = 1}^n \E_{\bG_{t-1} \sim \grg(t - 1, p)}\left[\dkl\left(\nu_t(\cdot \mid \bG_{t-1}) \, \| \, \mu_t \right) \right]
	$$
	Let $G_{S}$ denote a graph over vertices $S$, and let $\nu_t^S$, $\mu_t^S$ refer to the distribution of vertex $t$'s neighbors in $S$ under $\grg_d(n, p)$ and $\gnp$ respectively.
	By symmetry, $\E_{x_{[n] \setminus t} \sim \nu} \left[\dkl\left( \nu_t^{[n] \setminus t}(\cdot \mid x_{[n] \setminus t}) \,\|\, \mu_n \right)\right]$ is the same for all $t \in [n]$. Via the chain rule for relative entropy, and the non-negativity of relative entropy,
	\begin{align*}
	\dkl\left( \nu_t^{[n] \setminus t}(\cdot \mid G_{[n] \setminus t}) \,\|\, \mu_n \right) 
\geq \dkl\left( \nu_t^{[t - 1]}(\cdot \mid G_{[n] \setminus t}) \,\|\, \mu_t^{[t - 1]} \right) 
	= \dkl\left( \nu_t^{[t - 1]}(\cdot \mid G_{t - 1}) \,\|\, \mu_t^{[t - 1]} \right) 
	\end{align*}
	The final equality comes from the fact that $t$'s neighbors in $[t - 1]$ only depends on $\bG_{t - 1}$. 
	
	Upper bounding each $\dkl\left(\nu_t(\cdot \mid \bG_{t-1}) \, \| \, \mu_t \right)$ by $\dkl\left(\nu_n(\cdot \mid \bG_{n-1}) \, \| \, \mu_n \right)$ completes the proof.
\end{proof}

\subsection{Bounding neighborhood relative entropy}

Via \pref{claim:n-nbrs}, our goal now is to upper bound:
$$
\E_{\bG_{n-1} \sim \grg(n - 1, p)}\left[\dkl\left(\nu_n(\cdot \mid \bG_{n-1}) \, \| \, \mu_t \right) \right]
= \E_{\bG_{n-1} \sim \nu_{[n-1]}} \E_{\bS \sim \nu_n (\cdot \mid G_{n - 1})} \ln \frac{\nu_n(\bS \mid G_{n-1})}{\mu_n(\bS)} \leq o\left(\frac{1}{n}\right).$$
For most events under $\nu_{[n - 1]}$ and $\nu_n(\cdot | G_{n - 1})$, we will upper bound the relative entropy via a Chi-square-like quantity: we use the Chi-squared distance, but we allow the omission of a low-probability event $\calE$ (to allow the removal of events which cause the Chi-square distance to blow up).
We then specialize the resulting Chi-square-like bound (\pref{lem:chi-square}) by removing different low-probability events for the general $p$ case (\pref{sec:gen-p}) and the sparse case (\pref{sec:sparse}), and separately conclude \pref{thm:all-p} and \pref{thm:sparse}. 
We now formally state the Chi-square bound:
\begin{lemma} \label{lem:chi-square}
	Let $\calE$ be an event satisfying both 
	$$
		\Pr_{\bG_{n-1\sim}\grg_d(n-1, p),\bS\sim\nu_n(\cdot|\bG_{n-1})}[\calE] \leq o\left(\frac{1}{n^2 \ln n}\right) \text{ and } \Pr_{\bG_{n - 1} \sim \grg_d(n - 1, p), \bS \sim \mu_n}[\calE] \leq o\left(\frac{1}{n^2 \ln n}\right)
	$$
	and for $S\subseteq [n-1]$, define $\Delta_{G_{n-1}}(S) = \frac{\nu_{n}(S  \mid G_{n-1})}{\mu_n(S)} - 1$. Then, 
	$$
		\E_{\bG_{n-1} \sim \nu_{[n-1]}} \E_{\bS \sim \nu_n (\cdot \mid \bG_{n - 1})} \ln \frac{\nu_n(\bS \mid \bG_{n-1})}{\mu_n(\bS)} \leq \E_{\bG_{n-1} \sim \nu_{[n-1]}} \E_{\bS \sim \mu_n} \Delta_{\bG_{n-1}}(\bS)^2 \cdot \Ind(\overline{\calE}) + o\left(\frac{1}{n}\right)
	$$
\end{lemma}
\begin{proof}
Before we introduce the Chi-square bound, we first perform some conditioning on $\overline{\calE}$.
\begin{align*}
\E_{\bG_{n-1} \sim \nu_{[n-1]}} & \E_{\bS \sim \nu_n (\cdot \mid \bG_{n - 1})} \ln \frac{\nu_n(\bS \mid \bG_{n-1})}{\mu_n(\bS)} 
= \E_{\bG_{n-1} \sim \nu_{[n-1]}} \E_{\bS \sim \nu_n(\cdot \mid \bG_{n-1})} \ln\left(\Delta_{G_{n-1}}(\bS) + 1\right) \\ 
&= \E_{\bG_{n-1} \sim \nu_{[n-1]}} \E_{\bS \sim \nu_n(\cdot \mid \bG_{n-1})} \ln\left(\Delta_{\bG_{n-1}}(\bS) + 1\right) \cdot \Ind(\overline{\calE}) + \max_S \left( \ln \frac{\nu_n(S \mid \bG_{n-1})}{\mu_n(S)}\right) \cdot \Pr_{\grg_d(n, p)}[\calE]
\end{align*}
The maximum of $\ln \frac{\nu_n(S \mid G_{n-1})}{\mu_n(S)}$ is upper bounded by $n \ln \frac{1}{p}$, which comes from bounding $\nu_n(S \mid G_{n-1})$, a probability, above by $1$, and achieving the smallest possible $\mu_n(S)$ when $|S| = n - 1$.
Applying our assumption on $\Pr[\calE]$, the term $\max \left( \ln \frac{\nu_n(S \mid G_{n-1})}{\mu_n(S)}\right) \cdot \Pr[\calE]$ is at most $o\left( \frac{1}{n} \right)$.

We now turn our attention to the remaining expectation term:
\begin{align*}
\E_{\bG_{n-1} \sim \nu_{[n-1]}} \E_{\bS \sim \nu_n(\cdot \mid G_{n-1})} \ln\left(\Delta_{G_{n-1}}(\bS) + 1\right) \cdot \Ind(\overline{\calE}) &\le \E_{\bG_{n-1} \sim \nu_{[n-1]}} \E_{\bS \sim \nu_n(\cdot \mid G_{n-1})} \Delta_{G_{n - 1}}(\bS) \cdot \Ind(\overline{\calE}) \\
= \E_{\bG_{n-1} \sim \nu_{[n-1]}} & \E_{\bS \sim \mu_n}(1+\Delta_{G_{n-1}}(\bS))\Delta_{G_{n-1}}(\bS) \cdot \Ind(\overline{\calE})  \\
= \E_{\bG_{n-1} \sim \nu_{[n-1]}} & \E_{\bS \sim \mu_n} \Delta_{G_{n-1}}(\bS) \cdot \Ind(\overline{\calE}) 
+ \E_{\bG_{n-1} \sim \nu_{[n-1]}} \E_{\bS \sim \mu_n} \Delta_{G_{n-1}}(\bS)^2 \cdot \Ind(\overline{\calE})
\end{align*}
The inequality follows because $1+x \le e^x$ for all $x$, and the second equality follows from a change in the randomness of $S$, and because $(1+\Delta_{G_{n-1}}(S)) = \frac{\nu_n(S \mid G_{n-1})}{\mu_n(S)}$.

By the definition of $\Delta_{G_{n-1}}$, we can further simplify the expectation of $\Delta_{G_{n-1}}(S) \cdot \Ind(\overline{\calE})$:
\begin{align*}
\left| \E_{\bG_{n-1} \sim \nu_{[n-1]}} \E_{\bS \sim \mu_n} \Delta_{G_{n-1}}(\bS) \cdot \Ind(\overline{\calE}) \right|
&= \left|\Pr_{\bG_{n - 1} \sim \grg_d(n - 1, p), \bS \sim \mu_n}[\overline{\calE}] - \Pr_{\bG_{n-1\sim}\grg_d(n-1, p),\bS\sim\nu_n(\cdot|\bG_{n-1})}[\calE] \right| \\
&\leq  \Pr_{\bG_{n - 1} \sim \grg_d(n - 1, p), \bS \sim \mu_n}[\calE] + \Pr_{\bG_{n-1\sim}\grg_d(n-1, p),\bS\sim\nu_n(\cdot|\bG_{n-1})}[\calE] \leq o\left(\frac{1}{n^2 \ln n} \right)
\end{align*}
which completes the proof. 
\end{proof}
\begin{remark} \label{rem:reduction}
We will ultimately choose the event $\calE$ based on when the Chi-square estimate is too loose of an upper bound on the relative entropy. 
If the conditions of \pref{lem:chi-square} hold, the overall TV bound we want would follow from
\begin{equation} \label{eq:goal}
\E_{\bG_{n-1} \sim \grg_d(n - 1, p)}\E_{\bS \sim \Binom(n - 1, p)}\left[\left(\frac{\nu_n (\bS \mid \bG_{n-1})}{p^{|\bS|}(1-p)^{n-1-|\bS|}} - 1\right)^2 \cdot \Ind(\overline{\calE}) \right] = o\left(\frac{1}{n}\right).
\end{equation}
\end{remark}

\subsubsection{The general $p$ case} \label{sec:gen-p}

The goal of this section is to prove \pref{thm:all-p}. 
\restatetheorem{thm:all-p}

\begin{proof}
	We apply \pref{lem:chi-square}, when $\calE$ is the failure of degree concentration: $|S| \geq pn + \Delta$, with $\Delta = 10 \max( \log n, pn)$. 
	Via \pref{lem:degree-bound}, $\calE$ has probability $O(n^{-3})$, regardless of whether $S \sim \mu_n$ or $S \sim \nu_n(\cdot | \bG_{n - 1})$ and $\bG_{n - 1} \sim \grg_d(n - 1, p)$.
	(The latter distribution is equivalent to $\bG_n \sim \grg$.)
	We thus satisfy:
	$$
	\Pr_{\bG_{n - 1} \sim \grg_d(n - 1, p), S \sim \mu_n}[\calE], \Pr_{\bG_{n - 1} \sim \grg_d(n - 1, p), S \sim \nu_n(\cdot \mid \bG_{n - 1})}[\calE] \leq o\left(\frac{1}{n^2 \ln n}\right)
	$$
	
	\noindent After the reduction to \pref{eq:goal} in \pref{rem:reduction}, we complete the proof by bounding:
	\begin{align*}
		\E_{\bG_{n-1} \sim \nu_{[n-1]}} & \E_{\bS \sim \Binom(n - 1, p)}\left[\left(\frac{\nu_n (S \mid \bG_{n-1})}{p^{|S|}(1-p)^{n-1-|S|}} - 1\right)^2 \cdot \Ind(\overline{\calE}) \right] \\
		&\leq \int_0^\infty \Pr_{\bG_{n-1} \sim \nu_{[n-1]}, \bS \sim \Binom(n - 1, p)} \left[ \left(\frac{\nu_n (S \mid \bG_{n-1})}{p^{|S|}(1-p)^{n-1-|S|}} - 1\right)^2 > t \mid \overline{\calE} \right] dt
	\end{align*}
	We now apply \pref{cor:caps-and-anti} to control these tail probabilities. 
	By conditioning on $\overline{\calE}$, we may assume $\Delta = 10 \max(pn, \log n)$ in the tail bound, as this choice of $\Delta$ leads to the worst case tail probability for all $|S| \leq pn + \Delta$. 
	Recall that $M(n, p, \Delta) = \max(n H(p), |\Delta| \ln \frac{1}{p}, \ln n)$ in \pref{cor:caps-and-anti}, and note that $p^{|S|}(1-p)^{n-1-|S|} = e^{-n H(p)} \cdot \left(\frac{p}{1 - p}\right)^{\Delta}$.
	\begin{align*}
	\int_0^\infty \Pr\left[ \left(\frac{\nu_n (S \mid \bG_{n-1})}{p^{|S|}(1-p)^{n-1-|S|}} - 1\right)^2 > t \mid \overline{\calE} \right] dt &\leq \int_0^{\frac{1}{n \log n}} 1 \cdot dt + \int_{\frac{1}{n \log n}}^1 \exp\left(\frac{-\Ccaa d \cdot t}{M(n, p, \Delta)^2 \ln n}\right) dt \\
	& \quad \quad \quad \quad +	\int_1^{e^{2n H(p)} \cdot \left(\frac{p}{1 - p}\right)^{\Delta}} \exp\left( \frac{-\Ccaa d}{M(n, p, \Delta)^2 \ln n} + \Ccaa' \log n \right) dt \\
	&\leq \frac{1}{n \log n} - \frac{M(n, p, \Delta)^2 \ln n}{\Ccaa d} \exp \left( \frac{-\Ccaa d \cdot t}{M(n, p, \Delta)^2 \ln n} \right) \Big|_\frac{1}{n \log n}^1 \\
	& \quad \quad \quad \quad + e^{2n H(p)} \cdot \left(\frac{p}{1 - p}\right)^{\Delta} \cdot \exp\left( \frac{-\Ccaa d}{M(n, p, \Delta)^2 \ln n} + \Ccaa' \log n \right) 
	\end{align*}
	In the first line, we split the integral based on the appropriate tail bound expression in \pref{cor:caps-and-anti}, and also remark that it suffices to consider $t \leq e^{2n H(p)} \cdot \left(\frac{p}{1 - p}\right)^{\Delta}$, as $\frac{\nu_n (S \mid \bG_{n-1})}{p^{|S|}(1-p)^{n-1-|S|}}$ is maximized at that value when $|S| \leq np + \Delta$. 
	If we choose $d \geq n \cdot M(n, p, \Delta)^2 \cdot \ln^3 n$, we can bound each summation term individually to obtain the desired bound on $\E_{\bG_{n-1} \sim \nu_{[n-1]}} \E_{\bS \sim \Binom(n - 1, p)}\left[\left(\frac{\nu_n (S \mid \bG_{n-1})}{p^{|S|}(1-p)^{n-1-|S|}} - 1\right)^2 \cdot \Ind(\overline{\calE}) \right]$.
	\begin{align*}
	\int_0^\infty \Pr\left[ \left(\frac{\nu_n (S \mid \bG_{n-1})}{p^{|S|}(1-p)^{n-1-|S|}} - 1\right)^2 > t \mid \overline{\calE} \right] dt &\leq \frac{1}{n \log n} + \frac{1}{n} \cdot e^{-\log n} + e^{2n H(p)} \cdot e^{-Cn \ln^2 n} \leq o\left(\frac{1}{n}\right)
	\end{align*}
	We note that $M(n, p, \Delta) \leq \max(np \ln n, \ln^2 n)$ for $\frac{1}{n} \leq p \leq \frac{1}{2}$, so we can restate our requirement on $d$ as $d = \wt{\Omega}(n^3 p^2)$. 
\end{proof}

\subsubsection{The sparse case} \label{sec:sparse}

For the rest of this section, we assume $p = \frac{\alpha}{n}$, where $\alpha$ is a constant. 
The goal of this section is to prove \pref{thm:sparse}, reproduced below:
\restatetheorem{thm:sparse}

\begin{proof}
As with the general case, we apply \pref{lem:chi-square}, but let $\calE$ be the union of the following events occurring for $\bG\sim\grg_d(n-1,p)$:
	\begin{itemize}
		\item Failure of degree concentration: $|S| \geq 10 \log^2 n$. 
		
		\item Failure of \pref{prop:dom}:
 		If we sample $\bG_- \sim \ER(n - 1, (1 - \eps)p)$, and $\bG_+ \sim \ER(n - 1, (1 + \eps)p)$, corresponding to $\bG_{n - 1}$ via the three-way coupling promised by \pref{prop:dom} with $\eps > \sqrt{\frac{n \cdot H(p)}{d}}$, the graphs do not satisfy $\bG_- \subseteq \bG_{n - 1} \subseteq \bG^+$.  

		\item Failure of \pref{lem:main-nbrhood-lemma}: There exists an $S$ such that $|S|\le\log^2 n$ for which the neighborhood probability
		$$
			\frac{ \nu_n(S \mid \bG_{n - 1})}{p^{|S|}(1-p)^{n-1-|\bS|}} \notin \left(1\pm\eta(\ell_{\bG_{n-1}, S})\right).
		$$
		Here, we can define $\eta(\ell) \coloneqq \min\braces*{\frac{\log^8 n}{\sqrt{d}}, \max\braces*{\left(\frac{\log^{29} n}{\sqrt{d}}\right)^{\ell}, \frac{\log^{12}n}{\sqrt{nd}}}}$, and 
		$$
		\ell_{G, S} :=\min \left\{\frac{\log n}{\log \log n}, \argmax_{\ell} \left\{ \bigcup_{i \in S} B_{G}(i, \ell) \text{ form } |S| \text{ disjoint trees in } G \right\} \right\}
		$$
	\end{itemize} 
Via \pref{lem:degree-bound}, the first event has probability $n^{-\Omega(\log n)}$, regardless of whether we are working over $\bG_{n - 1} \sim \grg_d(n - 1, p), S \sim \mu_n$ or $\bG_{n - 1} \sim \grg_d(n - 1, p), S \sim \nu_n(\cdot \mid \bG_{n - 1})$.

The second and third events also exclusively deal with $\bG_{n - 1}$. 
The probability of the second event is $n^{-\Omega(\log n)}$ by \pref{prop:dom}. 
The third event also has probability at most $n^{-\Omega(\log n)}$ over $\bG_{n - 1} \sim \grg_d(n-1, p)$ by \pref{lem:main-nbrhood-lemma}.
We thus satisfy the conditions of \pref{lem:chi-square}:
$$
	\Pr_{\bG_{n - 1} \sim \grg_d(n - 1, p), \bS \sim \mu_n}[\calE], \Pr_{\bG_{n - 1} \sim \grg_d(n - 1, p), \bS \sim \nu_n(\cdot \mid \bG_{n - 1})}[\calE] \leq o\left(\frac{1}{n^2 \ln n}\right)
$$
It suffices to show that the following quantity is at most $o\left(\frac{1}{n}\right)$:
\[
	\E_{\bG_{n - 1} \sim \grg_d(n - 1, p)}\E_{\bS \sim \Binom(n - 1, p)} \left[\left(\frac{ \nu_n(\bS \mid \bG_{n - 1})}{p^{|\bS|}(1-p)^{n-1-|\bS|}} - 1\right)^2 \cdot \Ind[\overline{\calE}] \right].
\]
Due to the $\Ind[\overline{\calE}]$ term, we can simplify the expression inside the expectation while assuming the outcome of \pref{lem:main-nbrhood-lemma} holds, so this expression is bounded by
$$
\E_{\bG_{n - 1} \sim \grg_d(n - 1, p)}\E_{\bS \sim \Binom(n - 1, p)} \left[\eta(\ell_{\bG_{n - 1}, \bS})^2 \cdot \Ind[\overline{\calE}] \right] 
$$
We can think of $\ell_{\bG_{n - 1}, \bS}$ itself as a random variable over $\{0, \ldots, \lceil \frac{\log n}{\log \log n} \rceil \}$, whose randomness is governed by $\bG_{n - 1}$.
Wishfully, if $\ell_{\bG_{n-1},\bS} \geq \Omega \left(\frac{\log n}{\log \log n}\right)$ always held true, then $\eta(\ell_{\bG_{n-1},\bS})^2 \leq O\left(\frac{\log^{24} n}{nd}\right)$. 
Then, as desired, $\E_{\bG_{n - 1} \sim \grg_d(n - 1, p)}\E_{\bS \sim \Binom(n - 1, p)} \left[\eta(\ell_{\bG_{n - 1}, \bS})^2 \cdot \Ind[\overline{\calE}] \right] \leq o\left(\frac{1}{n}\right)$ when $d \geq \Omega(\log^{30} n)$.

The reality is: with some (small) probability, $\ell_{\bG_{n - 1}, \bS}$ is small, so $\eta(\ell_{\bG_{n - 1}, \bS})^2$ could be as large as $\frac{\log^{16}n}{d}$.
We need to upper bound the probability that $\ell_{\bG_{n - 1}, \bS}$ is small, so that the contribution of this event to $\E_{\bG_{n - 1} \sim \grg_d(n - 1, p)}\E_{\bS \sim \Binom(n - 1, p)} \left[\eta(\ell_{\bG_{n - 1}, \bS})^2 \cdot \Ind[\overline{\calE}] \right]$ remains at most $o\left(\frac{1}{n}\right)$.

But in a sparse \erdos-\renyi graph $\bG_+ \sim \ER(n,(1+\eps)p)$, it is well-known that $\ell_{\bG_+,\bS} \ge \Omega\left(\frac{\log n}{\log\log n}\right)$ with high probability when $\bS \sim \Binom(n,p)$, and further there exist known bounds on the probability that $\ell_{\bG_+, \bS}$ is small.
We can now use the success of the coupling between $\bG_{n-1}$ and $\bG_+$, guaranteed by $\overline{\calE}$, to argue that the same bounds hold over $\bS$ and $\bG_{n-1}$ with high probability.
After we make this relationship between $\ell_{\bG_{n - 1}, \bS}$ and $\ell_{\bG_+, \bS}$ explicit below, the remainder of the proof is accounting for the events when $\ell_{\bG_+, \bS}$ is small. 

Before diving into the distribution of $\ell_{\bG_+, \bS}$, we can first simplify 
\begin{align*}
\E_{\bG_{n - 1} \sim \grg_d(n - 1, p)} & \E_{\bS \sim \Binom(n - 1, p)} \left[\eta(\ell_{\bG_{n - 1}, \bS})^2 \cdot \Ind[\overline{\calE}] \right] \\
&\leq \E_{\bG_{n - 1} \sim \grg_d(n - 1, p)} \E_{\bS \sim \Binom(n - 1, p)}\left[\min \left(\frac{\log^{16} n}{d}, \max\left( \left(\frac{\log^{29}n}{d} \right)^{\ell_{\bG_{n - 1}, \bS}} , \frac{\log^{24}n}{nd} \right) \right) \cdot \Ind[\overline{\calE}] \right] \\ 
&\leq \E_{\bG_+ \sim \ER(n-1, (1 + \eps) p)} \E_{\bS \sim \Binom(n - 1, p)}  \left[\min \left( \frac{\log^{16} n}{d}, \max\left( \left(\frac{\log^{29}n}{d} \right)^{\ell_{\bG_+, \bS}} , \frac{\log^{24} n}{nd} \right) \right) \cdot \Ind[\overline{\calE}] \right] \\
&\leq \E_{\bS \sim \Binom(n - 1, p)} \E_{\bG_+ \sim \ER(n-1, (1 + \eps) p)} \left[\min \left( \frac{\log^{16} n}{d}, \max\left( \left(\frac{\log^{29}n}{d} \right)^{\ell_{\bG_+, \bS}} , \frac{\log^{24} n}{nd} \right) \right) \cdot \Ind[\overline{\calE}] \right] 
\end{align*}
The second-to-last line relies on the success of the coupling from \pref{prop:dom}, with $\eps > \sqrt{\frac{n \cdot H(p)}{d}}$, and noting that when $\bG_{n - 1} \subseteq \bG_+$, we have $\ell_{\bG_+, \bS} \leq \ell_{\bG_{n - 1}, \bS}$. 
The last line involves switching the order of the expectations.

Again, as the bound need only apply when $\overline{\calE}$ occurs, we may assume $|S| \leq \log^2 n$. 
The theorem now follows directly from the claim below.
\begin{claim*}
	For all $|S| \leq \log^2 n$, $\E_{\bG_+ \sim \ER(n-1, (1 + \eps) p)} \left[ \min \left( \frac{\log^{16} n}{d}, \max\left( \left(\frac{\log^{29}n}{d} \right)^{\ell_{\bG_+, S}}, \frac{\log^{24}n}{nd} \right) \right) \right] \leq o\left(\frac{1}{n}\right)$.
\end{claim*} 
\noindent \emph{Proof of Claim.} To understand the magnitude of $d^{\ell_{\bG_{n - 1},S}}$, we need to case on different depths $\ell$. 
\begin{align*}
	\E_{\bG_+ \sim \ER(n-1, (1 + \eps) p)} & \left[\min \left( \frac{\log^{16} n}{d}, \max\left( \left(\frac{\log^{29}n}{d} \right)^{\ell_{\bG_+, S}}, \frac{\log^{24}n}{nd} \right) \right)\right]
	= \Pr_{\bG_+ \sim \ER(n-1, (1 + \eps) p)} \left[\ell_{\bG_+,S} > \log_d(nd)\right] \cdot \frac{\log^{16} n}{nd} \\
	& \quad \quad \quad \quad + \Pr_{\bG_+ \sim \ER(n - 1, (1 + \eps)p)}[\ell_{\bG_+, S} = 0] \cdot \frac{\log^{24} n}{d} + \sum_{\ell = 1}^{\log_d(nd)}  \Pr_{\bG_+ \sim \ER(n-1, (1 + \eps) p)} \left[\ell_{\bG_+,S} = \ell \right] \cdot \left(\frac{\log^{29}n}{d} \right)^{\ell} \\
	&\leq \frac{\log^{16}n}{nd} + \Pr_{\bG_+ \sim \ER(n - 1, (1 + \eps)p)}[\ell_{\bG_+, S} = 0] \cdot \frac{\log^{24} n}{d} + \sum_{\ell = 1}^{\log_d(nd)}  \Pr_{\bG_+ \sim \ER(n-1, (1 + \eps) p)} \left[\ell_{\bG_+,S} \leq \ell \right] \cdot \left(\frac{\log^{29}n}{d} \right)^{\ell} 
\end{align*}
The $\frac{\log^{16}n}{nd}$ term is at most $o\left(\frac{1}{n}\right)$ when $d \geq \Omega(\log^{35} n)$.
To understand the event $\ell_{\bG_{n - 1},S} \leq \ell$, we compute the probability of $\bigcup_{i \in S} B_{\bG_{n - 1}}(i, \ell + 1)$ does not form $|S|$ disjoint trees. 
This event is contained in the union of the following events:
\begin{itemize}
	\item[$\mathcal{I}(\ell)$:] There exist $i, j \in S$ that are distance $\leq 2\ell + 2$ apart, so $B_{\bG_{n - 1}}(i, \ell + 1)$ and $B_{\bG_{n - 1}}(j, \ell + 1)$ \textbf{intersect}. 
	
	\item[$\mathcal{C}(\ell)$:] For some $i \in S$, the ball $B_{\bG_{n - 1}}(i, \ell + 1)$ contains a \textbf{cycle}.
\end{itemize}
Applying the union bound to events $\mathcal{I}$ and $\mathcal{C}$, we can bound:
\begin{align*}
	\Pr_{\bG_+ \sim \ER(n-1, (1 + \eps) p)} \left[\ell_{\bG_+,S} = 0 \right] \cdot \frac{\log^{24} n}{d} &\leq \Pr_{\bG_+ \sim \ER(n-1, (1 + \eps) p)} [\mathcal{I}(0)] \cdot \frac{\log^{24} n}{d} + \Pr_{\bG_+ \sim \ER(n-1,(1 + \eps) p)} [\mathcal{C}(0)] \cdot \frac{\log^{24} n}{d}  \\
	\sum_{\ell = 1}^{\log_d(nd)}  \Pr_{\bG_+ \sim \ER(n-1, (1 + \eps) p)} \left[\ell_{\bG_+,S} = \ell \right] \cdot \left(\frac{\log^{29}n}{d} \right)^{\ell} &\leq \sum_{\ell = 1}^{\log_d(nd)} \Pr_{\bG_+ \sim \ER(n-1, (1 + \eps) p)} [\mathcal{I}(\ell)] \cdot \left(\frac{\log^{29}n}{d} \right)^{\ell}  \\
	&\quad \quad \quad \quad + \sum_{\ell = 1}^{\log_d(nd)}  \Pr_{\bG_+ \sim \ER(n-1,(1 + \eps) p)} [\mathcal{C}(\ell)] \cdot \left(\frac{\log^{29}n}{d} \right)^{\ell} 
\end{align*}

\paragraph{Contribution of event $\mathcal{I}$:}
To analyze event $\mathcal{I}$, we consider its complement $\overline{\mathcal{I}}$. 
We lower bound the probability of $\overline{\mathcal{I}}$ by counting how many ways we can choose the vertices $[i]$ via this greedy process: we first choose $v \in V$, and then set $V \coloneqq V \setminus B_{\bG_{n - 1}}(v, 2\ell + 2)$ before choosing the next vertex. 
Using \pref{lem:ball-bound} and a union bound over all $|S|$ vertices, with probability $\geq 1 - \frac{1}{n^2}$, the size of each $B_{\bG_{n - 1}}(v, 2 \ell + 2)$ is at most $c \log^3 n (pn)^{2 \ell + 2}$, for a universal constant $c$. 
This greedy process ensures that all vertices chosen for $S$ are distance $\geq 2\ell + 2$ apart. 
\begin{align*}
	\Pr_{\bG_+ \sim \ER(n-1, (1 + \eps) p)} [\overline{\mathcal{I}(\ell)}] &\geq \frac{n \cdot (n - c \log^3 n (pn)^{2 \ell + 2}) \cdots (n - (|S| - 1)c \log^3 n (pn)^{2 \ell + 2})}{n^{|S|}} \cdot \left( 1 - \frac{1}{n^2} \right) \\
	&\geq \left[1 - \frac{c \log^3 n |S|(pn)^{2 \ell}}{n}\right]^{|S|} - \frac{1}{n^2} \geq 1 - \frac{c \log^3 n |S|^2 (pn)^{2 \ell + 2}}{n} - \frac{1}{n^2} 
\end{align*}
First, we consider what happens when $\ell_{\bG_+, S} = 0$. (Recall $p = \frac{\alpha}{n}$.) 
\begin{align*}
\Pr_{\bG_+ \sim \ER(n-1, (1 + \eps) p)} [\mathcal{I}(0)] \cdot \frac{\log^{24} n}{d} &\leq \left( \frac{c \log^3 n |S|^2 (pn)^{2}}{n} + \frac{1}{n^2} \right) \cdot \frac{\log^{24} n}{d} \leq \frac{5 c \alpha^2 \log^{29} n}{nd}.
\end{align*}
When $d \geq \Omega(\log^{36} n)$, this quantity is at most $o\left(\frac{1}{n}\right)$. Next, we consider what happens when $\ell_{\bG_+, S} > 0$. 
\begin{align*}
	\sum_{\ell = 1}^{\log_d(nd)} \Pr_{\bG_+ \sim \ER(n - 1, (1 + \eps) p)} [\mathcal{I}(\ell)] \cdot \left(\frac{\log^{29}n}{d} \right)^{\ell} &\leq \sum_{\ell = 1}^{\log_d(nd)} \left( \frac{c\log^3 n |S|^2(pn)^{2\ell + 2}}{n} + \frac{1}{n^2} \right) \cdot \left(\frac{\log^{29}n}{d} \right)^{\ell} \\
	&\leq 2 \left(\sum_{\ell = 1}^{\log_d(nd)}  \frac{c (\log^{29 \ell + 3} n) |S|^2(pn)^{2\ell + 2}}{n \cdot d^{\ell}}\right)
\end{align*}
The summand when $\ell = 1$, $\frac{c (\log^{32} n) |S|^2 \alpha^4}{nd}$ has the largest magnitude, if $\frac{\alpha^2 \log^{29}n}{d} \leq 1$. 
Thus, we can upper bound the summation by $4c \alpha^4 \cdot \frac{\log^{33} n \log_d(nd)}{n\log^{36} n} \leq o \left(\frac{1}{n}\right)$, as desired for the claim. 

\paragraph{Contribution of event $\mathcal{C}$:} Let the event $\mathcal{C}(\ell)_i$  denote the event that $B_{\bG_{n - 1}}(i, \ell + 1)$ contains a cycle for $i \in S$.
Using a union bound, we have $\Pr_{\bG_+ \sim \ER(n -1, (1 + \eps)p)} [\mathcal{C}(\ell)] \leq |S| \cdot \Pr_{\bG_+ \sim \ER(n-1, (1 + \eps)p)} [\mathcal{C}(\ell)_i]$.
Now, using \pref{lem:neighborhood-tree}, there exists a universal constant $c'$ such that
\begin{align*}
	|S| \cdot \Pr_{\bG_{n - 1} \sim \ER(n-1, (1 + \eps)p)} [\mathcal{C}_j \text{ for } \ell] &\leq |S| \cdot \frac{c' (pn)^{\ell + 1}}{n} 
\end{align*}
First, we consider the contribution of the event $\ell_{\bG_+, S} = 0$. (Recall $p = \frac{\alpha}{n}$.)
\begin{align*}
	\Pr_{\bG_+ \sim \ER(n-1, (1 + \eps)p)} [\mathcal{C}(\ell)] \cdot \frac{\log^{24} n}{d} &\leq \frac{c' (\log^{24} n) |S| (pn)}{n d} \leq \frac{2c' \alpha \log^{25} n}{nd} \leq o\left(\frac{1}{n}\right)
\end{align*}
When $d \geq \Omega(\log^{35} n)$, this quantity is at most $o\left(\frac{1}{n}\right)$. 
We lastly consider the events when $\ell_{\bG_+, S} > 0$.
\begin{align*}
	\sum_{\ell = 1}^{\log_d(nd)} \Pr_{\bG_+ \sim \ER(n-1, (1 + \eps)p)} [\mathcal{C}(\ell)] \cdot \left(\frac{\log^{29}n}{d} \right)^{\ell} &\leq \sum_{\ell = 1}^{\log_d(nd)} \frac{c' \log^{29 \ell} |S| (pn)^{\ell + 1}}{n d^{\ell}} 
\end{align*}
Again, the term for $\ell = 1$, $\frac{c' (\log^{29} n) |S| (pn)^2}{d}$, is the largest term of the summation, if $\frac{\alpha \log^{29}n}{d} \leq 1$. 
The summation is thus upper bounded by $2 c' \alpha^2 \cdot \frac{\log^{31} n  \log_d(nd)}{n \cdot \log^{36} n} \leq o \left(\frac{1}{n}\right)$ as well. 
\end{proof}

\section*{Acknowledgments}
Much of this work was conducted when we were participating in the Simons Institute Fall 2020 program on ``Probability, Geometry, and Computation in High Dimensions''.
We thank the Simons Institute for their support, and the program for inspiration.   
S.M. would like to thank Abhishek Shetty for inspiring conversations on geometric concentration. 

\bibliographystyle{alpha}
\bibliography{main}

\appendix
\section{Extending the analysis of the signed triangle statistic to all densities}
\label{app:signed-triangles}

Associate the graph $G$ with the string $G \in \{0,1\}^{\binom{n}{2}}$ with $G_{ij} = \Ind[i \sim j \text{ in }G]$, and define the signed triangle count statistic 
\[
T(G) = \sum_{i < j < k \in [n]} (G_{ij}-p)(G_{jk}-p)(G_{ik}-p).
\]
In \cite{BDER16}, Theorem 2, the authors prove that for any fixed $p \in (0,1)$, $T(\bG)$ distinguishes between $\bG \sim \grg_d(n,p)$ and $\bG\sim \ER(n,p)$ whenever $d \ll n^3 = \Theta( (nH(p))^3)$.
They also show in Theorem 3 that the {\em unsigned} triangle count is a distinguishing statistic for $p = \Theta(\frac{1}{n}))$ whenever $d \ll \log^3 n = \Theta( (nH(p))^3)$.

Here we show that the analysis of $T(\bG)$ from \cite{BDER16} can be extended to show that the signed count $T(\bG)$ distinguishes between $\ER(n,p)$ and $\grg_d(n,p)$ for most $p$, so long as $d \ll (nH(p))^3$.
\begin{lemma}\label{lem:signed-tri}
If $\frac{1}{n^2} \ll p \le 1-\delta$ for any fixed constant $\delta > 0$, then so long as $d \ll (nH(p))^3$,
\[
\left|\E_{\grg_d(n,p)} T(\bG) - \E_{\ER(n,p)} T(\bG)\right| \gg \max \left(\sqrt{\Var_{\grg_d(n,p)} T(\bG)},\sqrt{\Var_{\ER(n,p)} T(\bG)} \right),
\]
where $H(p) = p \log \frac{1}{p} + (1-p) \log \frac{1}{1-p}$ is the binary entropy function.
\end{lemma}

Utilizing the independence of edges in $\ER(n,p)$, it is easy to show that 
\[
\E_{\ER(n,p)} T(\bG) = 0 \qquad \text{ and } \qquad \Var_{\ER(n,p)} T(\bG) = \binom{n}{3}p^3(1-p)^3.
\]
In \cite{BDER16}, in the course of proving their Lemma 1, the authors prove that for any $p < 0.49$ there exists a universal constant $C$ such that 
\[
\E_{\grg_d(n,p)}[T(\bG)] \ge C \binom{n}{3} p^3 \sqrt{\frac{(\log \frac{1}{p})^3}{d}}.
\]
Hence, if $d \ll (nH(p))^3$, one can already see that the expectation of $T(\bG)$ under $\grg_d(n,p)$ overwhelms the standard deviation of $T(\bG)$ under $\ER(n,p)$.
All that remains is to bound the variance under $\grg_d(n,p)$.
For this, \cite{BDER16} do not obtain sharp estimates for all $p$, obtaining simpler bounds for the special cases $p = \Theta(1)$ and $p = \Theta(\frac{1}{n})$.
En route, the authors prove the following claim:
\begin{claim}[From the proof of Lemma 1 in \cite{BDER16}]\label{claim:qbd}
Let $Q = \Pr_{\grg_d(n,p)}[\bG_{12}\bG_{13}\bG_{23} = 1 \mid \bG_{23} = 1]$.
There exists a universal constant $c>0$ such that for all $p < 0.49$, $Q = p^2(1+\eps)$ for $\eps \ge c \sqrt{\frac{\log^3 \frac{1}{p}}{d}}$.
\end{claim}
Here we will use this claim to extend the analysis of \cite{BDER16} to all $p \in (\frac{1}{n^2},0.49)$.

\begin{proof}[Proof of \pref{lem:signed-tri}]
What remains is to bound the variance of $T(\bG)$ under $\bG \sim \grg_d(n,p)$, so hereafter all expectations are taken with respect to $\bG\sim \grg_d(n,p)$.
We may also assume $p < 0.49$, as for constant $p \in (0.49,1)$ we may invoke Theorem 2 of \cite{BDER16}.
Let $T_{ijk} = (G_{ij}-p)(G_{ik}-p)(G_{jk}-p)$ and $\ol{T}_{ijk} = T_{ijk} - \E[\bT_{ijk}]$ for notational simplicity.
We have that
\begin{align}
\Var \left[\sum_{i < j < k \in [n]} \bT_{ijk}\right]
&= 
\E \left[\sum_{i<j<k} \ol \bT_{ijk}^2 + 2\sum_{i < j < k < \ell \in [n]} \ol 
\bT_{ijk} \ol \bT_{jk\ell} + \ol \bT_{ij\ell} \ol \bT_{j\ell k} + \ol \bT_{ik\ell} \ol T_{jk\ell}\right]\nonumber\\
&= \binom{n}{3} \cdot \E\left[\ol \bT_{123}^2\right] + 6 \cdot \binom{n}{4} \cdot \E\left[\ol \bT_{123} \ol \bT_{234}\right],\label{eq:va}
\end{align}
where the first line follows because if triangles $i,j,k$ and $a,b,c$ do not share an edge, $\bT_{ijk}$ and $\bT_{abc}$ are independent, and the second line follows from symmetry.

Because $G_{ij}^2 = G_{ij}$, we have that $(G_{ij} - p)^2 = (1-2p) G_{ij} + p^2$, and so
\begin{align}
T_{123}^2 
&= ((1-2p)G_{12} + p^2)\cdot ((1-2p)G_{13} +p^2)\cdot ((1-2p)G_{23} +p^2),\label{eq:qts}\\
T_{123}T_{234} 
&= ((1-2p)G_{23}+p^2)\cdot (G_{12} - p)\cdot (G_{13} - p)\cdot (G_{24} - p)\cdot (G_{34} - p).\label{eq:kitec}
\end{align}
Now we will bound each of $\E[\ol \bT_{123}^2]$ and $\E[\ol \bT_{123} \ol \bT_{234}]$ separately.
First, we bound $\E[\ol \bT_{123}^2]$ by expanding \pref{eq:qts} and utilizing the symmetry of $\bG_{12},\bG_{23},\bG_{13}$,
\begin{align*}
\E[\ol \bT_{123}^2 ]
&= \E[\bT_{123}^2] - \E[\bT_{123}]^2\\
&= (1-2p)^3\E[\bG_{12}\bG_{13}\bG_{23}] + 3(1-2p)^2p^2 \E[\bG_{12}\bG_{13}] + 3 (1-2p)p^4 \E[\bG_{12}] + p^6\\
&\qquad - \left(\E[\bG_{12}\bG_{13}\bG_{23}] - 3 p \E[\bG_{12}\bG_{13}] + 3p^2 \E[\bG_{12}] - p^3\right)^2\\
&= (1-2p)^3 \E[\bG_{12}\bG_{13}\bG_{23}] + 3(1-2p)^2p^4 + 3(1-2p)p^5 + p^6\\
&\qquad - \left(\E[\bG_{12}\bG_{13}\bG_{23}] - p^3\right)^2, 
\end{align*}
where we have used that for any $S \subseteq \binom{[n]}{2}$ which corresponds to a tree, $\E \prod_{(i,j) \in S} \bG_{ij} = p^{|S|}$.
Recalling that we have defined $Q = p^2(1+\eps)= \E[\bG_{12}\bG_{13}\bG_{23} \mid \bG_{23} = 1]$,
\begin{align}
\E[\ol \bT_{123}^2]
&=  (1-2p)^3 \cdot p\cdot Q + 3(1-2p)^2 p^4 + 3(1-2p)p^5 + p^6 - (pQ - p^3)^2\nonumber \\
&= p^3(1-p)^3 + (1-2p)^3 p^3 \eps - p^6 \eps^2\nonumber\\
&= O( (1+\eps) p^3).\label{eq:tbd}
\end{align}

Dealing now with the second term, starting with \pref{eq:kitec},
\begin{align*}
\E[\ol \bT_{123}\ol \bT_{234}]
&= \E[\bT_{123}\bT_{234}] - \E[\bT_{123}]^2\\
&= \E[\bT_{123}\bT_{234}] - \eps^2p^6\\
&= \E[((1-2p)\bG_{23} +p^2)(\bG_{12}-p)(\bG_{13}-p)(\bG_{24}-p)(\bG_{34}-p)] - \eps^2p^6\\
&= (1-2p)\cdot p \cdot \E[(\bG_{12}-p)(\bG_{13}-p) \mid \bG_{23} = 1]^2 + p^2 \E[(\bG_{12}-p)(\bG_{13}-p)(\bG_{24}-p)(\bG_{34}-p)] - \eps^2p^6
\intertext{Where we have used that $(\bG_{12}-p)(\bG_{13}-p)$ and $(\bG_{24}-p)(\bG_{34}-p)$ are independent conditioned on $\bG_{23}$.
Now again applying the simplification that tree-shaped products are independent,}
&= (1-2p)p(Q-p^2)^2 + p^2(\E[\bG_{12}\bG_{13}\bG_{24}\bG_{34}] - p^4) - \eps^2p^6 \\
&= (1-2p)p^5\eps ^2 + p^2(\E[\bG_{12}\bG_{13}\bG_{24}\bG_{34}] - p^4) - \eps^2p^6.
\intertext{Now, we use that $\E[\bG_{12}\bG_{13}\bG_{24}\bG_{34}] \le \E[\bG_{12}\bG_{13}\bG_{24}\bG_{34} \mid \bG_{23} = 1]$, and again applying independence conditioned on $\bG_{23}$,}
&\le (1-2p)p^5\eps ^2 + p^2(\E[\bG_{12}\bG_{13}\mid \bG_{23}]^2 - p^4) - \eps^2p^6\\
&= (1-2p)p^5\eps^2 + p^2(p^4(1+\eps)^2 - p^4) - \eps^2p^6\\
&= (1-2p)p^5 \eps^2 + 2 p^6 \eps.
\end{align*}
Hence we have 
\begin{equation}
\E[\ol \bT_{123}\ol \bT_{234}] = O( p^5 \eps^2 + 2p^6 \eps). \label{eq:kite}
\end{equation}

Putting together \pref{eq:va}, \pref{eq:tbd}, and \pref{eq:kite}, we have that there exists a constant $C$ such that
\[
\frac{\sqrt{\Var[T(\bG)]}}{\E[T(\bG)]} \le C \cdot \frac{\sqrt{(1+\eps)p^3 n^3 + p^5 \eps^2 n^4 + p^6 \eps n^4}}{n^3 p^3 \eps} \le C \cdot \left(\sqrt{\frac{1+\eps}{p^3 n^3 \eps^2}} + \frac{1}{n\sqrt{p}} + \frac{1}{n\sqrt{\eps}}\right).
\]
Since we have assumed that $p \gg \frac{1}{n^2}$, $\frac{1}{n\sqrt{p}} \to_n 0$.
Also, by \pref{claim:qbd}, $\frac{1}{n \sqrt{\eps}} \le \frac{d^{1/4}}{nc^{1/2}\log^{3/4} \frac{1}{p}} \to_n 0$ so long as $d \ll n^4 \log^3 \frac{1}{p}$, which is implied by our assumption that $d \ll (nH(p))^3$.
Finally, if $\eps > 1$, then applying \pref{claim:qbd},
\[
\frac{1+\eps}{p^3 n^3 \eps^2} \le \frac{2\eps}{p^3 n^3 \eps^2} \le \frac{2}{p^3 n^3 \eps} \le \frac{2d^{1/2}}{c p^3 n^3\log^{3/2} \frac{1}{p}} \to_n 0,
\]
whenever $d \ll (nH(p))^3$.
Alternatively, if $\eps < 1$, applying \pref{claim:qbd} again,
\[
\frac{1+\eps}{p^3 n^3 \eps^2} \le \frac{2}{p^3 n^3 \eps^2} \le \frac{2d}{c^2 p^3 n^3\log^{3} \frac{1}{p}} \to_n 0,
\]
since $d \ll (nH(p))^3$.
This concludes the proof.
\end{proof}

\section{Deferred proofs from the preliminaries}\label{app:prelim}
\label{app:cap-proofs}

\noindent We first provide a proof of \pref{lem:upper-bound-tau}, which provides a convenient upper bound on the dot product threshold $\tau(p)$ corresponding to a $p$-cap.
\restatelemma{lem:upper-bound-tau} 
\noindent We apply the following constant factor approximation for the measure of a sphere cap in terms of the threshold $\tau(p)$, when $\tau(p) \geq \sqrt{\frac{2}{d}}$ given by \cite[Lemma 2.1(b)]{BGKKLS01}:
\begin{theorem} \label{thm:brieden}
	Consider a $p$-cap where $\tau(p) \geq \sqrt{\frac{2}{d}}$. Then:
	$$
	\frac{1}{6 \tau(p) \sqrt{d}} \left(1 - \tau(p)^2 \right)^{(d - 1)/2} \leq p \le \frac{1}{2 \tau(p) \sqrt{d}} \left(1 - \tau(p)^2 \right)^{(d - 1)/2} 
	$$
\end{theorem}

\begin{proof}[Proof of \pref{lem:upper-bound-tau}]
	We case on whether $\tau(p)$ is smaller or larger than $\sqrt{\frac{2}{d}}$.  In the first case observe that $\log(1/p)\ge 1$ by our bound on $p$ and so $\tau(p)\le\sqrt{\frac{2}{d}}\le\sqrt{\frac{2\log(1/p)}{d}}$.
	
	When $\tau(p) \geq \sqrt{\frac{2}{d}}$, we use the upper bound in \pref{thm:brieden}. Let $\tau' = \sqrt{\frac{2 \log(1 / p)}{d}}$, and let $p'$ denote its corresponding tail probability $\Pr_{x, y \in \bbS^{d - 1}}[\langle x, y \rangle \geq \tau']$. If we can show $p' \leq p$, then we know $\tau(p) \leq \tau'$. 
	\begin{align*}
		p' \leq \frac{1}{2 \tau' \sqrt{d}} &\cdot \left(1 - t^2 \right)^{(d - 1)/2} = \frac{1}{2 \sqrt{\log(1 / p)}} \cdot \left(1 - \frac{2 \log(1 / p)}{d} \right)^{(d - 1)/2} \\
		&\leq \frac{1}{2 \sqrt{\log(1 / p)}} \cdot p \leq p
	\end{align*}
	Since $p' \leq p$, this tells us $\tau \leq \tau'$ as well, giving us the desired inequality.
\end{proof}

\noindent We next prove \pref{lem:beta-concentration}. 
Understanding the joint distribution of two independently chosen unit vectors in $\bbS^{d - 1}$ helps complete this proof. 
We first present a theorem from \cite{BBN20} that present some useful bounds on this distribution $\psi$.  

\begin{lemma}[{\cite[Lemma 5.1]{BBN20}}] \label{lem:beta-pdf}
	Let $\tau(p)$ be the value of $t$ at which $\Pr_{x, y \in \bbS^{d - 1}}[\langle x, y \rangle \geq t] = p$, and let $\psi_d$ be the distribution of the inner products of two random unit vectors in $\mathbb{R}^2$. 
	\begin{enumerate}
		\item \label{part:psi-ratio-bound} For $0 \leq \tau \leq \frac{1}{2}$ and $\delta > 0$, we have:
		\[
		\frac{\psi_d(\tau- \delta)}{\psi_d(\tau)} \leq \exp(2\tau d \delta).
		\]
		\item \label{part:psi-bound-tau} $\psi_d(\tau) \leq \CPsiBound p \cdot \max\{\sqrt{d}, d\tau\}$ for a universal constant $\CPsiBound$.
		\item \label{part:psi-formula} $
		\psi(\tau(p)) = \dfrac{\Gamma\left( \frac{d}{2} \right)}{\Gamma \left( \dfrac{d - 1}{2} \right) \sqrt{\pi}} \cdot \left( 1 - \tau(p)^2 \right)^{(d - 3)/2}
		$.
	\end{enumerate}
\end{lemma}

\begin{proof}[Proof of \pref{lem:beta-concentration}]
	When $t \geq 0$, the density $\psi_d(t)$ is a decreasing function in $t$. Thus:
	\begin{align*}
		\Pr_{z \sim \rho}[\tau - \varepsilon \leq \langle z, x \rangle \leq \tau + \varepsilon]
		&= \int_{\tau - \varepsilon}^{\tau + \varepsilon} \psi_d(t) dt \\
		&\leq (2 \varepsilon) \cdot [\psi_d(\tau) \cdot \exp(2d \tau \varepsilon)]
	\end{align*}
	In the last line, we used \pref{part:psi-ratio-bound} of \pref{lem:beta-pdf}, and noted that this is an upper bound even when $\varepsilon > \tau$. Using \pref{part:psi-bound-tau} of \pref{lem:beta-pdf}, and \pref{lem:upper-bound-tau}, we can upper bound $\psi_d(\tau)$:
	\begin{align*}
		\psi_d(\tau) &\leq \CPsiBound p \cdot \max\{\sqrt{d}, d \tau\} \leq p \cdot \left(\CPsiBound \sqrt{2} \cdot \sqrt{d} \cdot \sqrt{\log(1 / p)} \right)
	\end{align*}
	from which the desired bound follows.
\end{proof}

\end{document}